%% file: Article1.tex
\documentclass[a4paper,11pt]{article}
\pdfoutput=1
\usepackage[english]{bricearticle}

\title{Computing discrete equivariant harmonic maps}
\author{
Jonah Gaster\footnote{McGill University, Department of Mathematics and Statistics. Montreal, QC H3A 0B9, Canada.\newline
E-mail: \url{jbgaster@gmail.com}}, ~
Brice Loustau\footnote{Rutgers University - Newark, Department of Mathematics. Newark, NJ 07105 USA; and
TU Darmstadt, Department of Mathematics. 64289 Darmstadt, Germany. 
E-mail: \url{loustau@mathematik.tu-darmstadt.de}}, ~and
Léonard Monsaingeon\footnote{IECL Université de Lorraine, Site de Nancy. F-54506 Vand{\oe}uvre-lès-Nancy Cedex, France; and 
GFM Universidade de Lisboa. 1749-016 Lisboa, Portugal. 
E-mail: \url{leonard.monsaingeon@univ-lorraine.fr}}}

\date{} 

\hypersetup{
    pdftitle={Computing discrete equivariant harmonic maps},
    pdfauthor={Jonah Gaster, Brice Loustau, and Léonard Monsaingeon},
}

\newcommand{\Harmony}{\texttt{Harmony}}

\begin{document}

\pdfbookmark[1]{Title page, abstract}{Title}
\hypersetup{pageanchor=false}
\begin{titlepage}
\newgeometry{top=0.09\paperheight, bottom=0.11\paperheight}
\maketitle
\thispagestyle{empty}
\centering
\includegraphics[width=90mm]{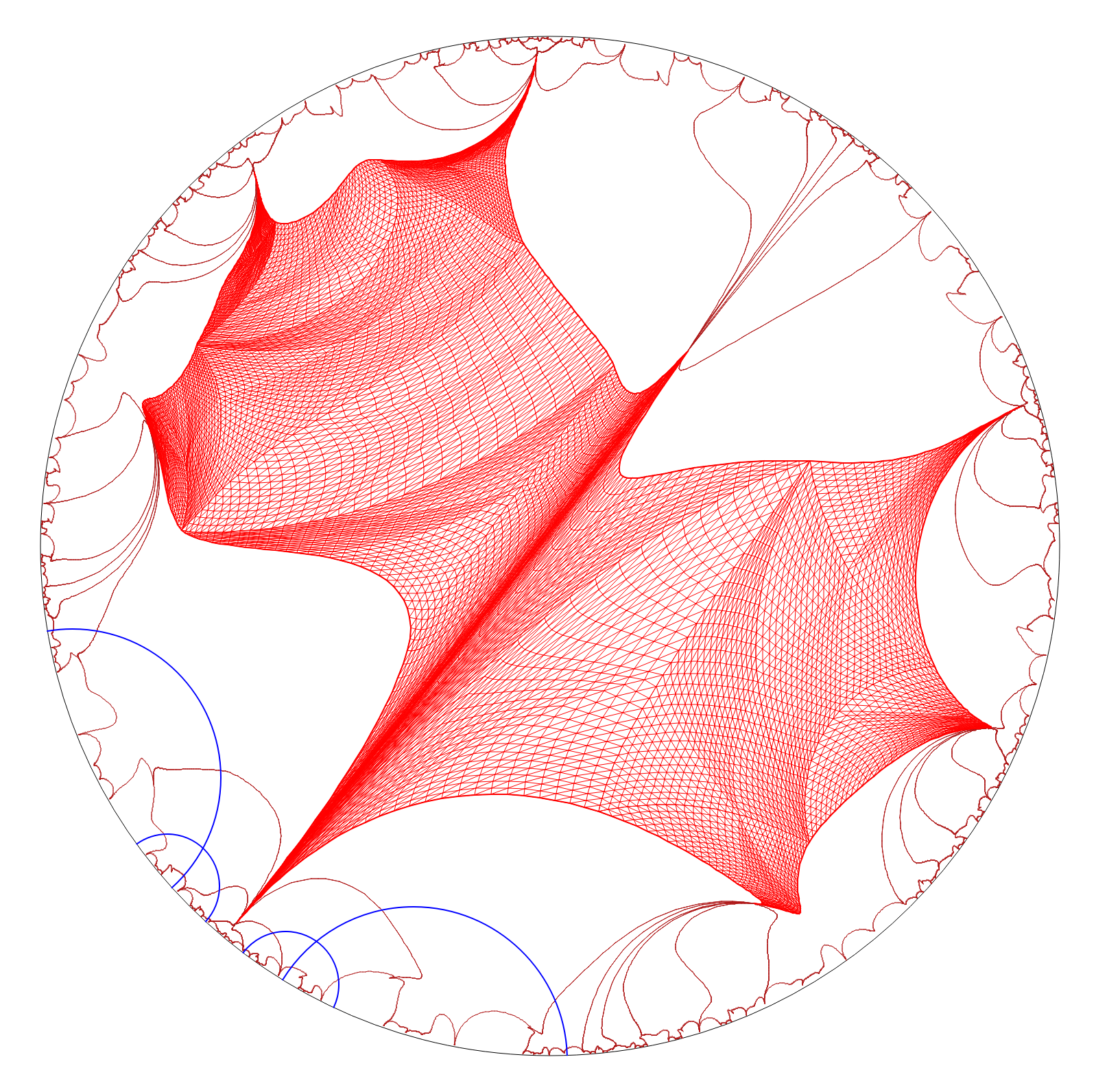}
\bigskip
\input{Abstract}
\end{titlepage}
\hypersetup{pageanchor=true}

\thispagestyle{empty}
\setcounter{tocdepth}{2}
\pdfbookmark[1]{Contents}{Contents}  
{\tableofcontents}
\addtocounter{page}{1}
\restoregeometry

\input{Introduction}
\input{Acknowledgments}
\input{HarmonicMaps}
\input{Discretization}

\input{StrongConvexity}

\input{DiscreteHeatFlow}

\input{CenterOfMass}
\input{ComputerImplementation}

\cleardoublepage\phantomsection 
\bibliographystyle{alpha}
{\small \bibliography{biblio}}

\end{document}

%% file: Abstract.tex
\begin{abstract}

We present effective methods to compute equivariant harmonic maps from the universal cover of a surface into a nonpositively curved space. By discretizing the theory appropriately, we show that the energy functional is strongly convex and derive convergence of the discrete heat flow to the energy  minimizer, with explicit convergence rate. We also examine center of mass methods, after showing a generalized mean value property for harmonic maps. We feature a concrete illustration of these methods with \Harmony{}, a computer software that we developed in \Cpp{}, whose main functionality is to numerically compute and display equivariant harmonic maps.

\bigskip \bigskip

\noindent \textbf{Key words and phrases:}
Harmonic maps $\cdot$ Heat flow $\cdot$ Convexity $\cdot$ Gradient descent $\cdot$ Centers of mass $\cdot$ Discrete geometry  $\cdot$ Riemannian optimization $\cdot$ Mathematical software

\bigskip

\noindent \textbf{2000 Mathematics Subject Classification:} 
Primary: 
58E20; 
Secondary: 
53C43 $\cdot$ 
65D18 

\end{abstract}

%% file: Introduction.tex
\cleardoublepage\phantomsection 
\section*{Introduction}
\addcontentsline{toc}{section}{Introduction}

The theory of harmonic maps has its roots in the foundations of Riemannian geometry and the essential work of Euler, Gauss, Lagrange, and Jacobi. 
It includes the study of real-valued harmonic functions, geodesics, minimal surfaces, and holomorphic maps between Kähler manifolds.

A harmonic map $f \colon M \to N$
is a critical point of the energy functional
\begin{equation}
 E(f) = \frac{1}{2} \int_M \Vert \upd f \Vert^2 \, \upd v_M~.
\end{equation}

The theory of harmonic maps was brought into a modern context for Riemannian manifolds with the seminal work of Eells-Sampson \cite{MR0164306} (also Hartman \cite{MR0214004} and Al'ber \cite{MR0230254}). Eells-Sampson studied the \emph{heat flow} associated to the energy, \ie{} the nonlinear parabolic PDE
\begin{equation}
 \ddt f_t = \tau(f_t)~,
\end{equation}
where $\tau(f)$ is the \emph{tension field} of $f$ (see \autoref{sec:HarmonicMaps}). The tension field can be described as minus 
the gradient of the energy functional on the infinite-dimensional Riemannian space $\cC^\infty(M,N)$, so that the heat flow is just 
the gradient flow for the energy. When $N$ is compact and nonpositively curved, the heat flow is shown to converge to an energy-minimizing map as $t\to \infty$.

The theory has since been developed and generalized to various settings where the domain or the target are not smooth manifolds
\cite{MR1215595,MR1266480,MR1340295,MR1451625,MR1848068, MR1938491, MR2394023}. In particular, Korevaar-Schoen developed an extensive Sobolev theory 
when the domain is Riemannian but the target is a nonpositively curved metric space \cite{MR1266480,MR1483983}.
Jost generalized further to a domain that is merely a measure space \cite{MR756629, MR1385525,MR1360608,MR1449406}. 
Both took a similar approach, 
constructing the energy functional $E$ as a limit
of approximate energy functionals. 

These tools have become powerful and widely used, with celebrated rigidity results \cite{MR584075,MR1215595,MR1775138,MR2827014,MR2174098} and 
dramatic implications for the study of deformation spaces and Teichmüller theory when the domain $M$ is a surface (see \eg \cite{MR2349668}),
especially via the nonabelian Hodge correspondence \cite{MR887285,MR887284,MR965220,MR982185,MR1049845,MR1159261}.

This paper and its sequel \cite{Gaster-Loustau-Monsaingeon2} are concerned with effectiveness
of methods for finding harmonic maps. 
In addition to the mathematical content, we feature \Harmony{}, a computer program that we developed
in \Cpp{} whose main functionality is to numerically compute equivariant harmonic maps.  
Our project was motivated by the question: is it possible to study the nonabelian Hodge correspondence experimentally? 
Though the heat flow is constructive to some extent, it does not provide qualitative information about convergence in general. 
We show that one can design entirely effective methods to compute harmonic maps by discretizing appropriately.

Some of the existing literature treats similar questions, though seldom in the Riemannian setting. Notably, Bartels applies finite element methods on submanifolds of $\bR^n$ to nonlinear PDEs such as the Euler-Lagrange equations for minimizing the energy \cite{MR2177142,MR2629993,MR3309171}. There is also an extensive literature on 
discrete energy functionals of the form we consider in this paper \cite{MR1151746,MR1246481,MR1721305,MR1783793,MR2346504,MR1775138,MR2174098,MR1848068,MR2471368,MR3423736}.

However, several features set our work apart. Firstly, we study sequences of arbitrarily fine discretizations 
as opposed to one fixed approximation of the manifold, 
particularly in the sequel paper \cite{Gaster-Loustau-Monsaingeon2}.
Moreover, we follow the maxim of Bobenko-Suris \cite[p.~xiv]{MR2467378}:
\begin{center}
\emph{Discretize the whole theory, not just the equations}.
\end{center}
Our discrete structures record two independent systems of weights on a given triangulation of the domain manifold $M$, on the set of vertices and on the set of edges respectively.
The vertex weights are a discrete record of the volume form, while the 
edge weights generalize the well-known ``cotangent weights'' popularized by Pinkall-Polthier in the Euclidean setting \cite{MR1246481}. In the $2$-dimensional case,
the edge weights can be thought of as a discrete record of the conformal structure.
This discretization endows the space of discrete maps with a finite-dimensional Riemannian structure, approximating the $\upL^2$ metric on $\cC^\infty(M,N)$. 
We obtain the right setting for a study of the convexity of the discrete energy, 
and for the definitions of the discrete energy density, discrete tension field, and discrete heat flow. 
Among the practical benefits, we find that the discrete energy satisfies stronger convexity properties than those known to hold in the smooth setting.

The theory of harmonic maps is also intimately related to centers of mass.
In the Euclidean setting, harmonic functions satisfy the well-known mean value property.
While the latter no longer holds \emph{stricto sensu} in the more general Riemannian setting, 
we prove a generalization in terms of centers of mass of harmonic maps
on very small balls. 
More generally, averaging a function 
on small balls offers a viable method in order to decrease its energy, 
a viewpoint well adapted to Jost's theory of generalized harmonic maps \cite{MR1385525, MR1360608, MR1449406, MR1451625}. As an alternative to the discrete heat flow, we pursue a discretization of the theory of Jost by analyzing discrete center of mass methods. 

For both heat flow and center of mass methods, the present paper focuses on a fixed discretization of the domain, while the sequel paper \cite{Gaster-Loustau-Monsaingeon2} analyzes convergence of the discrete theory back to the smooth one as we take finer and finer discretizations approximating a smooth domain. 
While this paper focuses on two dimensional hyperbolic manifolds, and chooses the equivariant setting, the second paper embraces the general Riemannian setting for the most part, and drops the equivariant formulation.

\bigskip
Now let us describe more precisely some of the main theorems of the paper. After discussing harmonic maps in \autoref{sec:HarmonicMaps} and developing a discretized theory in \autoref{sec:Discretization}, we study the convexity of the energy functionals in \autoref{sec:StrongConvexity}. We show:
\begin{theorem*}[\autoref{thm:StrongConvexityEnergyGraphGeneral}]
Let $\cS$ be a discretized surface of negative Euler characteristic and let $N$ be a compact surface of nonzero Euler characteristic with nonpositive sectional curvature. 
The discrete energy functional is strongly convex 
in any homotopy class of nonzero degree.
\end{theorem*}
\noindent 
We stress that while the convexity of the energy comes for free, the content of the theorem above is a positive modulus of convexity, \ie{} a uniform lower bound for the Hessian.
We actually show a more general version of this theorem involving equivariant maps and the notion of \emph{biweighted triangulated graph} which we introduce in \autoref{sec:Discretization}. See \autoref{thm:StrongConvexityEnergyGraphGeneral} for the precise statement.

When the target $N$ is specialized to a hyperbolic surface, we find explicit bounds for the Hessian of the energy functional (see \autoref{thm:StrongConvexityEnergyGraphH2}). 
We achieve this through detailed calculations in the hyperbolic plane, which we then generalize to negatively curved target surfaces using $\CAT(k)$-type comparisons.
Roughly speaking, the key idea is that if the energy of some function has a very small second variation, then one can construct an almost parallel
vector field on the image of the surface; however this is not possible for topological reasons.
\autoref{sec:StrongConvexity} is concerned with the significant work of making this argument precise and quantitative.

In \autoref{sec:DiscreteHeatFlow} we study gradient descent methods in Riemannian manifolds and specialize to the convergence of the discrete heat flow. Fully leveraging strong convexity of the energy, we show:
\begin{theorem*} [\autoref{thm:ConvergenceDiscreteHeatFlowGeneral}]
Let $\cS$ be a discretized surface of negative Euler characteristic and let $N$ be a compact surface of nonzero Euler characteristic with nonpositive curvature. There exists a unique discrete harmonic map $f^* \colon \cS \to N$
in any homotopy class of nonzero degree. Moreover, for any initial value $f_0:\cS\to N$ and for any sufficiently small $t>0$, the discrete heat flow with fixed stepsize $t$ converges to $f^*$ with 
exponential convergence rate.
\end{theorem*}
\noindent This is again a simplified version of the theorem we show; see \autoref{thm:ConvergenceDiscreteHeatFlowGeneral} for the precise statement.

Next we discuss center of mass methods in \autoref{sec:CenterOfMass}. 
An interesting alternative to the heat flow consists in averaging $f$ on balls of small radius $r>0$, producing a new map $B_r f \colon M \to N$. Repeating this process potentially produces energy-minimizing sequences for an approximate version of the energy $E_r$, a phenomenon that has been explored by Jost \cite{MR1385525}. The central theorem we prove in \autoref{subsec:GeneralizedMeanValueProperty}
is that in the Riemannian setting, this iterative process is almost the same as a fixed stepsize time-discretization of the heat flow. 
See \autoref{thm:AverageMap} for a precise statement. 

We prove in particular the following generalized mean value property for harmonic maps:
\begin{theorem*}[\autoref{thm:GeneralizedMeanValueProperty}]
Let $f \colon M \to N$ be a smooth map between Riemannian manifolds. Then $f$ is harmonic if and only if for all $x\in M$, we have as $r\to 0$:
\begin{equation}
d(f(x), B_r f(x)) = \bigO(r^4)\,.
\end{equation} 
\end{theorem*}
Under suitable conditions, we show that the center of mass method converges to a minimizer of the approximate energy $E_r$ (see \autoref{thm:CenterOfMassMethodSmooth}), recovering a theorem of Jost \cite[\S3]{MR1385525}. Jost's result is more general, but our conclusion is slightly stronger.

In the space-discretized setting, the approximate energy coincides with the discrete energy, making the discrete center of mass method
an appropriate alternative to the discrete heat flow.

\begin{theorem*}[\autoref{thm:CenterOfMassMethodDiscrete}]
Let $\cS$ be a discretized surfaceand $N$ be a compact manifold of negative sectional curvature. In any  $\pi_1$-injective homotopy class of maps $\cS\to N$, the center of mass method from any initial map converges to the unique discrete harmonic map.
\end{theorem*}

\bigskip
As a concrete demonstration of the effectiveness of our algorithms, we present in \autoref{sec:Harmony} our own computer implementation of a harmonic map solver: \Harmony{} is a freely available computer software with a graphical user interface written in \Cpp{} code, using the Qt framework. This program takes as input the Fenchel-Nielsen coordinates for a pair of Fuchsian representations $\rho_\tL, \rho_\tR \colon \pi_1 S \to \Isom(\bH^2)$ and computes and visualizes 
the unique equivariant harmonic map. \Harmony{}'s main user interface is illustrated in \autoref{fig:HarmonyUserInterface}.

In future development of \Harmony{}, we plan to compute and visualize harmonic maps for more general target representations that are not necessarily discrete, and for more general target spaces, such as $\bH^3$ and other nonpositively curved symmetric spaces. 

We have implemented both the discrete heat flow method, with fixed and optimal stepsizes separately, and the $\cosh$-center of mass method, a clever variant of the center of mass method suggested to us by Nicolas Tholozan that is better suited for computations in hyperbolic space (discussed in \autoref{subsec:CoshCenterOfMass}). In practice, the $\cosh$-center of mass method is the most effective, both in number of iterations and execution time (see \autoref{subsec:ExperimentalComparison}). 

%% file: Acknowledgments.tex
\phantomsection 
\subsection*{Acknowledgments}
\addcontentsline{toc}{section}{Acknowledgments}

The authors wish to thank
Tarik Aougab, 
Benjamin Beeker,
Alexander Bobenko,
David Dumas,
John Loftin,
Jean-Marc Schlenker, 
Nicolas Tholozan, and
Richard Wentworth
for valuable conversations and correspondences pertaining to this work. 
We especially thank David Dumas for 
his extensive advice and support with the mathematical content and the development of \Harmony{}, 
and Nicolas Tholozan for sharing key ideas that contributed to 
\autoref{thm:StrongConvexityEnergyGraphGeneral} and \autoref{thm:GeneralizedMeanValueProperty}.

The first two authors gratefully acknowledge research support from the NSF Grant DMS1107367 \emph{RNMS: GEometric structures And Representation varieties} (the \emph{GEAR Network}).
The third author was partially supported by the Portuguese Science Foundation FCT trough grant PTDC/MAT-STA/0975/2014 \emph{From Stochastic Geometric Mechanics to Mass Transportation Problems}.

%% file: HarmonicMaps.tex
\section{Harmonic maps} \label{sec:HarmonicMaps}


\subsection{Energy functional and harmonic maps} \label{subsec:EnergyFunctionalHarmonicMaps}

Let $(M,g)$ and $(N,h)$ be two smooth Riemannian manifolds. Assuming $M$ is compact, the \emph{energy} of a smooth map $f \colon M \to N$ is:
\begin{equation} \label{eq:EnergyFunctionalSmooth}
 E(f) = \frac{1}{2} \int_M \Vert \upd f \Vert^2 \, \upd v_g
\end{equation}
where $v_g$ is the volume density of the metric $g$. Note that $\upd f$ is a smooth section of the bundle $\upT^*M \otimes f^* \upT N$ over $M$, which admits
a natural metric induced by $g$ and $h$, giving sense to $\Vert \upd f \Vert$. This is the so-called \emph{Hilbert-Schmidt norm} of $\upd f$,
which is also described as $\Vert \upd f \Vert^2 = \tr_g(f^* h)$.

\begin{definition} \label{def:HarmonicMap}
 A map $f \colon M \to N$ is \emph{harmonic} if it is a critical point of the energy functional \eqref{eq:EnergyFunctionalSmooth}.
\end{definition}
\noindent This means concretely that:
\begin{equation}
 \ddt\evalat{t=0} E(f_t) = 0
\end{equation}
for any smooth deformation $(f_t) \colon (-\delta, \delta) \times M \to N$ of $f = f_0$. 
Note that one should work with compactly supported deformations when $M$ is not compact, as the energy could be infinite.

A more tangible characterization of harmonicity is given by the Euler-Lagrange equation for $E$, which
takes the form $\tau(f) = 0$ where $\tau(f)$ is the
\emph{tension field} of $f$: this is an immediate consequence of the first variational formula below (\autoref{prop:FirstVariationalFormulaSmooth}).
First we define the tension field.
Note that the bundle $\upT^*M \otimes f^* \upT N$ admits a natural connection $\nabla$ induced by the Levi-Civita connections of $g$
and $h$. Hence one can take the covariant derivative $\nabla (\upd f) \in \Gamma(\upT^*M \otimes \upT^*M \otimes f^* \upT N)$ (we use the notation $\Gamma$ for the space of smooth sections), also denoted $\nabla^2 f$. It is easily shown to be symmetric in the first two factors.

\begin{definition}
The \emph{vector-valued Hessian} of $f$ is
\begin{equation}
 \nabla^2 f \coloneqq \nabla(\upd f) \in \Gamma(\upT^*M \otimes \upT^*M \otimes f^* \upT N)~.
\end{equation}
The contraction (trace) of $\nabla^2 f$ on its first two indices using the metric $g$ is the \emph{tension field} of $f$:
\begin{equation}
 \tau(f) \coloneqq \tr_g(\nabla^2 f) \in \Gamma(f^* \upT N)~.
\end{equation}
\end{definition}

Note that the vector-valued Hessian generalizes both the usual Hessian (when $N = \bR$) and the (vector-valued) second fundamental form (when $f$ is an isometric immersion). 
Accordingly, the tension field generalizes both the Laplace-Beltrami operator and the (vector-valued) mean curvature.

\begin{proposition}[First variational formula for the energy] \label{prop:FirstVariationalFormulaSmooth}
Let $f \colon (M,g) \to (N, h)$ be a smooth map and let $(f_t)$ be a smooth deformation of $f$. 
Denote by $V \in \Gamma(f^* \upT N)$ the associated infinitesimal deformation defined as $V_x = \ddt\evalat{t=0} f_t(x)$. Then
\begin{equation} \label{eq:FirstVariationalFormula}
 \ddt\evalat{t=0} E(f_t) = -\int_M \langle \tau(f)_x, V_x\rangle \, \upd v_g(x)
\end{equation}
where $h = \langle \cdot , \cdot \rangle$ is the Riemannian metric in $N$.
\end{proposition}

One can introduce a natural $\upL^2$ inner product of two infinitesimal deformations $V, W \in \Gamma(f^* \upT N)$ (also called \emph{vector fields along $f$}):
\begin{equation}
\label{eq:InnerProductSmooth}
\langle V, W \rangle = \int_M \langle V_x, W_x\rangle \, \upd v_g(x)~.
\end{equation}
There is in fact a natural smooth structure on $\cC^\infty(M,N)$, making it
an infinite-dimensional manifold, which identifies the tangent space at $f$ as
\begin{equation} \label{eq:TangentSpaceMapSmooth}
 \upT_f \cC^\infty(M,N) = \Gamma(f^* \upT N)~,
\end{equation}
we refer to \cite[Chapter IX]{MR1471480} for details.
With respect to this smooth structure, \eqref{eq:InnerProductSmooth} defines a Riemannian metric on $\cC^\infty(M,N)$, and \eqref{eq:FirstVariationalFormula}
can simply be put:
\begin{equation} \label{eq:GradientEnergy}
  \grad E(f) = -\tau(f)~.
\end{equation}
Next we compute the second variation of the energy (like the first variation, this is already in \cite{MR0164306}):
\begin{proposition}[Second variational formula for the energy] \label{prop:SecondVariationalFormulaSmooth}
Let $(f_{st}) \colon (-\delta, \delta)^2 \times M \to N$ be a smooth deformation of $f = f_{00}$ .
Denote $V=\frac{\partial f}{\partial s}\evalat{s=0}$ and $W=\frac{\partial f}{\partial t}\evalat{t=0}$.
Then
\begin{equation}
 \frac{\partial^2 E(f_{st})}{\partial s \partial t}\evalat{s=t=0}
 = \int_M \left(\left \langle \nabla V, \nabla W \right \rangle - \tr_g \left \langle R^N(\upd f,V) W,\upd f \right\rangle
+ \left \langle \nabla_{\frac \partial {\partial t}} \frac{\partial f}{\partial s} , \tau(f) \right\rangle \right)  \, \upd v_g
\end{equation}
where $R^N$ is the Riemann curvature tensor\footnote{For us the curvature tensor is $R(X,Y)Z = \nabla^2_{X, Y} Z - \nabla^2_{Y, X} Z$. Some authors' convention differs in sign, \eg \cite{MR2088027}.} on $N$.
\end{proposition}
\noindent When $(f_{st})$ is a geodesic variation, \ie{} $f_{st}(x) = \exp_{f(x)}(s V_x + tW_x)$, the third term in the integral vanishes.
This yields the formula for the Hessian of the energy functional:
\begin{equation} \label{eq:HessianEnergySmooth}
 \Hess(E)\evalat{f}(V, W) = \int_M \left(\left \langle \nabla V, \nabla W \right \rangle - \tr_g \left \langle R^N(V, \upd f) \upd f, W \right\rangle \right)  \, \upd v_g~.
\end{equation}
When $M$ is closed, this can also be written $\Hess(E)\evalat{f}(V, W) = \langle J(V), W \rangle$ using the $\upL^2$ Riemannian metric \eqref{eq:InnerProductSmooth},
where $J(V) = -\tr_g(\nabla^2 V + R^N(V, \upd f) \upd f)$ is the \emph{Jacobi operator}.

\subsection{\texorpdfstring{Energy functional on $\upL^2(M,N)$ and more general spaces}{Energy functional on L2(M,N) and more general spaces}}
\label{subsec:JostEnergy}

The energy functional can be extended to maps that are merely in $\upL^2(M,N)$. 
First let us define this function space. Assume $M$ is compact.
The \emph{$L^2$-distance} between two measurable maps $f_1,f_2 \colon M \to N$ is
\begin{equation} \label{eq:L2DistanceSmooth}
 d(f_1, f_2) = \left(\int_M d( f_1(x), f_2(x))^2 \, \upd v_g(x)\right)^{\frac 12}~.
\end{equation}
If $f_1, f_2$ are both smooth, this is the distance induced by the 
$\upL^2$ Riemannian metric \eqref{eq:InnerProductSmooth}, provided there exists a geodesic between $f_1$ and $f_2$. 
A measurable map $f \colon M \to N$ is declared in $\upL^2(M,N)$ when it is within finite distance of a constant map. For $r> 0$, one can then 
define an \emph{approximate $r$-energy} of $f \in \upL^2(M,N)$:
\begin{equation} \label{eq:ApproximateEnergyJost}
 E_r(f) =  \frac{1}{2} \int_M \int_M  \eta_r(x,y) \, {d(f(x), f(y))}^2 \, \upd v_g(y) \, \upd v_g (x)
\end{equation}
where $\eta_r(x,y)$ is a \emph{kernel} that may be chosen $\eta_r(x,y) = \frac{\mathbf{1}_r(x,y)}{r^2 V_m(r)}$,
where $V_m(r)$ is the volume of a ball of radius $r$ in a Euclidean space of dimension $m = \dim M$
and $\mathbf{1}_r(x,y)$ is the characteristic function of $\{(x,y) \in M^2 : d(x,y) < r\}$ in $M \times M$
(see \cite[\S 4.1]{MR1451625} for a discussion of the choice of kernel).
One can show that the functional $E_r$ is continuous on $\upL^2(M,N)$. Moreover, the limit:
\begin{equation} \label{eq:EnergyFunctionalJost}
 E(f) \coloneqq \lim_{r \to 0} E_r(f)
\end{equation}
exists in $[0, \infty]$ for every $f\in \upL^2(M,N)$. The resulting energy functional $E$ is lower semi-continuous on $\upL^2(M,N)$
and coincides with \eqref{eq:EnergyFunctionalSmooth} on $\cC^\infty(M,N)$.
A measurable map $f \colon M \to N$ is declared in the Sobolev space $\upH^1(M,N)$ if it is in $\upL^2(M,N)$ and has finite energy. 
The spaces $\upL_{\text{loc}}^2(M,N)$ and $\upH_{\text{loc}}^1(M,N)$ are similarly defined by restricting to compact sets.
One can then define a (weakly) harmonic map as a critical point of the energy functional
in $\upH_{\text{loc}}^1(M,N)$. 
Any continuous weakly harmonic map is smooth \cite[Theorem 9.4.1]{MR3726907}  (the continuity assumption can be dropped when $M$ and $N$ are compact and $N$ has nonpositive curvature: \cite[Corollary 9.6.1]{MR3726907}). 

In addition to opening the way for tools from functional analysis, this approach can be generalized to much more general spaces than Riemannian manifolds. Indeed, assume $M = (M,\mu)$ is a measure space and $N = (N,d)$ is a metric space. The space $\upL^2(M,N)$ may be defined as before, and given a choice of kernel $\eta_r$ for $r>0$, one can define energy functionals $E_r \colon \upL^2(M,N) \to \bR$ using \eqref{eq:ApproximateEnergyJost}. For a suitable choice of $\eta_r$ and of a sequence $r_n \to 0$, the energy functional is $E = \lim_{n \to +\infty} E_{r_n}$. More precisely, one must ensure that $E$ is the $\Gamma$-limit of the functionals $E_{r_n}$. We refer to \cite[Chap. 4]{MR1451625} for details and \cite{MR1201152} for the theory of $\Gamma$-convergence. 
$\Gamma$-convergence is adequate here because it ensures that minimizers of $E_r$ converge to minimizers of $E$. This point of view on the theory of harmonic maps was developed by Jost \cite{MR1385525, MR1360608, MR1449406, MR1451625}. A similar approach was developed by Korevaar-Schoen \cite{MR1266480, MR1483983}.

\subsection{The heat flow} \label{subsec:SmoothHeatFlow}

Going back to the smooth setting, assume that $M$ is compact and $N$ is complete and has nonpositive curvature. The formula for the Hessian of the energy \eqref{eq:HessianEnergySmooth} shows that it is nonnegative, in other words $E$ is a convex function on $\mathcal{C}^\infty(M,N)$ with respect to the $\upL^2$ Riemannian metric. This makes it reasonable to expect existence and in certain cases uniqueness of harmonic maps, which are necessarily energy-minimizing in this setting
(we discuss this further in \autoref{subsec:ConvexityEnergyExistenceHarmonic}). A natural approach to minimize a convex function is the gradient flow, called \emph{heat flow} in this setting: given $f_0 \in \mathcal{C}^\infty(M,N)$, consider the initial value problem $\ddt f_t = -\grad E(f_t)$, that is in light of \eqref{eq:GradientEnergy}:
\begin{equation}
 \ddt f_t = \tau(f_t)~.
\end{equation}
This flow exists for all $t \geqslant 0$. Moreover, if the range of $f_t$ remains in some fixed compact subset of $N$,
then $f_t$ converges to a harmonic map as $t \to \infty$, uniformly and in $\upL^2(M,N)$ (in fact, in $\mathcal{C}^\infty(M,N)$). Otherwise, 
there exists no harmonic map homotopic to $f$. In particular, when $N$ is compact, any $f_0 \in \mathcal{C}^\infty(M,N)$ is homotopic to a smooth energy-minimizing harmonic map. Moreover, such a harmonic map is unique, unless it is constant or maps into a totally geodesic flat submanifold of $N$ in which case non-uniqueness is realized by translating $f$ in the flat. These foundational results are due to Eells-Sampson \cite{MR0164306} and Hartman \cite{MR0214004}.

\subsection{Equivariant harmonic maps} \label{subsec:EquivariantHarmonicMaps}

Instead of working with maps between compact manifolds, it can be useful to study their equivariant lifts to the universal covers. 
Indeed, up to being careful with basepoints, any continuous map $f \colon M \to N$ lifts to a unique $\rho$-equivariant map $\tilde{f} \colon \tilde{M}\to \tilde{N}$,
where $\rho \colon \pi_1 M \to \pi_1 N$ is the group homomorphism induced by $f$. Note that $\rho$ only depends on the homotopy class of $f$,
and if $N$ is aspherical (\ie{} $\tilde{N}$ is contractible), then conversely any $\rho$-equivariant continuous map $M \to N$ is the lift of some continuous map $M \to N$ homotopic to $f$.

This approach enables the following generalization: let $X$ and $Y$ be two Riemannian manifolds, denote $\Isom(X)$ and $\Isom(Y)$ their groups of isometries. Let $\Gamma$ be a discrete group. Given group homomorphisms
$\rho_\tL \colon \Gamma \to \Isom(X)$ and $\rho_\tR \colon \Gamma \to \Isom(Y)$, a map $f \colon X \to Y$ is called $(\rho_\tL, \rho_\tR)$-equivariant if:
\begin{equation}
 f \circ \rho_\tL(\gamma) = \rho_\tR(\gamma) \circ f
\end{equation}
for all $\gamma \in \Gamma$. Note that the quotients $X/\rho_\tL(\Gamma)$ and $Y/\rho_\tR(\Gamma)$ can be pathological, but
the space of equivariant maps $X \to Y$ remains ripe for study.

The heat flow approach of Eells-Sampson to show existence of harmonic maps between compact Riemannian manifolds when the target is nonpositively curved has been
successfully adapted to the equivariant setting by various authors. The adequate condition for guaranteeing existence of equivariant harmonic maps is the \emph{reductivity} of the target representation. More precisely:

\begin{theorem}[\cite{MR1049845}] \label{thm:Labourie}
Let $M$ and $N$ be Riemannian manifolds, assume $N$ is Hadamard. Denote by $\rho_\tL \colon \pi_1M \to \Isom(\tilde{M})$ the action by deck transformations
and let $\rho_\tR \colon \pi_1M \to \Isom(N)$ be any group homomorphism. If $\rho_\tR$ is reductive, then there exists a $(\rho_\tL, \rho_\tR)$-equivariant harmonic map $\tilde{M} \to N$. The converse also holds provided $N$ is without flat half-strips.
\end{theorem}
Less general versions of this theorem had previously been established by Donaldson \cite{MR887285} (for $N = \bH^3$) and Corlette \cite{MR965220} (for $N$ a Riemannian symmetric space of noncompact type).
The notion of being \emph{reductive} for a group homomorphism $\rho \colon \pi_1M \to G$ 
can be described algebraically when $N = G/K$ is a Riemannian symmetric space of noncompact type\footnote{When $G$ is an algebraic group, a subgroup $H \subseteq G$ is  \emph{completely reducible} if,
for every parabolic subgroup $P \subseteq G$ containing $H$, there is a Levi subgroup of $P$ containing $H$. 
Equivalently, the identity component of the algebraic closure of $H$ is a reductive subgroup (with trivial unipotent radical). A $G$-valued group homomorphism $\rho$ is called reductive (or completely reducible) when its image is a completely reducible subgroup. Refer to \cite{MR2931326} for details.}. Labourie \cite{MR1049845} generalized it to Hadamard manifolds. When $N$ has negative curvature, $\rho$ is reductive if and only if it fixes no point on the Gromov boundary $\partial_\infty N$ or it preserves a geodesic in $N$.

Wang \cite{MR1775138} and Izeki-Nayatani \cite{MR2174098} generalized \autoref{thm:Labourie} to Hadamard metric spaces using Jost's extended notion of reductivity \cite[Def. 4.2.1]{MR1451625}. Less general or different versions were previously established by \cite{MR1215595}, \cite[Thm 4.2.1]{MR1451625}, \cite{MR1483983}.

\subsection{Harmonic maps from surfaces}
\label{subsec:harmonicMapsSurfaces}

When $M = S$ is a surface, \ie{} $\dim M = 2$, it is easy to check that the energy density element 
$e(f) \, \upd v_g \coloneqq \frac{1}{2} \Vert \upd f \Vert^2 \, \upd v_g$ 
is invariant under conformal changes of the metric $g$.
Thus the energy functional only depends on the conformal class of $g$, as does the harmonicity of a map $S \to N$. 
A conformal structure on an oriented surface is equivalent to a complex structure (this follows from a result going back to Gauss \cite{Gauss1825}
on the existence of conformal coordinates). Hence one may talk about the energy and harmonicity of maps $X \to (N,h)$ where $X$ is a Riemann surface. 
Note however that the $L^2$ metric \eqref{eq:InnerProductSmooth} does change under conformal changes of $g$, therefore the tension field $\tau(f)$ does too,
as does the modulus of strong convexity of the energy (see \autoref{subsec:ConvexityRiemannian}).

One can see directly that the energy density element only depends on the complex structure $X$ on $S$ by writing the pullback metric $f^* h$ on $X$. 
Splitting it into types, one finds that 
\begin{equation}
 f^* h = \varphi_f + g_{f} + \bar{\varphi_f} 
\end{equation}
where $\varphi_f = (f^*h)^{(2,0)}$ is a complex quadratic differential on $X$, called the \emph{Hopf differential} of $f$, and $g_f = (f^*h)^{(1,1)}$ 
is $e(f)$
(more precisely, $g_f$ is the conformal metric with volume density $e(f) \, \upd v_g$).

The Hopf differential $\varphi_f$ plays an important role in Teichmüller theory.
First note that $f$ is conformal if and only if $\varphi_f = 0$. A key fact is that if $f$ is harmonic, then $\varphi_f$ is a holomorphic quadratic differential on $X$.
Wolf \cite{MR982185} proved that the Teichmüller space of $X$ is diffeomorphic to the vector space of holomorphic quadratic differentials
on $X$ by taking Hopf differentials of harmonic maps $X \to (S,h)$, where $h$ is a hyperbolic metric on $S$ (see \autoref{subsec:HighEnergy}) . 
We refer to \cite{MR2349668} for a beautiful review of the connections between harmonic maps and Teichmüller theory.

On a closed surface $S$ of negative Euler characteristic, it is convenient to choose the \emph{Poincaré metric} within a conformal class of metrics: it is the unique metric of constant curvature $-1$ (its existence is precisely the celebrated \emph{uniformization theorem}). 
This provides an identification of $\tilde{S}$ with the hyperbolic plane $\bH^2$ and an action of $\pi_1 S$ on $\bH^2$ by isometries.
Turning this identification around, whenever a Fuchsian (\ie{} faithful and discrete) representation $\rho_\tL \colon \pi_1 S \to \Isom^+(\bH^2)$ is chosen, we obtain  a hyperbolic surface $ \bH^2 / \rho_\tL(\pi_1 S) \approx S$.

%% file: Discretization.tex
\section{Discretization}
\label{sec:Discretization}

We fix some notation: for the remainder of the paper, $S$ is a smooth, closed, oriented surface of negative Euler characteristic (genus $\geqslant 2$).
We denote $\pi_1S$ the fundamental group of $S$ with respect to some basepoint that can be safely ignored.

\begin{remark}
 While this paper specializes the discretization to $2$-dimensional hyperbolic surfaces, most of the definitions can be generalized to higher-dimensional
 Riemannian manifolds. The sequel paper \cite{Gaster-Loustau-Monsaingeon2} treats the general Riemannian setting (while dropping the equivariant formulation, mostly for comfort).
\end{remark}

In \autoref{subsec:Meshes} we explain how to approach the energy minimization problem for smooth equivariant maps $\bH^2 \to N$
in order to allow effective computation by introducing meshes and subdivisions, discrete equivariant maps, and discrete energy.
Several of these notions are further discussed in \cite{Gaster-Loustau-Monsaingeon2}.
In the present paper they serve as a preamble to the more formal setting we develop in \autoref{subsec:Graphs} 
and they justify the choices made in the software \Harmony{}.

\subsection{Meshes and discrete harmonic maps} \label{subsec:Meshes}

Let us fix a hyperbolic structure on $S$ given by a Fuchsian representation $\rho_\tL$,
\ie{} an injective group homomorphism $\pi_1S \to \Isom^+(\bH^2)$ with discrete image.
This setup can be easily generalized to any Riemannian metric on $S$, but the hyperbolic metric is
best suited for computations.

\subsubsection*{Meshes and subdivisions}

Given a group homomorphism $\rho_\tR \colon \pi_1S \to \Isom(N)$ where $N$ is a Riemannian manifold (or a metric space), 
we would like to discretize $(\rho_\tL, \rho_\tR)$-equivariant maps $\bH^2 \to N$. To this end, we start by discretizing the 
domain hyperbolic surface with the notion of invariant mesh:

\begin{definition}
\label{def:Mesh}
A \emph{$\rho_\tL$-invariant mesh} of $\bH^2$ is an embedded graph $\cM$ in $\bH^2$ such that:
\begin{enumerate}[(i)]
 \item The vertex set $\cM^{(0)} \subset \bH^2$ (set of \emph{mesh points}) is invariant under 
 a cofinite action of $\rho_\tL(\pi_1S)$. 
 \item Every edge $e \in \cM^{(1)}$ is an embedded geodesic segment in $\bH^2$.
 \item The complementary components are triangles.
\end{enumerate}
\end{definition}

For the purpose of approximating smooth maps, we will need to take finer and finer meshes. This will be discussed in detail
in \cite{Gaster-Loustau-Monsaingeon2}, but let us describe the strategy that we have implemented
in the software \Harmony. A natural way to obtain a finer mesh from a given one is via geodesic subdivision. 
We indicate below an edge of $\cM$ with endpoints $x,y\in \cM^{(0)}$ by $e_{xy}$, 
and let $m(x,y) \in \bH^2$ be the midpoint of $x$ and $y$. It is easy to see that the following is well-defined:

\begin{definition}
\label{def:Refinement}
The \emph{refinement} of  a $\rho_\tL$-invariant mesh $\cM$ is the $\rho_\tL$-invariant mesh $\cM'$ such that:
\begin{enumerate}[(i)]
\item The vertices of 
$\cM'$ are the vertices of $\cM$ plus all midpoints of edges of $\cM$.
\item The edges of 
$\cM'$ are given by $x\sim m(x,y)$ and $y\sim m(x,y)$ for each edge $e_{xy}$, 
and $m(x,y) \sim m(x,z)$ for each triple of vertices $x,y,z$ that span a triangle in $\cM$.
\end{enumerate}
\end{definition}   

Evidently, this refinement may be iterated. See \autoref{fig:meshes} for an illustration of a $\rho_\tL$-invariant mesh and its refinement generated by the software
\Harmony{}.

\begin{figure}
\centering
\begin{subfigure}{.5\textwidth}
  \centering
  \includegraphics[width=.98\textwidth]{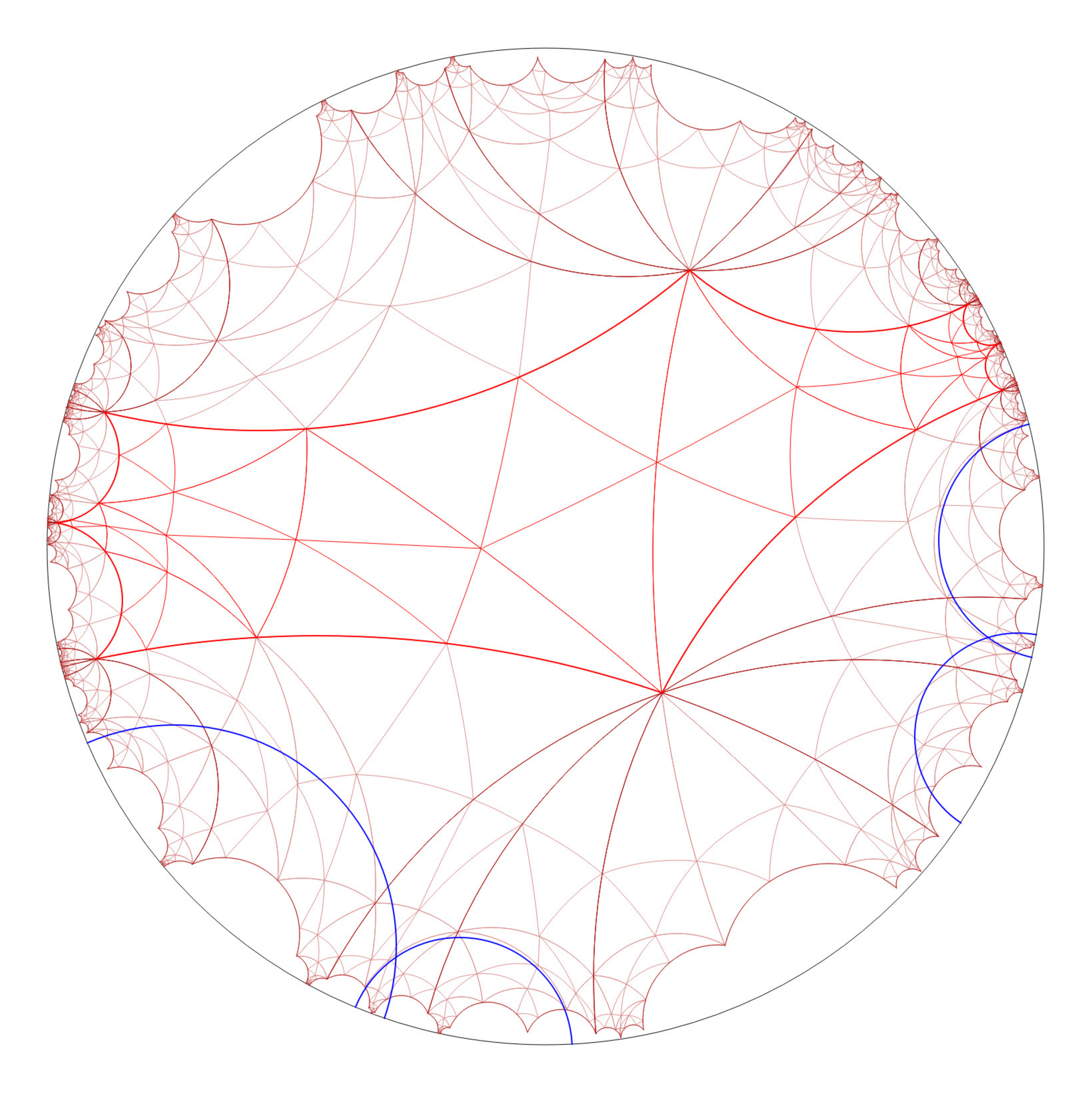}
  \caption{An invariant mesh of $\bH^2$}
  \label{fig:mesh0}
\end{subfigure}%
\begin{subfigure}{.5\textwidth}
  \centering
  \includegraphics[width=.98\textwidth]{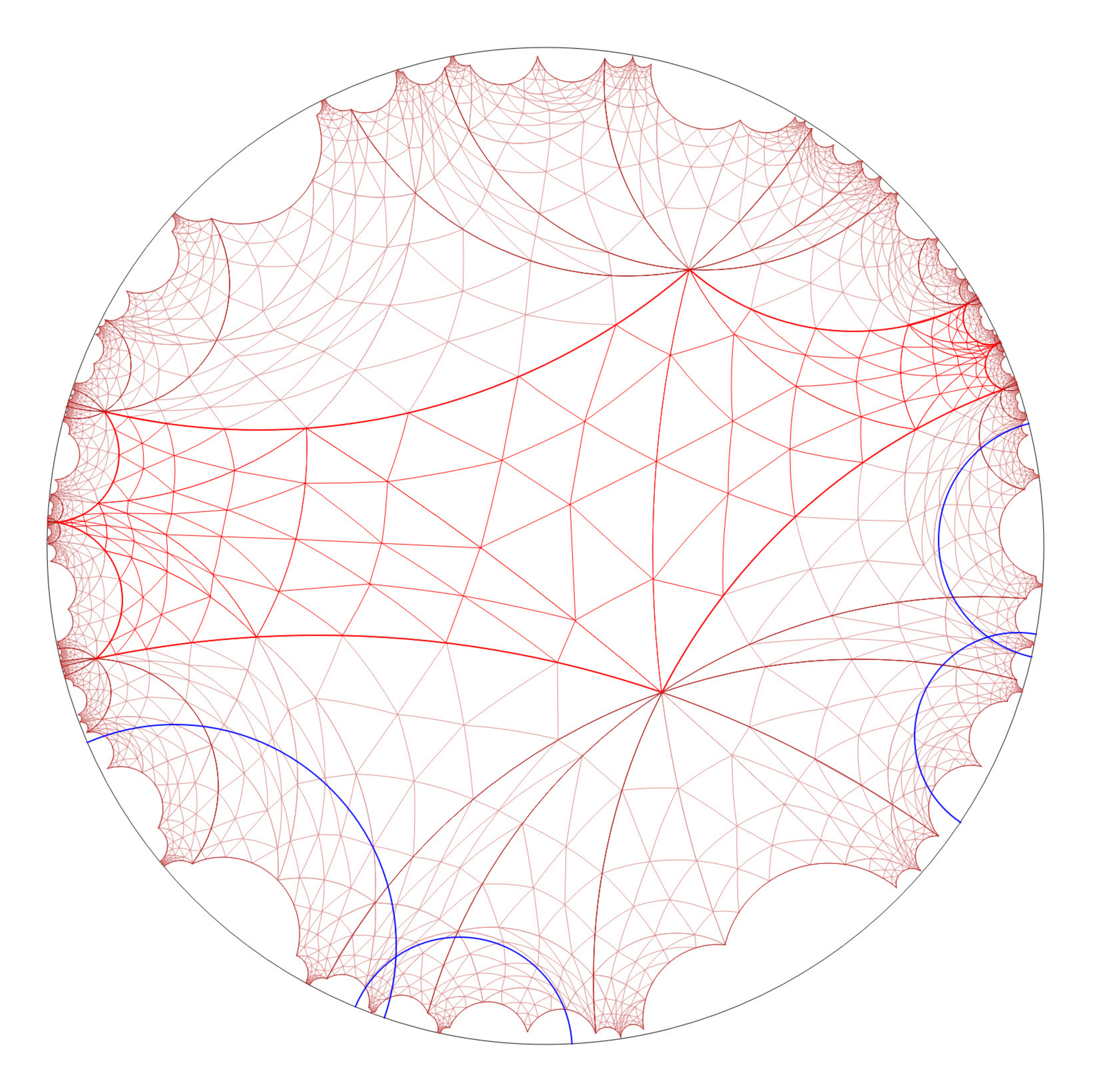}
  \caption{Refinement of order $1$}
  \label{fig:mesh1}
\end{subfigure}
\caption{An invariant mesh of the Poincaré disk model of $\bH^2$ on the left, its refinement of order 1 on the right.
The brighter central region is a fundamental domain. The blue circle arcs are the axes of the generators of $\rho_\tL(\pi_1S)$.
Pictures generated by \Harmony{}.}
\label{fig:meshes}
\end{figure}


\subsubsection*{Discrete equivariant maps}

Given a $\rho_\tL$-invariant mesh $\cM$ and a group homomorphism $\rho_\tR \colon \pi_1S \to \Isom(N)$,
we call \emph{discrete equivariant map $\bH^2 \to N$ along $\cM$} a $(\rho_\tL, \rho_\tR)$-equivariant map from the vertex set $\cM^{(0)}$ to $N$.
We denote $\Map_{\cM}(\bH^2, N)$ the set of discrete equivariant maps $\bH^2 \to N$ along $\cM$.
Note that $\Map_{\cM}(\bH^2, N) \approx N^{V}$ where $V = \cM^{(0)}/\rho_\tL(\pi_1S)$ is the set of equivalence classes of meshpoints, which is finite. Therefore:
\begin{proposition} \label{prop:EquivMapsSmoothManifold1}
If $N$ is a finite-dimensional smooth manifold, then so is $\Map_{\cM}(\bH^2, N)$.
\end{proposition}

Denoting $\cC_\text{eq}^0(\bH^2,N)$ the space of continuous equivariant maps $\bH^2 \to N$, 
the forgetful map 
\begin{equation} \label{eq:ForgetfulMap}
\cC_\text{eq}^0(\bH^2,N) \to \Map_{\cM}(\bH^2, N)
\end{equation}
consists in restricting a continuous function to the set of meshpoints $\cM^{(0)}$. 

\begin{definition} \label{def:InterpolationScheme}
We shall call a right inverse of the forgetful map \eqref{eq:ForgetfulMap} an \emph{interpolation scheme}.
\end{definition}

While there is one most natural way to interpolate
discrete maps between Euclidean spaces (affine interpolation), there is no preferred way for arbitrary
Riemannian manifolds. Even in the case where both domain and target manifolds are the hyperbolic plane $\bH^2$, 
there are several reasonable interpolations to consider such as the barycentric interpolation and the harmonic interpolation.
However, these are not explicit, and with \Harmony{}
we work with a neat variant, the $\cosh$-center of mass interpolation (see \autoref{subsec:CoshCenterOfMass}).

\subsubsection*{Discretization of energy}

For a smooth $(\rho_\tL,\rho_\tR)$-equivariant map $f:\bH^2\to N$ where $N$ is a Riemannian manifold,
one defines the total energy of $f$ as:
\begin{equation}  \label{eq:EnergyFunctionalSmoothEquivariant}
 E(f) = \int_D \Vert \upd f \Vert^2 \, \upd v_g
\end{equation}
where $D \subset \bH^2$ is any fundamental domain for the action of $\pi_1S$. If $D$ is picked so that it coincides with a union of triangles
defined by $\cM$, then the energy can be written as a finite sum of energy integrals over each triangle.
When $f$ is discretized along $\cM$, only the values of $f$ on the meshpoints are recorded. Thus, 
a natural discretization of $E$ is obtained if one knows how to define the energy of a map from a triangle
whose values are only known at the vertices. Given an interpolation scheme (cf \autoref{def:InterpolationScheme}),
one can simply take the energy of the interpolated map.


Another approach consists in defining the energy \emph{à la} Jost / Korevaar-Schoen as in \autoref{subsec:JostEnergy}.
One can take the graph defined by $\cM$ with metric induced from $\bH^2$ as a domain metric space, introduce a measure
that approximates the area density of $\bH^2$, and choose an appropriate kernel $\eta(x,y)$.

A third approach consists in choosing a discrete energy that is a weighted sum of distances squared as in \autoref{def:DiscreteEnergyMesh}
and \autoref{def:EnergyFunctionalGraph}. This approach provides a natural extension of the classical notion of real-valued harmonic functions defined on graphs 
\cite{MR2882891,MR1421568,MR1829620}, that is, functions whose value at any vertex is the average of the values on the neighbors.

It turns out that all three approaches can be made to coincide (\autoref{prop:JostEnergyGraph}), or almost coincide for fine meshes, for the appropriate choices involved in the different definitions. This is thoroughly discussed in the sequel paper \cite{Gaster-Loustau-Monsaingeon2}. 



\begin{definition} \label{def:DiscreteEnergyMesh}
Let $\cM$ be a $\rho_\tL$-invariant mesh in $\bH^2$ such that all the complementary triangles are acute.
The \emph{discrete energy} of a discrete equivariant map $f\in \Map_\cM(\bH^2, N)$ is defined by
\begin{equation} \label{eq:DiscreteEnergyMesh}
 E_{\cM}(f) = \frac{1}{2} \sum_{e = e_{xy} \in \cE} \omega_{xy} \,  d(f(x),f(y))^2
\end{equation}
where:
\begin{itemize}
 \item $\cE \subset \cM^{(1)}$ is any fundamental domain for the action of $\pi_1S$ on the set of edges.
 \item Inside the sum, $x$ and $y$ are the vertices connected by the edge $e$.
 \item $\omega_{xy}$ is the half-sum of the cotangents of two angles: one for each of the two triangles sharing the edge $e=e_{xy}$, in which we take the angle 
of the vertex facing the edge $e$.
\end{itemize}
\end{definition}

This definition is a generalization of the energy considered by Pinkall-Polthier \cite{MR1246481}, 
for whom the domain is a triangulated surface with a piecewise Euclidean metric and $N = \bR^n$. 
In their setting, the discrete energy coincides with the energy relative to the linear interpolation.
In \cite{Gaster-Loustau-Monsaingeon2} we show that the discrete energy $E_\cM$ converges to the smooth energy $E$ under iterated refinement of the mesh $\cM$ 
in a suitable function space. 

Of course, we can now define a \emph{discrete equivariant harmonic map} $f\in \Map_\cM(\bH^2, N)$ as a critical point of the discrete energy functional $E_\cM$.
While several authors have shown the existence and uniqueness of minimizers of the discrete energy in various contexts (\eg{} \cite{MR1775138,MR1848068,MR1938491}), 
in this paper we analyze its strong convexity, which makes it better suited for effective minimization.
Our approach requires a Riemannian metric on $\Map_\cM(\bH^2, N)$ (cf.~\autoref{subsec:ConvexityRiemannian}), which should approach the $\upL^2$ Riemannian metric
of $\mathcal{C}^\infty(M,N)$ (cf.~\autoref{subsec:EnergyFunctionalHarmonicMaps}). 
In the next subsection, we develop a more general framework where these ideas apply.

\subsection{Equivariant harmonic maps from graphs} \label{subsec:Graphs}

The definition of the discrete energy functional $E_\cM$ (\autoref{def:DiscreteEnergyMesh}) is easily generalized to any system of positive weights
indexed by the edges of $\cM$. On the other hand, the Riemannian structure of $\Map_\cM(\bH^2, N)$ requires a measure on the domain:
while in the smooth case 
one has the volume density of the Riemannian metric, in the discrete case it can be recorded by a system of weights on the vertices. 
All of this information can be captured using only the graph structure of $\cM$. 

\subsubsection*{\texorpdfstring{$\tilde{S}$-triangulated graphs}{S-triangulated graphs}}

Recall that a triangulation $\cT$ of a surface is the data of a simplicial complex $K$ and a homeomorphism $h$ from $K$ to the surface. 
Lifting $\cT$ to the universal cover, we find a triangulation whose underlying graph $\cG$ (\ie{} 1-skeleton) is \emph{locally cyclic}, meaning that the open neighborhood of any vertex (subgraph induced on the neighbors) is a cycle\footnote{
Note that $\cG$ has the property that every face ($2$-simplex) of the triangulation is a triangle ($3$-vertex complete subgraph) in $\cG$; 
when the converse is also true one says that $\cT$ is a \emph{Whitney triangulation}. In other words,
a Whitney triangulation can be recovered as the flag complex spanned by $\cG$. Let us cite \cite{MR2002076} here: \emph{Whitney triangulations are quite amenable for graph-theoretical considerations because they are determined by their underlying graph: the two-dimensional faces are just the triangles of the graph. In other words, we can think of a Whitney triangulation as an object wearing two hats: on one hand it is just a graph, and on the other hand it is a 2-dimensional simplicial complex which in turn can be considered either as a purely combinatorial object or as a topological surface with a fixed simplicial decomposition.}
It can be shown (\cite[Prop. 14]{MR2002076}) that a simple graph $\cG$ is the underlying graph of some Whitney triangulation of a surface
if and only if $\cG$ is locally cyclic.}.
This motivates the following definition:

\begin{definition}
\label{def:TriangulatedGraph}
Given a topological surface $S$, an \emph{$\tilde{S}$-triangulated graph} is a locally cyclic graph $\cG$ 
with a free, cofinite action of $\pi_1 S$ by graph automorphisms.
\end{definition}

$\tilde{S}$-triangulated graphs are precisely the graphs that arise as $1$-skeleta of triangulations. 
When $\cG$ is an $\tilde{S}$-triangulated graph, we denote the associated group action on the set of vertices 
by $\rho_\tL \colon \pi_1S \to \Aut(\cG^{(0)})$. Let $N$ be a metric space and $\rho_\tR \colon \pi_1S \to \Isom(N)$ a group homomorphism.

\begin{definition}
Given $\rho_\tL$ and $\rho_\tR$ as above, we call a $(\rho_\tL, \rho_\tR)$-equivariant map $\cG^{(0)} \to N$ an \emph{equivariant map from $\cG$ to $N$} .
The space of such equivariant maps will be denoted $\Map_{\text{eq}}(\cG, N)$.
\end{definition}

As in \autoref{prop:EquivMapsSmoothManifold1} we have:

\begin{proposition} \label{prop:MapGammaManifold}
If $N$ is a finite-dimensional smooth manifold, so is $\Map_{\text{eq}}(\cG, N)$.
\end{proposition}

\subsubsection*{Edge-weighted graphs and the energy functional}

\begin{definition} \label{def:EdgeWeightedGraph}
 Let $\cG$ be an $\tilde{S}$-triangulated graph. We say that $\cG$ is \emph{edge-weighted}
 if it is given a system of \emph{edge weights}, \ie{} a family of positive real numbers 
 $(\omega_e)_{e \in \cG^{(1)}}$ indexed by the set of edges $\cG^{(1)}$, that is invariant
 under the action of $\pi_1S$.
\end{definition}

Clearly, the data of a system of edge weights is equivalent to the data of a function
\begin{equation}
 \eta_0 \colon \cG^{(0)} \times \cG^{(0)} \to [0, +\infty)
\end{equation}
that is symmetric, invariant under the diagonal action of $\pi_1 S$, and such that $\eta_0(x,y) > 0$
if and only if $x$ and $y$ are adjacent. 
\begin{definition} \label{def:Prekernel}
A function $\eta_0$ as above is called a \emph{pre-kernel} on the $\tilde{S}$-triangulated graph $\cG$.
\end{definition}
The motivation for introducing this notion will become clear in \autoref{def:Kernel} and \autoref{prop:JostEnergyGraph}.
We are now ready to define the energy functional:
\begin{definition}
\label{def:EnergyFunctionalGraph}
Let $\cG$ be an $\tilde{S}$-triangulated graph with a system of edge weights 
 $(\omega_e)_{e \in \cG^{(1)}}$, and let $\rho_\tR \colon \pi_1S \to \Isom(N)$ be a group homomorphism where
 $N$ is a metric space. The \emph{energy functional} $E_\cG \colon \Map_{\text{eq}}(\cG, N) \to \bR$ is defined by
\begin{equation}
\label{eq:EnergyFunctionalGraph}
 E_{\cG}(f) = \frac{1}{2} \sum_{e = e_{xy} \in \cE} \omega_{xy} \, d(f(x),f(y))^2
\end{equation}
where $\cE \subset \cG^{(1)}$ is any fundamental domain for the action of $\pi_1S$.
\end{definition}

When $N$ is a Hadamard manifold\footnote{A \emph{Hadamard manifold} is a complete, simply connected Riemannian manifold of nonpositive curvature.
On a Hadamard manifold the distance squared function to a fixed point is smooth, while in general it may not be differentiable on the cut locus.}, $E_{\cG}$ is a 
smooth function on the manifold $\Map_{\text{eq}}(\cG, N)$. Of course we now call a map $f \in \Map_{\text{eq}}(\cG, N)$ \emph{harmonic} when it is a critical point of the energy functional $E_\cG$. When $N$ is not a Hadamard manifold but merely a metric space, one can still define (locally) energy-minimizing harmonic maps.

Note that, taking $\omega_e = 1$ for all $e \in \cG^{(1)}$ and $N = \bR$, a harmonic map from $\cG$ to $\bR$ in the sense above coincides with the classical notion of harmonicity for real-valued functions on graphs. The well-known mean value property of harmonic functions is generalized: 

\begin{proposition} \label{prop:HarmonicCenterOfMassGraph1}
Let $\cG$ be an edge-weighted triangulated graph and let $N$ be a metric space. If
$f \in \Map_{\text{eq}}(\cG, N)$ is an energy-minimizing harmonic map 
then $f(x)$ is a center of mass of the weighted system of points $\{(f(y), \omega_{xy})\}_{y \sim x}$ in $N$ for every $x \in \cG^{(0)}$.
\end{proposition}
Refer to \autoref{sec:CenterOfMass} for the definition and elementary properties of centers of mass.

\begin{proof}
If $f(x)$ was not the center of mass of its neighbors, then the part of \eqref{eq:EnergyFunctionalGraph} that involves $x$ could be decreased by replacing
$f(x)$ by the center of mass while leaving the other values unchanged.
\end{proof}

\subsubsection*{Vertex weighted-graphs and the Riemannian structure}

\begin{definition} \label{def:VertexWeightedGraph}
 Let $\cG$ be an $\tilde{S}$-triangulated graph. We say that $\cG$ is \emph{vertex-weighted}
 if it is given a system of \emph{vertex weights}, \ie{} a family of positive real numbers 
 $(\mu_v)_{v \in \cG^{(0)}}$ indexed by the set of vertices $\cG^{(0)}$, that is invariant
 under the action of $\pi_1S$.
\end{definition}

We think of a system of vertex weights as a $\pi_1S$-invariant Radon measure $\mu$ on $\cG^{(0)}$. Of course,
this is simply a $\pi_1S$-invariant function $\mu \colon \cG^{(0)} \to (0, +\infty)$, but our viewpoint for discretization is to approximate the smooth theory where $\mu$ is the volume density of a Riemannian manifold.

Assume now that $N$ is a finite-dimensional Riemannian manifold and let $\rho_\tR \colon \pi_1 S \to \Isom(N)$ be a group homomorphism.
We saw (\autoref{prop:MapGammaManifold}) that $\Map_{\text{eq}}(\cG, N)$ is a smooth manifold. Moreover, it is easy to describe its tangent space.
\begin{proposition} \label{prop:MapGammaTangentSpace}
The tangent space at $f \in \Map_{\text{eq}}(\cG, N)$ is:
\begin{equation} \label{eq:MapGammaTangentSpace1}
 \upT_f \Map_{\text{eq}}(\cG, N) = \Gamma_{\text{eq}}(f^* \upT N)
\end{equation}
where $f^* \upT N$ is the pullback of the tangent bundle $\upT N$ to $\cG^{(0)}$ and $\Gamma_{\text{eq}}(f^* \upT N)$ is its space
of $\pi_1S$- equivariant smooth sections
. Equivalently, if $\cV \subseteq \cG^{(0)}$ is any fundamental domain for the action of $\pi_1S$,
\begin{equation} \label{eq:MapGammaTangentSpace2}
 \upT_f \Map_{\text{eq}}(\cG, N)  = \bigoplus_{x \in \cV} \upT_{f(x)} N~.
\end{equation}
\end{proposition}

Notice of course the similarity of \eqref{eq:MapGammaTangentSpace1} with \eqref{eq:TangentSpaceMapSmooth}. Using the measure $\mu$,
one can define a natural $\upL^2$ Riemannian metric on $\Map_{\text{eq}}(\cG, N)$ analogous to \eqref{eq:InnerProductSmooth}:

\begin{definition} \label{def:L2RiemannianMetricGraph}
 Let $(\cG, \mu)$ be an $\tilde{S}$-triangulated vertex-weighted graph, and let $\rho_\tR \colon \pi_1S \to \Isom(N)$ where
 $N$ is a Riemannian manifold. The $\upL^2$ Riemannian metric on $\Map_{\text{eq}}(\cG, N)$ is given by:
\begin{equation} \label{eq:L2RiemannianMetricGraph1}
\langle V, W \rangle = \int_\cV \langle V_x, W_x\rangle \, \upd \mu(x)
\end{equation}
where $V, W \in \Gamma_{\text{eq}}(f^* \upT N)$ and $\cV \subseteq \cG^{(0)}$ is any fundamental domain for the action of $\pi_1S$.
\end{definition}

Of course, one can write more concretely:
\begin{equation} \label{eq:L2RiemannianMetricGraph2}
\langle V, W \rangle = \sum_{x \in \cV} \mu(x) \langle V_x, W_x\rangle~.
\end{equation}

One can easily derive that the unit speed geodesics in $\Map_{\text{eq}}(\cG, N)$ are the one-parameter families of functions
$(f_t(x))_{x \in \cG^{(0)}}$ given by $f_t(x) = \exp(t V_x)$, where $V \in \Gamma_{\text{eq}}(f^* \upT N)$ is a unit vector,
and, provided $N$ is connected, the Riemannian distance in $\Map_{\text{eq}}(\cG, N)$ is simply given by 
\begin{equation} \label{eq:L2DistanceDiscrete}
 d(f,g)^2 = \sum_{x \in \cV} \mu(x) \, d(f(x), g(x))^2~,
\end{equation}
where on the right hand-side $d$ is the Riemannian distance in $N$. Of course notice that \eqref{eq:L2DistanceDiscrete} is just the discretization of \eqref{eq:L2DistanceSmooth}.

\subsubsection*{Biweighted graphs}

\begin{definition} \label{def:BiweightedGraph}
 Let $\cG$ be an $\tilde{S}$-triangulated graph (\autoref{def:TriangulatedGraph}). We say that $\cG$ is \emph{biweighted}
 if it is both edge-weighted (\autoref{def:EdgeWeightedGraph}) and vertex-weighted (\autoref{def:VertexWeightedGraph}).
\end{definition}

From the discussion of the previous paragraph, when $\cG$ is an $\tilde{S}$-triangulated biweighted graph and $N$ is a Riemannian manifold 
with a group homomorphism $\rho_\tR \colon \pi_1S \to N$, the space of equivariant maps $\Map_{\text{eq}}(\cG, N)$ is a Riemannian manifold
and the energy is a continuous function $E_\cG \colon \Map_{\text{eq}}(\cG, N) \to \bR$. Moreover $E_\cG$ is smooth when $N$ is Hadamard.
In \autoref{sec:StrongConvexity} we show that $E_\cG$ is strongly convex under suitable restrictions on $\rho_\tR$, with an explicit bound on the modulus of strong convexity (\autoref{thm:StrongConvexityEnergyGraphGeneral}). 
This implies that there exists a unique equivariant harmonic map $\cG \to N$ that can be computed effectively through gradient descent (\autoref{sec:DiscreteHeatFlow}).

We pause to point out that our definition of the energy functional and harmonic maps in this setting coincides with Jost's theory
briefly described in \autoref{subsec:JostEnergy} (we refer to \cite{MR1449406, MR1451625} for details). 
First we introduce the kernel function associated to a biweighted graph:
\begin{definition} \label{def:Kernel}
The kernel function associated to a biweighted graph $\cG$ is the function 
 \begin{equation}
 \begin{split}
  \eta \colon \cG^{(0)} \times \cG^{(0)}  \to \bR\\
  (x,y) \mapsto \frac{\eta_0(x,y)}{2 \mu(x) \mu(y)}
 \end{split}
\end{equation}
where $\eta_0$ is the pre-kernel associated to the underlying edge-weighted graph (cf.~\autoref{def:Prekernel}) and $\mu$
is the measure on $\cG^{(0)}$ giving the vertex weights.
\end{definition}

The next proposition is trivial but conceptually significant:
\begin{proposition} \label{prop:JostEnergyGraph}
 The energy functional on $\Map_{\text{eq}}(\cG, N)$ is given by
 \begin{equation}
  E_\cG(f) = \frac{1}{2} \iint_\cV \eta(x,y)\, {d(f(x), f(y))}^2 \, \upd \mu(y) \, \upd \mu (x)
 \end{equation}
 where $\cV \subseteq \cG^{(0)} \times \cG^{(0)}$ is a fundamental domain for the diagonal action of $\pi_1S$.
\end{proposition}
\autoref{prop:JostEnergyGraph} implies that, choosing $\eta_r = \eta$ for all $r > 0$, the Jost energy functional
$E = \lim_{r\to 0} E_r$ (compare with \eqref{eq:EnergyFunctionalJost}) coincides with the energy functional $E_\cG$. In particular, our notion of harmonic maps
from graphs is a specialization of Jost's \emph{generalized harmonic maps}.

Next we observe that the Riemannian structure of $\Map_{\text{eq}}(\cG, N)$ allows us to define the discrete tension field as:
\begin{definition}
The \emph{tension field} of $f \in \Map_{\text{eq}}(\cG, N)$ is the vector field along $f$ denoted $\tau_\cG(f) \in \Gamma_{\text{eq}}(f^* \upT N)$
given by:
\begin{equation}
 \tau_\cG(f)\evalat{x} = \frac{1}{\mu(x)} \sum_{y \sim x} \omega_{xy} \exp_{f(x)}^{-1}(f(y))
\end{equation}
where we have denoted $\omega_{xy}$ the weight of the edge connecting $x$ and $y$.
\end{definition}

We have the discrete version of the first variational formula for the energy (\autoref{prop:FirstVariationalFormulaSmooth}):
\begin{proposition} \label{prop:FirstVariationalFormulaGraph}
The tension field is minus the gradient of the energy functional:
\begin{equation}
 \tau_\cG(f) = -\grad E_\cG(f)
\end{equation}
for any $f \in \Map_{\text{eq}}(\cG, N)$.
\end{proposition}

\begin{proof}
 In a Riemannian manifold $N$, 
 when $x_0 \in N$ is chosen such that $\exp_{x_0}$ is a diffeomorphism (any $x_0$ works when $N$ is Hadamard),
 the function $g \colon x \mapsto \frac{1}{2} d(x_0,x)^2$ is smooth and its gradient is given by $\grad g(x) = -\exp^{-1}_x(x_0)$.
\end{proof}

It follows, of course, that an equivariant map $\cG \to N$ is harmonic if and only if its tension field is zero,
and we obtain a characterization of discrete harmonic maps:

\begin{proposition} \label{prop:HarmonicCenterOfMassGraph2}
Let $\cG$ be an edge-weighted triangulated graph and let $N$ be a Hadamard manifold. Then
$f \in \Map_{\text{eq}}(\cG, N)$ is a harmonic map if and only if $f(x)$ is a center of mass of the weighted system of points $\{(f(y), \omega_{xy})\}_{y \sim x}$ in $N$ for every $x \in \cG^{(0)}$.
\end{proposition}

\bigskip
We conclude this section by looping back to \autoref{subsec:Meshes} and the approximation problem. The point is that when 
$S$ is equipped with a hyperbolic structure (or more generally any nonpositively curved metric), a mesh
in the sense of \autoref{def:Mesh} induces a biweighted graph structure:

\begin{definition} \label{def:MeshToGraph}
Let $\rho_\tL \colon \pi_1S \to \Isom^+(\bH^2)$ be a Fuchsian representation and let $\cM$ be an invariant mesh (cf.~\autoref{def:Mesh}).
The biweighted graph underlying $\cM$ is the biweighted graph $\cG$ such that:
\begin{itemize}
 \item $\cG$ is the abstract graph underlying $\cM$ (which is evidently $\tilde{S}$-triangulated).
 \item The edge weights are the $\omega_e$ as in \autoref{def:DiscreteEnergyMesh}.
 \item The vertex weights are given by, for every vertex $x$:
\begin{equation}
 \mu(x) = \frac{1}{3} \sum_T \Area(T) 
\end{equation}
 where the sum is taken over all triangles incident to the vertex $x$.
\end{itemize}
\end{definition}
Clearly, any discrete equivariant map along $\cM$ from $\bH^2$ to a Riemannian manifold $N$ induces an equivariant map $\cG \to N$,
and the energy $E_\cM$ agrees with the energy $E_\cG$. Of course, the systems of weights are chosen so that
the discrete energy functional $E_\cG$ 
approximates the smooth energy functional \eqref{eq:EnergyFunctionalSmoothEquivariant},
and the Riemannian structure of $\Map_\text{eq}(\cG, N)$ 
approximates the $\upL^2$ Riemannian metric on $\cC^\infty(M,N)$ (or $\upL^2(M,N)$),
with finer approximation when one takes finer meshes. The analysis of this phenomenon is treated in \cite{Gaster-Loustau-Monsaingeon2}.

\begin{remark}
 With this construction in mind, biweighted triangulated graphs can be roughly thought of as follows: the edge weights are a discrete record of the conformal structure of $S$, and the vertex weights, the area form. Note that the metric structure can be recovered from both, a phenomenon specific to dimension 2. 
This is reminiscent of \cite{MR3375525}-- though distinct-- in which two graphs with edge weights are considered \emph{conformally equivalent} if there is a function of the vertices that scales one set of weights to another.
 \end{remark}

%

%% file: StrongConvexity.tex
\section{Strong convexity of the energy}
\label{sec:StrongConvexity}

In this section we study the convexity of the discrete energy functional $E_\cG : \Map_{\text{eq}}(\cG,\bH^2) \to \bR$ introduced in the previous section (\autoref{def:EnergyFunctionalGraph}).
In \autoref{subsec:ConvexityRiemannian} we recall basics about convexity and strong convexity in Riemannian manifolds.
In \autoref{subsec:ConvexityEnergyExistenceHarmonic} we review the convexity of the energy functional for nonpositively curved target spaces.
Next we turn to proving the strong convexity of the discrete energy when the target space is $\bH^2$ with a Fuchsian representation: 
we first perform some preliminary computations in the hyperbolic plane in \autoref{subsec:H2computations}, and then prove the main theorem in \autoref{subsec:StrongConvexityH2}. In \autoref{subsec:StrongConvexityGeneral} we extend this result to Hadamard spaces with negative curvature.

\subsection{Convexity in Riemannian manifolds} \label{subsec:ConvexityRiemannian}

The classical notion of convexity in Euclidean vector spaces naturally extends to the Riemannian setting---as Udri\c{s}te puts it \cite[Chapter 1]{MR1326607}, 
\emph{Riemannian geometry is the natural frame for convexity}.

We first give a definition for metric spaces. Recall that geodesics in a metric space $(M,d)$ are harmonic maps from intervals
of the real line; more concretely, a curve $\gamma \colon I \subseteq \bR \to M$ in $(M,d)$ is a geodesic if and only if $d(\gamma(t_1), \gamma(t_2)) = v |t_2 - t_1|$ for any sufficiently close $t_1, t_2 \in I$,
where $v$ is a positive constant. A real-valued function on $M$ is then called (geodesically) convex when it is convex along geodesics. More precisely:

\begin{definition} \label{def:ConvexFunction}
Let $(M,d)$ be a metric space. A function $f \colon M \to \bR$ is \emph{convex} if,
for every geodesic $\gamma \colon [a, b] \to M$ and for all $t \in [0,1]$:
\begin{equation}
 f(\gamma((1-t)a + tb)) \leqslant  (1-t)f(\gamma(a)) + t f(\gamma(b))~.
\end{equation}
When the inequality is strict for all $t \in (a,b)$, $f$ is called \emph{strictly convex}. Furthermore $f$ is called \emph{$\alpha$-strongly convex}, where $\alpha >0$, if:
\begin{equation}
 f(\gamma((1-t)a + tb)) \leqslant  (1-t)f(\gamma(a)) + t f(\gamma(b)) - \alpha \frac{t(1-t)}{2} l(\gamma)^2
\end{equation}
where $l(\gamma)$ is the length of $\gamma$. 
The largest such $\alpha$ is called the \emph{modulus of strong convexity} of $f$.
\end{definition}

When $M = (M,g)$ is a Riemannian manifold and $f$ is $\cC^2$, one can quickly characterize convex functions in terms of the positivity of their Hessian as a quadratic form.
Recall that the Hessian of a $\cC^2$ function $f \colon M \to \bR$ is the symmetric $2$-covariant tensor field on $M$ defined by $\Hess(f) = \nabla (\upd f)$.
\begin{proposition} \label{prop:CharacConvexFunction}
Let $f \colon M \to \bR$ be a $\cC^2$ function on a Riemannian manifold $(M,g)$. Then:
\begin{itemize}
 \item $f$ is convex if and only if it has positive semidefinite Hessian everywhere.
 \item $f$ is strictly convex if 
 it has positive definite Hessian everywhere.
 \item $f$ is $\alpha$-strongly convex if and only if it has $\alpha$-coercive Hessian everywhere: 
 \begin{equation}
  \forall v \in\upT M \quad  \Hess(f)(v,v) \geqslant \alpha \Vert v \Vert^2
 \end{equation}
\end{itemize}
\end{proposition}

Convex functions enjoy several attractive properties. Among them, we highlight the straightforward fact that any sublevel set of a convex function is totally convex
(\ie{} it contains any geodesic whose endpoints belong to it). \autoref{def:ConvexFunction} and \autoref{prop:CharacConvexFunction}
work when $M$ is an infinite-dimensional Riemannian manifold (\eg{} $\cC^\infty(M,N)$ as in \autoref{subsec:EnergyFunctionalHarmonicMaps}),
however note that a convex function is not necessarily continuous in that case, whereas it is always locally Lipschitz in finite dimension.
We refer to \cite[Chap. 3]{MR1326607} for convex functions on finite-dimensional Riemannian manifolds. 

\subsection{Convexity of the energy functional}
\label{subsec:ConvexityEnergyExistenceHarmonic}

We review the convexity of the energy functional when the target is nonpositively curved, whether in the Riemannian sense or in the sense of Alexandrov, and we also address the possibility of strict or strong convexity in these settings.

\begin{remark}
While strict convexity of the energy is a clear-cut way to prove uniqueness of harmonic maps and strong convexity their existence, neither are necessary. The existence and uniqueness of harmonic maps has been properly characterized both in the smooth case and in more general spaces: see \autoref{subsec:SmoothHeatFlow} and \autoref{subsec:EquivariantHarmonicMaps}.
\end{remark}

\subsubsection*{Convexity of the energy in the smooth setting}

The second variation of the energy functional in the smooth context was first calculated by Eells-Sampson \cite{MR0164306} (cf.~\autoref{prop:SecondVariationalFormulaSmooth}). 
The next proposition follows immediately from \eqref{eq:HessianEnergySmooth}:
\begin{proposition} \label{prop:SmoothEnergyHessianInequality}
Let $M$ be and $N$ be smooth Riemannian manifolds. If $N$ has nonpositive sectional curvature, then the Hessian of the energy functional satisfies:
\begin{equation} \label{eq:HessianEnergySmoothNegativeCurvature}
 \forall V\in \Gamma(f^*\upT N) \quad \Hess(E)\evalat{f}(V, V) \geqslant \int_M \Vert \nabla V \Vert^2 \, \upd v_g~.
\end{equation}
\end{proposition}
Recall that the Hessian of the energy is taken with respect to the $L^2$ Riemannian structure on the infinite dimensional manifold $\cC^\infty(M,N)$.

In particular, \eqref{eq:HessianEnergySmoothNegativeCurvature} makes it clear that the energy functional is convex. It is tempting to try and get more
out of \eqref{eq:HessianEnergySmoothNegativeCurvature}: is $E$ strictly convex? Is it strongly convex? 
Neither can be true without some obvious restrictions: if $f$ maps into a flat (a totally geodesic submanifold of zero sectional curvature), then the energy is constant along the path
that consists in translating $f$ along some constant vector field on the flat. 
Even when $N$ has negative sectional curvature, whence
it has no flats of dimension $>1$, this issue remains for constant maps and maps into a curve. 

However, one can restrict to a connected component 
of $\cC^\infty(M,N)$ that does not contain such maps, and there the question becomes interesting. 
For example when $M$ is compact and $\dim N = \dim M > 1$, the degree of maps is an invariant on the components of $\cC^\infty(M,N)$,
and any component of nonzero degree contains only surjective map. 
When the target is negatively curved,  \eqref{eq:HessianEnergySmooth} does guarantee strict 
convexity:

\begin{proposition} \label{prop:StrictConvexitySmoothEnergy}
Let $M$ be a Riemannian manifold, let $N$ be a Riemannian manifold of negative sectional curvature. Then the energy functional is strictly convex on 
any connected component of $\cC^\infty(M,N)$ that does not contain any map of rank everywhere $\leqslant 1$.
\end{proposition}

\begin{proof}
Let $(E_i)$ be a local orthonormal frame in $M$. The integrand for the Hessian of the energy functional $\eqref{eq:HessianEnergySmooth}$ is:
\begin{equation}
 \Vert \nabla V \Vert^2 - \sum_{i=1}^n \left\langle R^N(V, \upd f(E_i)) \upd f(E_i), V \right\rangle
\end{equation}
Each term $\left\langle R^N(V, \upd f(E_i)) \upd f(E_i), V \right\rangle$ is nonpositive, and is nonzero unless 
$V$ and $\upd f(E_i)$ are collinear. Indeed, when $V$ and $\upd f(E_i)$ are not collinear:
\begin{equation}
\left\langle R^N(V, \upd f(E_i)) \upd f(E_i), V \right\rangle = K^N(V, \upd f(E_i)) \left(\Vert V \Vert^2 \, \Vert \upd f(E_i) \Vert^2 - \langle V, \upd f(E_i)\rangle^2\right) < 0
\end{equation}
where $K^N(V, \upd f(E_i))$ denotes the sectional curvature of the plane spanned by $V$ and $\upd f(E_i)$. 
If $\Hess(E)\evalat{f}(V, V)$ vanishes, then the integrand must vanish everywhere, so that (1) $\nabla V = 0$ everywhere, and (2) $\upd f(E_i)$ and $V$ must be collinear for every $i$. From (1) it follows that $V$ has constant length, 
and from (2) and the fact
that $V_x \neq 0$ it follows that $\upd_x f$ maps into $\operatorname{span}(V_x)$ for every $x \in M$. In particular, $f$ has rank $\leqslant 1$ everywhere.
\end{proof}

As far as the authors are aware, no sufficient conditions for strong convexity of the energy functional are known in the smooth setting. 
We believe that a quantitative refinement of the previous proof combined with a Poincar\'e-type inequality should guarantee:

\begin{conjecture} \label{conj:StrongConvexitySmooth}
Strong convexity holds in the setting of \autoref{prop:StrictConvexitySmoothEnergy}.
\end{conjecture}

\subsubsection*{Convexity of the energy for more general spaces}

Defining the energy \emph{à la } Jost as in \autoref{subsec:JostEnergy}, it is straightforward 
that the energy functional is convex when the target space is negatively curved in a suitable sense.
Indeed, let $(M, \mu)$ be a measure space and let $(N,d)$ be a \emph{Hadamard metric space}, \ie{} a complete $\CAT(0)$ metric space. Recall that a $\CAT(0)$ space is a geodesic metric space where any geodesic triangle $T$ is `thinner' than the triangle $T'$ with same side lengths in the Euclidean plane---more precisely, the comparison map $T \to T'$ is distance nonincreasing. In a Hadamard space the distance squared function
\begin{equation}
 d^2 \colon N \times N \to \bR
\end{equation}
is convex (see \cite{MR1744486} for details). It follows easily that for any choice of nonnegative symmetric kernel $\eta_r$ (cf \autoref{subsec:JostEnergy}), the energy functional $E_r$ is convex on $\upL^2(M,N)$. Furthermore if the energy functional $E$ on $\upL^2(M,N)$ is obtained as a $\Gamma$-limit of $E_r$,
then it must also be convex \cite[Thm 11.1]{MR1201152}. In particular, this applies to our energy functional $E_\cG$ by way of \autoref{prop:JostEnergyGraph}:

\begin{proposition} \label{prop:EnergyGraphConvex}
Let $\cG$ be any $\tilde{S}$-triangulated biweighted graph (\autoref{def:BiweightedGraph}) and let $N$ be a Hadamard metric space.
The energy functional $E_\cG \colon \Map_{\text{eq}}(\cG, N) \to \bR$ (\autoref{def:EnergyFunctionalGraph}) is convex.
\end{proposition}

We stress that the convexity is relative to a metric structure on $\Map_{\text{eq}}(\cG, N)$ which depends on a system of vertex weights
(see \autoref{def:VertexWeightedGraph}), but the fact that the energy is convex
(respectively strictly or strongly convex) does not depend on the choice of such vertex weights. 

We will examine conditions that ensure $E_\cG$ is strongly convex, first for $N = \bH^2$ (\autoref{thm:StrongConvexityEnergyGraphH2}), then in Hadamard manifolds with negative curvature (\autoref{thm:StrongConvexityEnergyGraphGeneral}).

We highlight some important context: Korevaar-Schoen 
obtained yet another form of convexity of the energy when the domain $M$ is Riemannian. Their energy functional $E$, which coincides with Jost's for suitable choices \cite{MR2348841}, satisfies the convexity inequality
\begin{equation} \label{eq:KorevaarSchoenConvexity}
E(f_t) \leqslant (1-t) E(f_0) + tE(f_1) - t(1-t) \int_M \Vert \nabla d(f_0,f_1) \Vert^2~,
\end{equation}
where $(f_t) \in \upL^2(M,N)$ is a geodesic, \ie{} $f_t(x)$ is a geodesic in $N$ for all $x\in M$.
This is a weaker analog of \autoref{prop:SmoothEnergyHessianInequality}. 
It is again tempting to 
investigate strong convexity when $N$ has negative curvature bounded away from $0$ and $f$ does not have rank everywhere $\leqslant 1$, but 
work remains to be done. 

Mese \cite{MR1938491} claims without proof both an improvement of the convexity statement \eqref{eq:KorevaarSchoenConvexity} and
the strict convexity of the energy functional at maps of rank $\leqslant 1$ as in \autoref{prop:StrictConvexitySmoothEnergy}, 
however neither of these claims are explained as far as we can tell. We also note that there is a mistake in the curvature term of \cite[eq.~(1)]{MR1938491}. 
(In fairness, this lack of explanation is probably due to Mese's focus on the task of extending the uniqueness of Korevaar-Schoen to the setting where $\partial M=\emptyset$.)

Although we prove strong convexity for biweighted graph domains under appropriate restrictions, we suspect that a more general version of this theorem is true, namely an analog of \autoref{conj:StrongConvexitySmooth} for singular spaces. 
In fact, one can 
further explore extensions to the equivariant setting, with a suitable condition on the target representation strengthening reductivity.

\subsection{Convexity estimates in the hyperbolic plane}
\label{subsec:H2computations}

In order to study the second variation of the discrete energy for $\bH^2$-valued equivariant maps, 
we first need some convexity estimates in the hyperbolic plane.
The strategy in \autoref{subsec:StrongConvexityH2} will be to reach a contradiction under the assumption that the second variation of the energy is too small;
here we derive necessary consequences of a small second variation of the energy in the the elementary cases consisting of two and three vertices.
We start with a formula for quadrilaterals in $\bH^2$.

\begin{center}
\begin{figure}[!ht]
\begin{minipage}{.47\textwidth}
\begin{center}
\begin{tikzpicture}[scale=1.0]
\draw (-3,0) node[left] {$A$} node {$\bullet$};
\draw (3,0) node[right] {$B$} node {$\bullet$};
\draw (1,3) node[above] {$C$} node {$\bullet$};
\draw (-2,4) node[above] {$D$} node {$\bullet$};
\draw [thick] (-3, 0) -- (3, 0) ;
\draw [thick] (3, 0) -- (1, 3) ;
\draw [thick] (1, 3) -- (-2, 4) ;
\draw [thick] (-2, 4) -- (-3, 0) ;
\draw [thick, ->, >=latex, blue] (-2,0) arc (0:atan(4):1) ;
\draw [blue] ({-3 + cos(atan(4)/2)} , {0 + sin(atan(4)/2)} ) node[above right] {$\alpha$} ;

\draw [thick, ->, >=latex, blue] ({3 + cos(pi r - atan(3/2))}, {0 + sin(pi r - atan(3/2))}) arc ({pi r - atan(3/2)}:{pi r}:1) ;
\draw [blue] ({3 + cos(pi r - atan(3/2)/2))} , {0 + sin(pi r - atan(3/2)/2)} ) node[above left] {$\beta$} ;
\end{tikzpicture} \caption{} \label{fig:quadrilateral1}
\end{center}
\end{minipage} \hspace{5mm}
\begin{minipage}{.47\textwidth}
\begin{center}
\begin{tikzpicture}[scale=1.0]
\draw (-3,0) node[left] {$A$} node {$\bullet$};
\draw (3,0) node[right] {$B$} node {$\bullet$};
\draw (1,3) node[above] {$C$} node {$\bullet$};
\draw (-2,4) node[above] {$D$} node {$\bullet$};
\draw [thick] (-3, 0) -- (3, 0) ;
\draw [thick] (3, 0) -- (1, 3) ;
\draw [thick] (1, 3) -- (-2, 4) ;
\draw [thick] (-2, 4) -- (-3, 0) ;
\draw (-3, 0) -- (1, 3) ;

\draw [thick, ->, >=latex, blue] (-2,0) arc (0:atan(3/4):1) ;
\draw [blue] ({-3 + cos(atan(3/4)/2)} , {0 + sin(atan(3/4)/2)} ) node[above right] {$\alpha_1$} ;

\draw [thick, ->, >=latex, blue] ({-3 + cos(atan(3/4)},{0 + sin(atan(3/4)}) arc (atan(3/4):atan(4):1) ;
\draw [blue] ({-3 + cos(atan(3/4)/2 + atan(4)/2)} , {0 + sin(atan(3/4)/2 + atan(4)/2)} ) node[above  right] {$\alpha_2$} ;

\draw [thick, ->, >=latex, blue] ({3 + cos(pi r - atan(3/2))}, {0 + sin(pi r - atan(3/2))}) arc ({pi r - atan(3/2)}:{pi r}:1) ;
\draw [blue] ({3 + cos(pi r - atan(3/2)/2))} , {0 + sin(pi r - atan(3/2)/2)} ) node[above left] {$\beta$} ;
\end{tikzpicture} \caption{} \label{fig:quadrilateral2}
\end{center}
\end{minipage}
\end{figure}
\end{center}

\begin{proposition} \label{prop:HyperbolicQuadrilateral}
 Let $A$, $B$, $C$, $D$ be four points in the hyperbolic plane. 
 Let $\alpha$ and $\beta$ denote the oriented angles as shown in \autoref{fig:quadrilateral1}. 
 Then:
\begin{equation}
\begin{aligned}
 \cosh (DC) &= \cosh (AB) \big[\cosh (DA) \cosh (BC) + \sinh (DA) \sinh (BC) \cos \alpha \cos \beta\big] \\
          &\quad - \sinh (AB) \big[\cosh (DA) \sinh (BC) \cos \beta + \sinh (DA) \cosh (BC) \cos \alpha\big] \\
          &\quad - \sinh (DA) \sinh (BC) \sin \alpha \sin \beta~.
\end{aligned}
\end{equation}
\end{proposition}
\begin{remark}
This equation holds without restriction on $\alpha$ and $\beta$; they may be negative or obtuse.
\end{remark}

\begin{proof}

Referring to \autoref{fig:quadrilateral2}, the hyperbolic law of cosines implies:
\begin{equation} \label{eq:coshACD}
 \cosh(DC) = \cosh(DA)\cosh(AC) - \sinh(DA) \sinh(AC) \cos(\alpha_2)~.
\end{equation}

The hyperbolic laws of sines and cosines in the triangle $ABC$ give
\begin{align} 
  \cosh(AC) & = \cosh(AB)\cosh(BC) - \sinh(AB) \sinh(BC) \cos(\beta)~, \ \text{ and } \label{eq:coshABC} \\
  \sinh(AC)  \cos(\alpha_2) &= \sinh(AC) \cos(\alpha_1 - \alpha)\\
  &= \sinh(AC) \cos (\alpha_1 ) \cos(\alpha) + \sinh(AC) \sin(\alpha_1) \sin(\alpha) \\
  &= \sinh(AC) \cos (\alpha_1 ) \cos(\alpha) + \sinh(BC) \sin(\beta) \sin (\alpha)~.
  \label{eq:alphasumh}
\end{align}
Moreover, it is a 
consequence of the two forms of the hyperbolic law of cosines (see \eg \cite[p.~82]{MR2249478}) in the triangle $ABC$ that we have
\begin{equation} \label{eq:sinhACcosalpha1}
 \sinh(AC) \cos(\alpha_1) = \sinh(AB) \cosh(BC) - \sinh(BC) \cosh(AB) \cos(\beta)~.
\end{equation}

Equation \eqref{eq:sinhACcosalpha1} allows us to rewrite equation \eqref{eq:alphasumh} as:
\begin{equation} \label{eq:ACcosalphatwoh}
 \begin{split}
  \sinh(AC)  \cos(\alpha_2) &= \big(\sinh(AB) \cosh(BC) - \sinh(BC) \cosh(AB) \cos(\beta)\big) \cos(\alpha) \\
  &\quad + \sinh(BC) \sin(\beta) \sin(\alpha)~.
 \end{split}
\end{equation}
Together \eqref{eq:coshABC} and \eqref{eq:ACcosalphatwoh} and \eqref{eq:coshACD} imply the desired equation.
\end{proof}

Next we study the convexity of the energy for two points, which amounts to analyzing the second variation of the half-distance squared function
$ \frac{d^2}{2} \colon \bH^2 \times \bH^2 \to \bR$. We perform this computation in two stages: first we study instead the function $(\cosh d) -1 \colon \bH^2 \times \bH^2 \to \bR$, as it is better suited to computations, and then we relate the second variation of the two functions.

\begin{proposition} \label{prop:DerivativesF}
 Let $A$ and $B$ be two points in the hyperbolic plane at distance $D$. Let $\vec{u}$ and $\vec{v}$ be tangent vectors at $A$ and $B$ respectively. Let $A_t = \exp_A (t \vec{u})$ and $B_t = \exp_B (t \vec{v})$ for $t \in \bR$, and consider the function $F_{AB}(t) = \cosh\left(d(A_t, B_t)\right) - 1$.
 Then:
 \begin{align}
\frac{\upd}{\upd t} \evalat{t=0} F_{AB}(t) & = -\sinh(D) \big( \Vert \vec{u} \Vert \cos \alpha - \Vert \vec{v} \Vert \cos \beta \big) \\
 \frac{\upd^2}{\upd t^2}\evalat{t=0} F_{AB}(t) & = 
 \cosh (D) \big(\Vert \vec{u} \Vert^2 + \Vert \vec{v} \Vert ^2 - 2 \Vert \vec{u} \Vert \Vert \vec{v} \Vert \cos \alpha \cos \beta\big) -2 \Vert \vec{u} \Vert \Vert \vec{v} \Vert \sin \alpha \sin \beta~.
 \end{align}
 where $\alpha$ (resp. $\beta$) is the oriented angle between the geodesic $AB$
 and the vector $\vec{u}$ (resp. $\vec{v}$).
\end{proposition}

\begin{proof}
Consider the quadrilateral given by the four points $A$, $B$, $C = B_t$, $D = A_t$. 
Note that the angle $\beta$ here
corresponds to the angle $\pi - \beta$
of \autoref{prop:HyperbolicQuadrilateral}. By direct application of \autoref{prop:HyperbolicQuadrilateral},
\begin{equation}
\begin{aligned}
 1 + F_{AB}(t) &= \cosh (D) \big[\cosh (t \Vert \vec{u} \Vert) \cosh (t \Vert \vec{v} \Vert) - \sinh (t \Vert \vec{u} \Vert) \sinh (t \Vert \vec{v} \Vert) \cos \alpha \cos \beta\big] \\
          & \quad - \sinh (D) \big[- \cosh (t \Vert \vec{u} \Vert) \sinh (t \Vert \vec{v} \Vert) \cos \beta + \sinh (t \Vert \vec{u} \Vert) \cosh (t \Vert \vec{v} \Vert) \cos \alpha\big] \\
          &\quad - \sinh (t \Vert \vec{u} \Vert) \sinh (t \Vert \vec{v} \Vert) \sin \alpha \sin \beta~.
\end{aligned}
\end{equation}
 The result follows immediately by taking the first and 
 second derivatives at $t=0$.
\end{proof}

\begin{proposition} \label{prop:DerivativesE}
We keep the same setup as \autoref{prop:DerivativesF}, and let $E_{AB}(t) = \frac{1}{2} d(A_t, B_t)^2$.
 Then:
 \begin{align}
 \frac{\upd^2}{\upd t^2}\evalat{t=0} E_{AB}(t) & = 
 a + b \; D \tanh (D/2) + c  \left(D \coth D -D\tanh(D/2) \right)  ~,
 \end{align}
where $a$, $b$, and $c \geqslant 0$ are given by
\begin{align}
a & = \left( \Vert \vec{u} \Vert \cos \alpha - \Vert\vec{v} \Vert\cos \beta \right)^2~, \\
b & =   \Vert \vec{u} \Vert^2 \sin^2 \alpha + \Vert\vec{v} \Vert^2 \sin^2 \beta ~, \text{ and } \\
c &=   \left( \Vert \vec{u} \Vert \sin \alpha - \Vert\vec{v} \Vert\sin \beta \right)^2~.
\end{align}
\end{proposition}

\begin{proof}
For ease in notation, we leave the subscripts $AB$ from $E_{AB}$ and $F_{AB}$ in what follows.
We have $E(t) = \phi \circ F(t)$, where $\phi(x) = \frac 12 \left( \arcosh (1+x) \right)^2$. It is straightforward to check that 
\begin{align}
\phi'(\cosh x -1) = \frac{x}{\sinh x}~, \ \ \ \text{ and } \ \ \ \phi''(\cosh x-1) = \frac{\sinh x- x\cosh x}{\sinh^3 x}~.
\end{align}
Since we have $E''(0) = \phi''(F(0)) \left( F'(0) \right)^2 + \phi'(F(0)) F''(0)$, using \autoref{prop:DerivativesF} we find that
\begin{align}
E''(0)  & =  \frac{\sinh D - D\cosh D}{\sinh^3 D} \cdot \sinh^2D  \left(\Vert\vec{u}\Vert\cos \alpha - \Vert\vec{v}\Vert \cos \beta\right)^2 \\
& \quad + \frac{D}{\sinh D} \left( \cosh D \left( \Vert\vec{u}\Vert^2+ \Vert\vec{v}\Vert^2 -2\Vert\vec{u}\Vert\Vert\vec{v}\Vert\cos \alpha\cos \beta\right) - 2\Vert\vec{u}\Vert\Vert\vec{v}\Vert\sin\alpha\sin\beta\right) \\
& =  \left(\Vert\vec{u}\Vert\cos \alpha - \Vert\vec{v}\Vert \cos \beta\right)^2 + D\coth D \left(  -\left(\Vert\vec{u}\Vert\cos \alpha - \Vert\vec{v}\Vert \cos \beta\right)^2 +  \right. \\
& \quad \left. \left( \Vert\vec{u}\Vert^2+ \Vert\vec{v}\Vert^2 -2\Vert\vec{u}\Vert\Vert\vec{v}\Vert\cos \alpha\cos \beta\right) \right)  - D \csch D \left( 2\Vert\vec{u}\Vert\Vert\vec{v}\Vert\sin\alpha\sin\beta\right) \\
& = a + b \; D\coth D 
+(c-b)\;  D \csch D ~.
\end{align}
To finish, note that $\coth D - \csch D = \tanh(D/2)$.
\end{proof}

%
%

The following quantitative control is at the core of 
strong convexity for $E_\cG$:

\begin{proposition} \label{prop:SecondDerivativeHyperbolicEstimates}
Let $E_{AB}(t)$ be the function as in \autoref{prop:DerivativesE}. 
We have 
\begin{equation} 
\frac{\upd^2}{\upd t^2}\evalat{t=0} E_{AB}(t) \geqslant \Vert \vec{u} - P_{[BA]}\vec{v} \Vert^2~,
\end{equation}
where $P_{[BA]}\vec{v}$ denotes the parallel transport of $\vec{v}$ along the geodesic segment $BA$.
\end{proposition}


\begin{proof}
 By rewriting $E_{AB}''(0)$ using \autoref{prop:DerivativesE}, and noting that $2b\geqslant c$ by the Cauchy-Schwarz inequality, we find
\begin{align}
\label{eq:abcEAB}
E_{AB}''(0) & = a + b \; D  \tanh \frac D2 + c \left(D\coth D - D\tanh \frac D2\right) \\ 
& \geqslant a + c\left( D\coth D - \frac D2\tanh \frac D2 \right) 
 = a + c \cdot \frac D2 \coth \frac D2~.
\end{align}
Because $x\coth x\geqslant1 $, we find that $E_{AB}''(0) \geqslant a+c$. 
%
%
The proof is completed by checking that the quantity $a+c$ is precisely $\Vert  \vec{u} - P_{[BA]}\vec{v} \Vert^2$:
\begin{align*}
a +c & = \left( \Vert \vec{u} \Vert \cos \alpha - \Vert\vec{v} \Vert\cos \beta \right)^2 +  \left( \Vert \vec{u} \Vert \sin \alpha - \Vert\vec{v} \Vert\sin \beta \right)^2~ \\
&= \Vert \vec{u} \Vert^2 - 2\Vert \vec{u} \Vert \Vert \vec{v} \Vert \cos(\alpha-\beta) + \Vert \vec{v} \Vert^2 \\
&= \langle \vec{u} ,\vec{u} \rangle  - 2\langle \vec{u} , P_{[BA]}\vec{v} \rangle + \langle \vec{v},\vec{v} \rangle \\
& = \Vert  \vec{u} - P_{[BA]}\vec{v} \Vert^2~. \qedhere
\end{align*}
\end{proof}

\subsection{\texorpdfstring{Strong convexity of the discrete energy in $\bH^2$}{Strong convexity of the discrete energy in H2}}
\label{subsec:StrongConvexityH2}

Let $\cG$ be any $\tilde{S}$-triangulated biweighted graph (\autoref{def:BiweightedGraph}) and let $\rho_\tR \colon \pi_1S \to \Isom(\bH^2)$ be a Fuchsian representation.
We are ready to prove strong convexity of the discrete energy functional $E_\cG \colon \Map_{\text{eq}}(\cG, \bH^2) \to \bR$ introduced in \autoref{def:EnergyFunctionalGraph}.

Choose once and for all a fundamental domain for the action of $\pi_1S$ on $\cG^{(0)}$, consisting of vertices $\{p_1, \dots, p_n\} \subseteq \cG^{(0)}$. 
Recall that $\cG$ is equipped with vertex and edge weights; 
these are completely determined by the weights $\mu_i = \mu(p_i)$ and $\omega_{ij}=\omega(e_{p_i p_j})$
for $i,j \in \{1, \dots, n\}$.

Fix an equivariant map $f\in\Map_{\text{eq}}(\cG,\bH^2)$, recorded by the tuple $(x_1,\dots, x_n) \in (\bH^2)^n$
where $x_i \coloneqq f(p_i)$.
Also fix a tangent vector $ \vec{v} \in \upT_f\Map_{\text{eq}}(\cG, \bH^2)$, given by $\vec{v} = (\vec{v}_1,\dots,\vec{v}_n)$ where $v_i \in \upT_{x_i}\bH^2$
as in \eqref{eq:MapGammaTangentSpace2}. We assume that $\vec{v}$ is a unit tangent vector: by \autoref{def:L2RiemannianMetricGraph} this means
$\sum_i \mu_i \Vert \vec{v_i} \Vert^2 = 1$.

We want to compute the second derivative at $t=0$ of $E_\cG(t) \coloneqq E_\cG \circ \exp_{f}(t \vec{v})$. Let us denote 
\begin{equation}
E_{ij}(t) \coloneqq  \frac 12  d\;(\exp_{x_i}(t\vec{v_i}) , \exp_{x_j}(t\vec{v_j}))^2
\end{equation} 
for each edge $e_{ij}$ between points $p_i$ and $p_j$.

First we observe that if the second variation of $E_\cG$ is small, then \autoref{prop:SecondDerivativeHyperbolicEstimates} implies that the tangent vectors $\vec{v_i}$ all have approximately the same length. Precisely,

\begin{lemma}
\label{lem:meshVectorsEqualLength}
For all $i$ we have 
\begin{equation}
\left| \; \Vert \vec{v_i} \Vert - \frac{1}{\sqrt A} \; \right| < 3 \sqrt{E_\cG''(0) \ \frac{ D}{\Omega}}~,
\end{equation}
where $D$ is the diameter of the quotient graph $\cG/\pi_1S$, $A=\sum_i \mu_i$, and $\Omega=\min \omega_{ij}$.
\end{lemma}

\begin{proof}
Let $\varepsilon_{ij} = E_{ij}''(0)$ for each edge $e_{ij}\in\cE$, so that 
$
E_\cG''(0) = \sum_{e_{ij}\in\cE} \omega_{ij} \varepsilon_{ij}
$. 
Observe that for each edge $e_{ij}\in\cE$, one finds that 
$|\Vert\vec{v_i}\Vert - \Vert\vec{v_j}\Vert|< \sqrt{\varepsilon_{ij}}$ by \autoref{prop:SecondDerivativeHyperbolicEstimates}.
For any pair of points $p_i$ and $p_j$, choose a 
path $p_i= p_{i_0}, p_{i_1}, \dots, p_{i_r} = p_j$ 
and observe that by the triangle inequality and the Cauchy-Schwarz inequality we have:
\begin{align}
\left| \Vert\vec{v_i}\Vert - \Vert \vec{v_j} \Vert \right| & \leqslant \sum_{k=1}^r \sqrt{\varepsilon_{i_{k-1}i_k} } 
\leqslant \left( \sum_{k=1}^r \varepsilon_{i_{k-1}i_k}\right)^{1/2} \cdot \left( \sum_{k=1}^r 1\right)^{1/2}
~. 
\end{align}
Now the path may be taken with length $r\leqslant D$, and $ \Omega \sum \varepsilon_{ij} \leqslant E_\cG''(0)$, so the above inequality 
implies
\begin{equation}
\label{eq:ViVj}
\left| \; \Vert\vec{v_i}\Vert - \Vert \vec{v_j} \Vert \; \right|  <  \sqrt{\frac{E_\cG''(0) \ D}{\Omega}}~.
\end{equation}

Let $\delta=\sqrt{ E_{\cG}''(0) \ D /\Omega}$. As $| \|\vec{v_i}\| - \| \vec{v_j}\| | < \delta$ for all $i,j$,
\begin{align*}
| \|\vec{v_i}\|^2 - \|\vec{v_j}\|^2 | &= | \|\vec{v_i}\| - \|\vec{v_j}\| | \cdot | \|\vec{v_i}\| + \|\vec{v_j}\| | \\
& \le  | \|\vec{v_i}\| - \|\vec{v_j}\| | \cdot (  | \|\vec{v_i}\| - \|\vec{v_j}\| | + 2\|\vec{v_i}\|) \\
& < \delta^2 + 2 \delta \|\vec{v_i}\|~.
\end{align*}
As $\vec{v}$ is a unit tangent vector and $A = \sum_j \mu_j$, we have
\begin{align*}
1 &= \sum_j \mu_j \|\vec {v_j}\|^2 = A \|\vec{v_i}\|^2 + \sum_j \mu_j(\|\vec{v_j}\|^2 - \|\vec{v_i}\|^2) ~.
\end{align*}
Rearranging and taking absolute values we find
\[
| A \|\vec{v_i}\|^2 - 1 | \leqslant \sum_j \mu_j | \|\vec{v_i}\|^2 - \|\vec{v_j}\|^2 | < A (\delta^2 + 2 \delta \|\vec{v_i}\|)~.
\]
Dividing by $A$ we obtain
\[
\left| \|\vec{v_i}\|^2 - \frac 1A \right| < \delta^2 + 2 \delta \|\vec{v_i}\|~.
\]
This implies that we have
\begin{align*}
\|\vec{v_i}\|^2 - 2 \delta \|\vec{v_i}\| - \delta^2 &< \frac 1 A~, \text{ and} \\
\|\vec{v_i}\|^2 +2 \delta \|\vec{v_i}\| + \delta^2 &> \frac 1 A~.
\end{align*}
In other words,
\begin{align*}
(\|\vec{v_i}\| - \delta )^2 &< \frac 1 A + 2\delta^2~, \text{ and} \\
(\|\vec{v_i}\| + \delta )^2 &< \frac 1 A ~.
\end{align*}
We conclude that $\frac1{\sqrt{A}} -\delta < \|\vec{v_i}\| < \delta + \sqrt{ \frac 1A + 2\delta^2 } <\frac1 {\sqrt{A}} + (1+\sqrt2)\delta < \frac1 {\sqrt{A}} + 3\delta$, as desired.
\end{proof}

Together with \autoref{prop:SecondDerivativeHyperbolicEstimates}, this is enough to produce a lower bound for the second variation:
\begin{proposition}
\label{prop:SecondDerivativeEstimateBelow}
We have 
\begin{equation}
\label{eq:SecondDerivativeLowerBound1}
\frac{\upd^2}{\upd t^2}\evalat{t=0}
E_\cG (t) \geqslant \frac1{9A\left(1+\sqrt{\frac{D}{\Omega}}\right)^2}~,
\end{equation}
where $D$ is the diameter of the quotient graph $\cG/\pi_1S$, $A=\sum_i \mu_i$, and $\Omega = \min \omega_{ij}$ is the minimum edge weight.
\end{proposition}

\begin{remark}
\label{rem:collectingConstants}
Observe that this demonstrates that there is a constant $C>0$ independent of $\cG$ so that 
the modulus of convexity of $E_\cG$ is at least $ \frac{C \Omega}{AD}$. When $\cG$ is the biweighted graph underlying a Fuchsian representation as in \autoref{def:MeshToGraph}, the modulus of convexity is at least $\frac{C \; \Omega}{|\chi(S)|D}$.
\end{remark}

\begin{proof}

Suppose towards contradiction that $E_\cG''(0) < \frac19 A^{-1} \left(1+\sqrt{\frac{D}{\Omega}}\right)^{-2}$, i.e.~
\begin{equation}\label{eq:contradiction}
\sqrt{E_\cG''(0)} < \frac1{3\sqrt A} - \sqrt{E_\cG''(0)\cdot \frac{D}{\Omega}}~.
\end{equation}
Rearranging, we find that
\begin{equation}
\frac{\sqrt{E_\cG''(0)}}{\frac1{\sqrt{A}} - 3\sqrt{E_\cG''(0)\cdot \frac{D}{\Omega}}} < \frac13~.
\end{equation}

Now consider an edge $x_i\sim x_j$ in $\cG$, let $\delta_{ij} = \| \vec{v_i} - P_{[x_jx_i]} \vec{v_j} \|$, and let $\theta_{ij}$ indicate the principal value of the angle formed between $\vec{v_i}$ and $P_{[x_jx_i]}\vec{v_j}$. Evidently, $\| \vec{v_i} \| |\theta_{ij}|$ is bounded by $\frac \pi 2$ times the distance between $ \vec{v_i} $ and $\frac{\|\vec{v_i}\|}{\|\vec{v_j}\|} P_{[x_jx_i]}\vec{v_j}$, and the latter distance is bounded by $\delta_{ij} + \left| \|\vec{v_i}\| - \|\vec{v_j}\|\right| \leqslant 2\delta_{ij}$. Therefore we have 
\begin{equation}
\| \vec{v_i} \|| \theta_{ij}| \leqslant  \pi \cdot \delta_{ij} ~.
\end{equation}
It is not hard to see that \eqref{eq:contradiction} implies $3(E_\cG''(0)\cdot D/\Omega)^{1/2} < \frac 1{\sqrt{A}}$, so by \autoref{lem:meshVectorsEqualLength} we conclude that $ \|\vec{v_i}\| \geqslant \frac1{\sqrt{A}} - 3(E_\cG''(0)\cdot D/\Omega)^{1/2}>0$. In that case we may rearrange the above to find that
\begin{equation}\label{eq:bound angle}
|\theta_{ij}| \leqslant \pi \frac{\delta_{ij}}{\|\vec{v_i}\|} \leqslant \pi \cdot \frac{\sqrt{E_\cG''(0)}}{\frac1{\sqrt{A}} - 3\sqrt{E_\cG''(0)\cdot D/\Omega}} < \frac\pi 3~,
\end{equation}
by our assumption \eqref{eq:contradiction}.

Now extend the vectors $\vec{v_i}$ at the vertices of $\cG$ over each of the edges of the graph $\cG \subset Y$ by interpolating linearly.
%
%
Consider a triangle $T$ in $\cG$ with vertices $x_i,x_j,x_k$ and boundary contour $\gamma$, oriented counterclockwise. The winding number of $\vec{V}$ with respect to $\gamma$, denoted $\omega_\gamma(\vec{V})$, is given by 
\begin{equation}
\omega_\gamma(\vec{V}) = \frac{1}{2\pi}\left( \theta_{ij}+\theta_{jk}+\theta_{ki} + \theta_i + \theta_j + \theta_k \right)~,
\end{equation} 
where $\theta_i,\theta_j,\theta_k$ are the exterior angles of $T$ at $x_i,x_j,x_k$, respectively. 
Evidently, the sum of the exterior angles is at least $\pi$, and at most $3\pi$.
Together with \eqref{eq:bound angle}, we find that $0<\omega_\gamma(\vec{V}) <2$, so it follows that $\omega_\gamma(\vec{V})=1$ for every triangle $T$. Now $\vec{V}$ can be extended over the triangles to a nonvanishing section of $f^*\upT Y$, contradicting \autoref{lem:no sections}.   
\end{proof}

\begin{lemma}
\label{lem:no sections}
Suppose that $f:S \to Y$ is a map between surfaces with $\chi(Y)\ne0$ and $\deg f \ne 0$. Then $f^*\upT Y$ has no nonvanishing sections.
\end{lemma}

\begin{proof}
Let $e$ indicate the \emph{Euler class}. We have
\begin{equation}
\label{eq:euler cl}
\langle e(f^*\upT Y),[S] \rangle = \langle f^*e(\upT Y) , [S] \rangle = \langle e(\upT Y) , f_*[S] \rangle = \deg f \cdot \chi(Y)~,
\end{equation}
which is nonzero by assumption. Therefore $e(f^*\upT Y) \ne 0$, and $f^*\upT Y$ has no nonvanishing sections.
\end{proof}

The discrete heat flow will also require an upper bound for the Hessian of $E_\cG$ (see \autoref{sec:DiscreteHeatFlow}). 
\begin{proposition}
\label{prop: second derivative estimate above}
Suppose that $E_\cG(0) \leqslant E_0$. Then  
\begin{equation}
\label{eq:SecondDerivativeUpperBound}
\frac{\upd^2}{\upd t^2}\evalat{t=0}
E_\cG (t)  \leqslant \frac{2 V W}{ U}\left( 1+ \sqrt{\frac{E_0}{\Omega}} \coth \sqrt{\frac{E_0}{\Omega}}\right) \eqqcolon \beta
\end{equation}
where $V$ is the maximum valence of vertices of $\cG$, $U= \min\{\mu_i\}$ is the minimum vertex weight, 
$\Omega=\min\{\omega_{ij}\}$ is the minimum edge weight, and $W=\max\{\omega_{ij}\}$ is the maximum edge weight.
\end{proposition}
 
\begin{proof}
It is not hard to see from \autoref{prop:DerivativesE} that we have
\begin{equation}
\label{eq:HessianUpperBound}
\frac{\upd^2}{\upd t^2}\evalat{t=0} E_{ij}(t) \leqslant 2 \left( \Vert\vec{v_i}\Vert^2 + \Vert\vec{v_j}\Vert^2\right) \left(1+ d(x_i,x_j) \coth d(x_i,x_j) \right)~.
\end{equation}
Letting $L=\max\{ d(x_i,x_j)\}$
we find 
\begin{align}
\frac{\upd^2}{\upd t^2}\evalat{t=0} E_\cG(t) &= \sum_{e_{ij}\in\cE} \omega_{ij} \frac{\upd^2}{\upd t^2}\evalat{t=0} E_{ij}(t)  
\leqslant 2W \left(1+ L \coth L\right) \sum_{e_{ij}\in\cE} \left(\Vert \vec{v_i}\Vert^2 + \Vert \vec{v_j}\Vert^2 \right) \\
& \leqslant  \frac {2W}{U} \left(1+ L \coth L\right) \sum_{e_{ij}\in\cE} \left( \mu_i \Vert \vec{v_i}\Vert^2 + \mu_j \Vert \vec{v_j}\Vert^2 \right) \\
& \leqslant  \frac{2 VW}{U} \left(1 + L \coth L\right) \sum_i \mu_i \Vert\vec{v_i}\Vert^2 =   \frac{2  VW}{U} \left(1 + L \coth L\right)~.
\end{align}
The assumption $E_\cG(0) \leqslant E_0$ implies that $\Omega L^2 \leqslant E_0$, so we are done.
\end{proof}

Together, \autoref{prop: second derivative estimate above} and \autoref{prop:SecondDerivativeEstimateBelow} imply:

\begin{theorem}\label{thm:StrongConvexityEnergyGraphH2}
 Suppose that $\cG$ is a biweighted triangulated graph, $N = \bH^2$ is the hyperbolic plane and $\rho = \rho_\tR \colon \pi_1S\to \Isom^+(\bH^2)$ is Fuchsian. Then the energy functional $E_\cG : \Map_{\text{eq}}(\cG, N) \to \bR$ is strongly convex. More precisely,
\begin{equation}
 \forall \vec{v} \quad 
   \alpha \Vert \vec{v} \Vert^2 \leqslant \Hess(E_\cG)(\vec{v},\vec{v})
\end{equation}
where $\alpha$ is given explicitly by \eqref{eq:SecondDerivativeLowerBound1}. Moreover, on the compact convex set $\{E \leqslant E_0\} \subseteq \Map_{\text{eq}}(\cG, N)$,
\begin{equation}
 \forall \vec{v} \quad \Hess(E_\cG)(\vec{v},\vec{v})
 \leqslant 
 \beta \Vert \vec{v} \Vert^2
\end{equation}
where $\beta$ is given explicitly by \eqref{eq:SecondDerivativeUpperBound}.
\end{theorem}

\subsection{More general target spaces}
\label{subsec:StrongConvexityGeneral}

Some of the steps in the proof of \autoref{thm:StrongConvexityEnergyGraphH2} actually hold in much greater generality. 
For instance, \autoref{prop:SecondDerivativeHyperbolicEstimates} holds verbatim if we replace $\bH^2$ with any
Hadamard manifold. Keeping the same setup as in \autoref{subsec:H2computations}, let $A,B\in N$, let $\vec{u}$ and $\vec{v}$ be vectors in $\upT_AN$ and $\upT_BN$, respectively, let $A_t= \exp_A(t\vec{u})$ and $B_t=\exp_B(t\vec{v})$, and let $E_{AB}(t) = \frac 12 d_N(A_t,B_t)^2$.

\begin{proposition}
\label{prop:SecondDerivativeGeneralEstimates}
Suppose that $N$ is a 
Hadamard manifold. 
Then we have 
\begin{equation} 
\frac{\upd^2}{\upd t^2}\evalat{t=0} E_{AB}(t) \geqslant \Vert \vec{u} - P_{[BA]}\vec{v} \Vert^2~.
\end{equation}
\end{proposition}

The proof below 
builds on \autoref{subsec:H2computations}. 
\begin{proof}
We start by noting that the computations for $E''_{AB}(0)$ (\autoref{prop:DerivativesF} and \autoref{prop:DerivativesE}) simplify dramatically 
when $N=\bR^2$. 
In this case we have
\begin{equation}
\label{eq:secondDerivEuclidean}
 \frac{\upd^2 }{\upd t^2} E_{AB}(t)  = \frac 12\Vert \vec{u} - \vec{v} \Vert^2 ~.
 \end{equation}
In particular, if $\vec{u}\ne \vec{v}$ then $E_{AB}(t)$ is $\frac 12 \Vert \vec{u} - \vec{v} \Vert^2$-strongly convex.

Now we turn to the general case. Let $r,s\in \bR$, and consider the quadrilateral through $A_{r}$, $B_{r}$, $B_{s}$, and $A_{s}$. Because Hadamard manifolds are $\CAT(0)$, this quadrilateral has a comparison quadrilateral in $\bR^2$ with vertices $ A_{r}'$, $ B_{r}'$, $ B_{s}'$, $ A_{s}'$. For any $t\in\bR$, as a consequence of \eqref{eq:secondDerivEuclidean} we have 
\begin{align}
\label{eq:comparisonQuad}
d(A_{(1-t)r+ts},B_{(1-t)r+ts})^2 & \leqslant d(A_{(1-t)r+ts}',B_{(1-t)r+ts}')^2 \\
& \leqslant (1-t) d(A_{r}',B_{r}')^2  + t d(A_{s}',B_{s}')^2 - \Vert\vec{u}' - \vec{v}'\Vert^2 \frac{ t(1-t)}{4} |r-s|^2 \\
& \leqslant (1-t) d(A_{r},B_{r})^2  + t d(A_{s},B_{s})^2 - \Vert\vec{u}' - \vec{v}'\Vert^2 \frac{ t(1-t)}{4} |r-s|^2~,
\end{align}
where $\vec{u}'=A_{s}'-A_{r}'$ and $\vec{v}'=B_{s}'-B_{r}'$ in $\bR^2$, respectively.
In terms of $E_{AB}$ we have
\begin{equation}
\label{eq:EABinequality1}
E_{AB}((1-t)r+ts)  \leqslant (1-t) E_{AB}(r) + t E_{AB}(s) - \Vert\vec{u}' - \vec{v}'\Vert^2 \frac{ t(1-t)}{4}|r-s|^2  ~.
\end{equation}

This is almost an equivalent formulation of the strong convexity of $E_{AB}$ (see \autoref{def:ConvexFunction}), however the term $ \Vert\vec{u}' - \vec{v}'\Vert^2 $ depends on both $r$ and $s$. In fact, this detail is essential, as $E_{AB}$ may fail to be strongly convex. 
For $\delta>0$, taking $r=s=\delta$ and $t=\frac 12$, one finds
\begin{equation}
E_{AB}(0)  = E_{AB}\left(\frac 12 \delta + \frac 12(-\delta) \right) \leqslant \frac 12 E_{AB}(\delta) + \frac 12 E_{AB}(-\delta) -  \Vert\vec{u}'-\vec{v}'\Vert^2 
\frac{\delta^2}{2}~.
\end{equation}
(We stress the dependence of $ \Vert\vec{u}'-\vec{v}'\Vert^2$ on $\delta$.)
Rearranging we find that 
\begin{equation}
\label{eq:EABinequality2}
 \Vert\vec{u}'-\vec{v}'\Vert^2 \leqslant \frac{ E_{AB}(\delta) - 2E_{AB}(0) + E_{AB}(-\delta)}{\delta^2}~.
\end{equation}
As $\delta\to 0$, the right-hand side approaches $\frac{\upd^2}{\upd t^2}\evalat{t=0} E_{AB}(t)$.
As for the left-hand side, we may rewrite
\begin{align}
\Vert\vec{u}' - \vec{v}'\Vert^2 & = \Vert\vec{u}'\Vert^2 + \Vert\vec{v}'\Vert^2 - 2\Vert\vec{u}'\Vert \Vert\vec{v}'\Vert \cos (\alpha' -\beta') \\
& =\Vert\vec{u}_{\delta}\Vert^2 + \Vert\vec{v}_{\delta}\Vert^2 - 2\Vert\vec{u}_{\delta}\Vert \Vert\vec{v}_{\delta}\Vert \cos (\alpha' -\beta')~,
\end{align}
where $\vec{u}_{\delta} = \exp_{A_{-\delta}}^{-1}(A_\delta)$ and $\vec{v}_{\delta}=\exp_{B_{-\delta}}^{-1}(B_\delta)$ (so that $\Vert\vec{u}_{\delta}\Vert=\Vert\vec{u}'\Vert$ and $\Vert\vec{v}_{\delta}\Vert=\Vert\vec{v}'\Vert$), and $\alpha'$ (resp.~$\beta'$) are the oriented angles (measured in $\bR^2$) between the geodesic $A_{-\delta}B_{-\delta}$ and $\vec{u}'$ (resp.~$\vec{v}'$).
By a well-known theorem of Alexandrov (see \cite{MR0049584} or \cite{MR1744486}), the interior angles of a quadrilateral of $N$ are smaller than those of the model, and we conclude that $\alpha_\delta \leqslant \alpha'$ and $\pi-\beta_\delta \leqslant \pi- \beta'$, where $\alpha_\delta$ and $\beta_\delta$ are the interior angles at $A_{-\delta}$ and $B_{-\delta}$, respectively, of the quadrilateral with vertices $A_{-\delta}$, $B_{-\delta}$, $B_{\delta}$, and $A_\delta$. 
Thus $\alpha_\delta - \beta_\delta \leqslant \alpha'-\beta'$, and we may conclude that 
\begin{equation}
\Vert\vec{u}' - \vec{v}'\Vert^2 \geqslant \Vert\vec{u}_{\delta}\Vert^2 + \Vert\vec{v}_{\delta}\Vert^2 - 2\Vert\vec{u}_{\delta}\Vert \Vert\vec{v}_{\delta}\Vert \cos (\alpha_\delta -\beta_\delta)~.
\end{equation}

Using the latter for the lefthand side of \eqref{eq:EABinequality2} and taking the limit as $\delta\to 0$, we find that 
\begin{equation}
\label{eq:EABinequality3}
\Vert\vec{u}\Vert^2 + \Vert\vec{v}\Vert^2 - 2\Vert\vec{u}\Vert \Vert\vec{v}\Vert \cos (\alpha -\beta)  \leqslant \frac{\upd^2}{\upd t^2}\evalat{t=0} E_{AB}(t)~.
\end{equation}
%
%
%
%
%
The lefthand side here is precisely $ \Vert  \vec{u} - P_{[BA]}\vec{v} \Vert^2$, as in the proof of \autoref{prop:SecondDerivativeHyperbolicEstimates}.
\end{proof}

This in turn makes a generalization of \autoref{thm:StrongConvexityEnergyGraphH2} 
possible:

\begin{theorem} \label{thm:StrongConvexityEnergyGraphGeneral}
Let $\cG$ be a biweighted $\tilde{S}$-triangulated graph, let $N$ be a two-dimensional closed Riemannian manifold of nonzero Euler 
characteristic and 
nonpositive sectional curvature
and suppose that $\rho \colon \pi_1S\to \pi_1 N$ is a group homomorphism induced by a homotopy class of maps $S \to N$ of nonzero degree.
Then the energy functional $E_\cG : \Map_{\text{eq}}(\cG, \tilde{N}) \to \bR$ is strongly convex 
with modulus of convexity given by \eqref{eq:SecondDerivativeLowerBound1}.
\end{theorem}


\begin{proof}
\autoref{prop:SecondDerivativeGeneralEstimates} may take the place of \autoref{prop:SecondDerivativeHyperbolicEstimates}, and we see that \autoref{lem:meshVectorsEqualLength} holds for any nonpositively curved target.
Now the proof of \autoref{prop:SecondDerivativeEstimateBelow} holds verbatim. (The reader can observe that the fourth sentence of the last paragraph still holds in the setting of nonpositive curvature: the sum of exterior angles of a geodesic triangle is between $\pi$ and $3\pi$.)
\end{proof}

%% file: DiscreteHeatFlow.tex
\section{Discrete heat flow}
\label{sec:DiscreteHeatFlow}

\subsection{Gradient descent in Riemannian manifolds}
\label{subsec:Gradient descent}

The area of mathematics concerned with methods for finding the minima of a convex function $F \colon \Omega \to \bR$, called \emph{convex optimization}, has been intensely developed in the last few decades and finds countless applications. 
The majority of the existing literature deals with the classical case where $\Omega$ is a 
subset of a Euclidean space; 
the Riemannian setting has been far less explored although it is a 
natural and useful extension. Udri\c{s}te's book \cite{MR1326607} is a good standard reference for Riemannian convex optimization (see \eg{} \cite{MR2364186}, \cite{ZhangSra2016} for more recent developments). Our goal is to present a simple but effective method that can be implemented to find the minimum of the discrete energy functional, with a rigorous proof of convergence and explicit control of the convergence rate. Of course, there are more advanced and faster algorithms in practice. For example, the C++ library \emph{ROPTLIB} \cite{ROPTLIB} was developed for this purpose. 

\smallskip

A \emph{gradient descent method} is an iterative algorithm for minimizing a function $F \colon \Omega \subseteq \bR^N \to \bR$ which produces a sequence $(x_k)_{k \geqslant 0}$ of points in $\Omega$, defined inductively by:
\begin{equation} \label{eq:GradientDescentMethodEuclidean}
 x_{k+1} = x_k - t_k \grad F(x_k)~.
\end{equation}
In this relation, $t_k \geqslant 0$ is a chosen \emph{stepsize}. 
If $F$ has good convexity properties such as being strongly convex, then a small enough stepsize $t_k = t > 0$ guarantees convergence of the sequence $(x_k)$ to a minimum of $F$, with explicit control of the convergence rate.

This method naturally extends to the Riemannian setting, \ie{} when $F \colon \Omega \subseteq M \to \bR$ is defined on a subset $\Omega$ of a Riemannian manifold $M$, in which case \eqref{eq:GradientDescentMethodEuclidean} should be understood as
\begin{equation} \label{eq:GradientDescentMethodRiemannian}
  x_{k+1} = \exp_{x_k}( - t_k \grad F(x_k))~.
\end{equation}
As in the Euclidean setting, $-\grad F(x_k)$ is the direction of steepest descent 
for $F$ at $x_k$, so it is natural to look for $x_{k+1}$ in the geodesic ray based at $x_k$ in this direction. 
Note that the gradient descent method can simply be described as Euler's method for the gradient flow ODE:
\begin{equation}
 x'(t) = -\grad F(x(t))~.
\end{equation}

\subsubsection*{Gradient descent method with fixed stepsize for strongly convex functions}

The gradient descent method with fixed stepsize remains valid for $\cC^2$ strongly convex functions on Riemannian manifolds:
\begin{theorem}[{\cite[Chap. 7, Theorem 4.2]{MR1326607}}] \label{thm:GradientDescentFixedStepRiemannian}
 Let $(M,g)$ be a complete Riemannian manifold and let $F \colon M \to \bR$ be a function of class $\cC^2$. 
 Assume that there exists $\alpha, \beta >0$ such that: 
 \begin{equation}
  \forall v \in \upT M \quad  \alpha \, \Vert v\Vert^2 \leqslant (\Hess F)(v,v) \leqslant \beta \, \Vert v\Vert^2
 \end{equation}
 Then $F$ has a unique minimum $x^*$. Furthermore, for $t \in (0, \frac{1}{\beta}]$, the gradient descent method with fixed stepsize $t_k = t$
converges to $x^*$ with a linear 
convergence rate:
\begin{equation} \label{eq:ConvergenceRateFixedStepsize}
 d(x_k, x^*) \leqslant c \, q^k
\end{equation}
The constants $c \geqslant 0$ and $q \in [0,1)$ are given by: 
\begin{equation}\label{eq:ConstantsConvergenceRate}
 c = \sqrt{\frac{2}{\alpha}(F(x_0)-F(x^*))} \qquad  q = \sqrt{1 - \frac{t}{2} \alpha\left(1 + \frac{\alpha}{\beta}\right)}~.
\end{equation}
\end{theorem}

\begin{remark}
 A key step in the proof of \autoref{thm:GradientDescentFixedStepRiemannian} is that $(F(x_k) - F(x^*))$ is nonincreasing and limits to $0$ with a linear convergence rate. In particular, $(F(x_k))_{k \geqslant 0}$ is nonincreasing, therefore any sublevel set of $F$ is stable under the gradient descent. Moreover, such a set is convex and compact by strong convexity of $F$. 
Thus, the gradient descent method is valid even if the Hessian of is not bounded above: one can restrict to a sublevel set, where the Hessian of $F$ is bounded. 
\end{remark}

\subsubsection*{Gradient descent method with optimal stepsize for strongly convex functions}

There are many variants of the gradient descent method that can be more or less useful depending on the context (see \eg{} \cite{ZhangSra2016}, \cite{FerreiraLouzeiroPrudente}). One of them is the \emph{optimal stepsize gradient descent}, an instance of the gradient descent method \eqref{eq:GradientDescentMethodRiemannian} where one performs a \emph{line search} in order to determine a stepsize $t_k$ that minimizes $F(x_{k+1})$.
Clearly, when $F$ is strongly convex, such a $t_k$ exists, is unique, and is $> 0$ unless $x_k = x^*$.
When the Hessian of $F$ is known analytically, Newton's method offers a very fast line search.

\begin{theorem} \label{thm:OptimalStepRiemannian}
 Let $F \colon M \to \bR$ as in \autoref{thm:GradientDescentFixedStepRiemannian}.
 The optimal stepsize gradient descent has a linear convergence rate at least as
 fast as that of \autoref{thm:GradientDescentFixedStepRiemannian}, for any choice of the fixed stepsize.
\end{theorem}

\begin{proof}
\autoref{thm:OptimalStepRiemannian} is derived from a careful analysis of the proof of \autoref{thm:GradientDescentFixedStepRiemannian} which can be found in 
\cite[Chapter 7, Theorem 4.2]{MR1326607}. This proof is a combination of three observations:
\begin{enumerate}[(i)]
 \item For any $x\in M$:
 \begin{equation} \label{eq:Observation1}
  \frac{\alpha}{2} d(x, x^*)^2 \leqslant F(x) - F(x^*) \leqslant \frac{\beta}{2} d(x, x^*)^2~.
  \end{equation}
  This follows from a Taylor expansion of $F$ at $x^*$ along the geodesic $[x^*, x]$.
 \item For any $x\in M$:
 \begin{equation} \label{eq:Observation2}
  \Vert \grad F(x) \Vert^2 \geqslant \alpha\left(1 + \frac{\alpha}{\beta}\right)(F(x) - F(x^*))~.
 \end{equation}
 This follows from a Taylor expansion of $F$ at $x$ along the geodesic $[x, x^*]$ and from \eqref{eq:Observation1}.
 \item For any $x\in M$ and for any $t \in [0, \frac{1}{\beta}]$:
 \begin{equation} \label{eq:Observation3}
  F(x) - F(x^+(t)) \geqslant \frac{t}{2}  \Vert \grad F(x) \Vert^2
 \end{equation}
 where $x^+(t) \coloneqq \exp_x(-t \grad F(x))$. This follows from a Taylor formula for $F$ along $[x, x^+(t)]$.
\end{enumerate}
It follows immediately from these three observations that for any $x\in M$ and for any $t \in [0, \frac{1}{\beta}]$:
\begin{equation} \label{eq:Obs4}
 F(x^+(t)) - F(x^*) \leqslant Q(t) \left(F(x) - F(x^*)\right)
\end{equation}
where $Q(t) = 1 - \frac{t}{2} \alpha(1 + \frac{\alpha}{\beta}) = q^2$. 

When one performs the gradient descent method with fixed stepsize $t$, by assumption $x_{k+1} = x_k^+(t)$. 
\autoref{thm:OptimalStepRiemannian} is then easily concluded by finding $F(x_k) - F(x^*) \leqslant Q^k \left(F(x_0) - F(x^*)\right)$ from \eqref{eq:Obs4}
(with an obvious induction) and making one last use of \eqref{eq:Observation1}. 

If instead one performs an optimal stepsize gradient descent,
then $x_{k+1} = x_k^+(t_k)$, where $t_k$ is the optimal step. Fix $t \in [0, \frac{1}{\beta}]$. By definition of the optimal step, $F(x_k^+(t_k)) \leqslant F(x_k^+(t))$,
so $F(x_{k+1}) - F(x_k) \leqslant  F(x_k^+(t)) - F(x_k)$. Therefore we can derive from \eqref{eq:Obs4} that 
\begin{equation}
 F(x_{k+1}) - F(x^*) \leqslant Q(t) \left(F(x_k) - F(x^*)\right)
\end{equation}
and the conclusion follows like before.
\end{proof}

\subsection{Convergence of the discrete heat flow}\label{subsec:ConvergenceOfDiscreteHeatFlow}

The discrete heat flow can be described as a discretization both in time and space of the heat flow on $\cC^\infty(M,N)$ .
Recall that the smooth heat flow is the gradient flow of the smooth energy functional:
\begin{equation} \label{eq:HeatFlowODE2}
 \ddt f_t = \tau(f_t)
\end{equation}
where $\tau(f_t) = - \grad E(f_t)$ is the tension field of $f_t$ (cf \autoref{subsec:EnergyFunctionalHarmonicMaps}, \autoref{subsec:SmoothHeatFlow}). 

We recall the setup of our discretization:
Let $\cG$ be a biweighted $\tilde{S}$-triangulated graph
(\autoref{def:BiweightedGraph}), let $N$ be a Riemannian manifold, and let $\rho \colon \pi_1S \to \Isom(N)$ be a group homomomorphism. 
Recall that the discrete energy is a function $E_\cG \colon \Map_{\text{eq}}(\cG, N) \to \bR$ (\autoref{def:EnergyFunctionalGraph}), 
where $\Map_{\text{eq}}(\cG, N)$ is the space of $\rho$-equivariant maps $\cG \to N$ . 
The latter space has a natural Riemannian structure with respect to which the gradient of the energy is minus the discrete tension field $\tau_\cG$ (\autoref{def:L2RiemannianMetricGraph} and \autoref{prop:FirstVariationalFormulaGraph}). 
Thus we define the discrete heat flow:

\begin{definition} \label{def:DiscreteHeatFlow}
The \emph{discrete heat flow} is the iterative algorithm which, given $f_0 \in \Map_{\text{eq}}(\cG, N)$, produces the sequence $(f_k)_{k \in \bN}$ in $\Map_{\text{eq}}(\cG, N)$
defined inductively by the relation
\begin{equation} \label{eq:DiscreteHeatFlow1}
f_{k+1}(x) = \exp_{f_k(x)}\left(t_k (\tau_\cG f_k)_x\right)~,
\end{equation}
where $t_k \in \bR$ is a chosen stepsize. 
\end{definition}

The main theorem of this section is an immediate application of \autoref{thm:StrongConvexityEnergyGraphGeneral} and \autoref{thm:GradientDescentFixedStepRiemannian}:

\begin{theorem} \label{thm:ConvergenceDiscreteHeatFlowGeneral}
Let $\cG$ be a biweighted $\tilde{S}$-triangulated graph. Let $N$ be two-dimensional Hadamard manifold of nonpositive curvature 
and $\rho \colon \pi_1S\to \Isom(N)$ is a faithful representation whose image is contained in a discrete subgroup of $\Isom(N)$ acting freely and properly. 
Then there exists a unique $\rho$-equivariant harmonic map $f^* \colon \cG \to N$. 
Moreover, for any $f_0 \in \Map_{\text{eq}}(\cG, N)$
and for any sufficiently small $t>0$, the discrete heat flow with initial value $f_0$ and fixed stepsize $t$ converges to $f^*$ with a linear convergence rate:
\begin{equation} \label{eq:ConvergenceRateDHFGeneral}
 d(f_k, f^*) \leqslant c q^k
\end{equation}
where $c>0$ and $q \in [0, 1)$ are constants, and $d(f_k, f^*)$ is the $\upL^2$ distance in $\Map_{\text{eq}}(\cG, N)$.
\end{theorem}

Of course, it also follows from \autoref{thm:OptimalStepRiemannian} that the discrete heat flow with optimal stepsize converges to $f^*$ as well,
with a linear convergence rate as least as fast as \eqref{eq:ConvergenceRateDHFGeneral}.

We emphasize that in our favorite setting where $N = \bH^2$ and $\rho \colon \pi_1S \to \Isom^+(\bH^2)$ is Fuchsian, \autoref{thm:StrongConvexityEnergyGraphH2}
enables explicit estimates on the constants $c$ and $q$ in \eqref{eq:ConvergenceRateDHFGeneral}: the expressions of $c$ and $q$ are given
by \eqref{eq:ConstantsConvergenceRate}, in which $\alpha$ is given by \eqref{eq:SecondDerivativeLowerBound1} and $\beta$ is given by \eqref{eq:SecondDerivativeUpperBound} with $E_0 = E(f_0)$.

\subsection{Experimental comparison of convergence rates}
\label{subsec:ExperimentalComparison}

In \autoref{fig:graphConstants} and \autoref{fig:graphComparison} we present numerical experiments performed with the software $\Harmony$. 

\subsubsection*{Comparison of different fixed stepsizes}

In the first experiment (\autoref{fig:graphConstants}) we observe the number of iterations required for the discrete heat flow with fixed stepsize to converge 
as a function of the stepsize.

Let $S$ be a closed oriented surface of genus $2$. We fix a domain Fuchsian representation $\rho_\tL \colon \pi_1S \to \Isom^+(\bH^2)$: the representation pictured on the left in \autoref{fig:InitialFunctionPic}, and let $\Harmony$ construct an invariant mesh (depth $4$, $1921$ vertices). We let the target Fuchsian representation $\rho_\tR$ vary, 
taking Fenchel-Nielsen lengths $(2,2,\ell)$ and twists $(-1.5,2,0.5)$, where $\ell \in \{2.5, 1.5, 0.5, 0.2\}$. 

We observe that the plotted points resemble in profile functions of the form $-C_1\left( \log(1-C_2t) \right)^{-1}$, which is precisely the type of function predicted by \autoref{thm:ConvergenceDiscreteHeatFlowGeneral}.

\begin{figure}
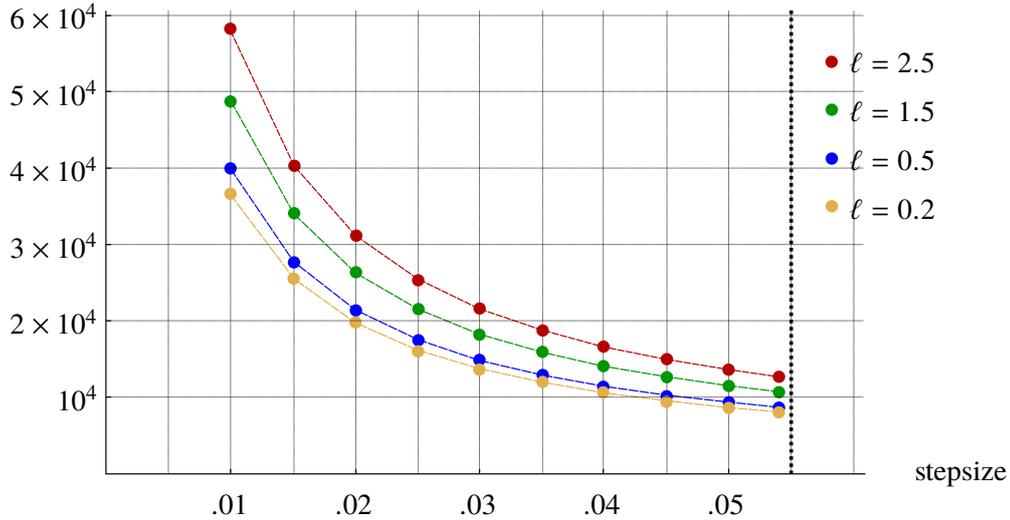

\centering
\begin{lpic}{graphConstantStepPerformance2(10cm)}
	\lbl[]{29,-7;$.01$}
	\lbl[]{58,-7;$.02$}
	\lbl[]{87,-7;$.03$}
	\lbl[]{116,-7;$.04$}
	\lbl[]{145,-7;$.05$}
	\lbl[]{-6.5,18;$10^4$}
	\lbl[]{-12,36;$2\times 10^4$}
	\lbl[]{-12,54;$3\times 10^4$}
	\lbl[]{-12,72;$4\times 10^4$}
	\lbl[]{-12,90;$5\times 10^4$}
	\lbl[]{-12,108;$6\times 10^4$}
	\lbl[]{184,97;$\ell=2.5$}
	\lbl[]{184,85.5;$\ell=1.5$}
	\lbl[]{184,74;$\ell=0.5$}
	\lbl[]{184,62.5;$\ell=0.2$}
	\lbl[]{-10,125;number of iterations}
	\lbl[]{200,0;stepsize}
\end{lpic}
\vspace{.8cm}
\caption{Number of iterations against stepsize in the discrete heat flow with fixed stepsize performed by \Harmony{}.}
\label{fig:graphConstants}
\end{figure}

\subsubsection*{Comparison of our three methods}

\begin{figure}
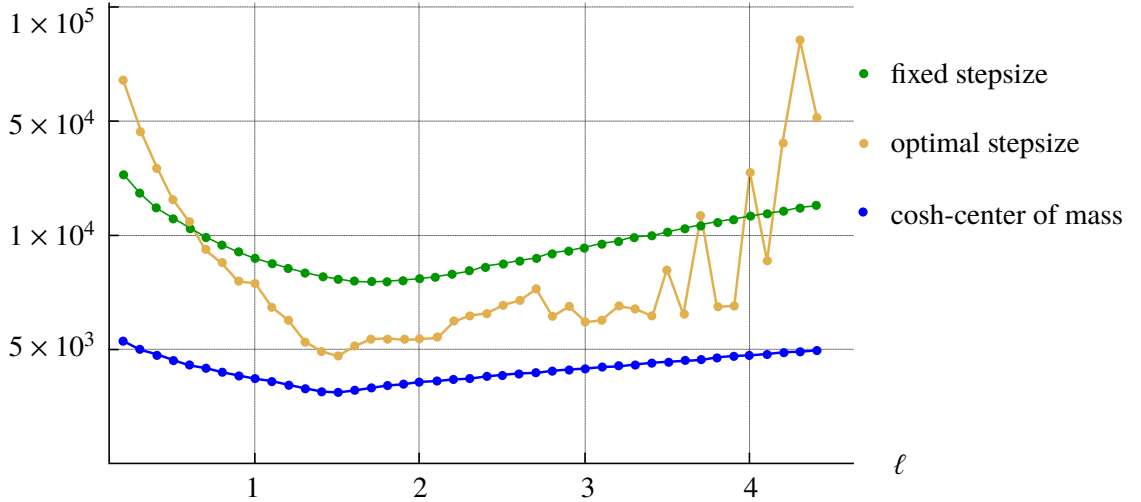

\centering
\begin{lpic}{graphComparison2(10cm)}
	\lbl[]{-6,115;number of iterations}
	\lbl[]{-12,26;$5\times 10^3$}
	\lbl[]{-12,50;$1\times 10^4$}
	\lbl[]{-12,74;$5\times 10^4$}
	\lbl[]{-12,96;$1\times 10^5$}
	\lbl[]{31,-5;$1$}
	\lbl[]{67,-5;$2$}
	\lbl[]{103,-5;$3$}
	\lbl[]{138,-5;$4$}
	\lbl[]{170,0;$\ell$}
	\lbl[l]{168,84;fixed stepsize}
	\lbl[l]{168,69;optimal stepsize}
	\lbl[l]{168,54;cosh-center of mass}
\end{lpic}
\vspace{.8cm}
\caption{Comparison of the three methods performed by \Harmony{}.}
\label{fig:graphComparison}
\end{figure}

For the second experiment (\autoref{fig:graphComparison}) we compare the convergence rate, in terms of number of iterations, of our three methods:
\begin{itemize}
 \item Discrete heat flow with fixed stepsize (see \autoref{subsec:ConvergenceOfDiscreteHeatFlow}),
 \item Discrete heat flow with optimal stepsize (see \autoref{subsec:ConvergenceOfDiscreteHeatFlow}),
 \item Cosh-center of mass method (see \autoref{subsec:CenterOfMassMethods}).
\end{itemize}
We keep the same setting as before, letting $\ell$ this time vary between $0.2$ and $4.4$. As the figure shows, the cosh-center of mass method is 
more effective than either gradient descent methods. 

%% file: CenterOfMass.tex
\section{Center of mass methods}
\label{sec:CenterOfMass}

In this section we investigate a center of mass algorithm towards the minimization of the discrete energy.
We shall see that it is in some sense a variant of the fixed stepsize discrete heat flow. 
In the current state of our software \Harmony{}, it is the most effective method (see \autoref{subsec:ExperimentalComparison}). 

First we recall facts about centers of mass in Riemannian manifolds and investigate how they relate to harmonic maps.
We shall prove in particular a generalized mean value property for harmonic maps between Riemannian manifolds (\autoref{thm:GeneralizedMeanValueProperty}).

\subsection{Centers of mass in metric spaces and Riemannian manifolds} \label{subsec:CenterOfMass}

Let $(\Omega, \cF, \mu)$ be a probability space, $(X,d)$ be a metric space, and $h \colon \Omega \to X$ a measurable map.

\begin{definition} \label{def:CenterOfMass}
A \emph{center of mass} (or \emph{barycenter}) of $h$ is a minimizer of the function 
\begin{equation} \label{eq:DefinitionCenterOfMass}
\begin{split}
P_h \colon X  & \to \bR\\
x & \mapsto \frac{1}{2} \int_\Omega d(x,h(y))^2 \, \upd \mu(y)~.
\end{split}
\end{equation}
\end{definition}

In general, neither existence nor uniqueness of centers of mass hold. If $X$ is a Hadamard space 
and $h \in \upL^2(\Omega, X)$ then existence and uniqueness do hold 
\cite[Lemma 2.5.1]{MR1266480}.
For Riemannian manifolds we have:

\begin{theorem}[Karcher \cite{MR0442975}] \label{thm:Karcher}
Assume that $X$ is a complete Riemannian manifold and $h$ takes values in a ball $B = B(x_0, r) \subset X$ such that:
\begin{itemize}
 \item $B$ is \emph{strongly convex}: any two points of $B$ are joined by a unique minimal geodesic $\gamma \colon [0,1] \to X$, and each such geodesic maps entirely into $B$.
 \item $B$ has nonpositive sectional curvature, or $r < \frac{\pi}{4 \sqrt{K}}$ where $K > 0$ is an upper bound for the sectional 
 curvature in $B$.
\end{itemize}
Under these conditions, the function $P_h$ of $\eqref{eq:DefinitionCenterOfMass}$
only has interior minimimizers on $\bar{B}$ and is strongly convex inside $B$. Consequently, existence and uniqueness of the center of mass hold. 
\end{theorem}

Note that if $X$ is a complete Riemannian manifold, any sufficiently small $r>0$ meets the requirements of \autoref{thm:Karcher}. 
The center of mass of $h$ is the unique point $G \in X$ such that
\begin{equation} \label{eq:ImplicitEquationForCenterOfMass}
 \int_\Omega \exp_G^{-1}(h(y)) \, \upd \mu(y) = 0~.
\end{equation}
This simply expresses the vanishing at $G$ of the gradient of the function $P_h \colon X \to \bR$ of \eqref{eq:DefinitionCenterOfMass}.

Of course, \autoref{def:CenterOfMass} generalizes the usual notion of center of mass 
in $\bR^n$: when $h\in\upL^2(\Omega,\bR^n)$, the center of mass given by $G = \int_\Omega h(y) \, \upd \mu(y)$.
When $X$ is not Euclidean, the center of mass is only defined implicitly, 
but one can estimate proximity to the center of mass:
\begin{lemma} \label{lem:ProximityToCenterOfMass}
Assume that the conditions of \autoref{thm:Karcher} are satisfied. In particular the center of mass $G$ of $h$ is well-defined. If $G'$ is a point in $X$ such that:
\begin{equation}
 \left\Vert \int_\Omega \exp_{G'}^{-1}(h(y)) \, \upd \mu(y) \right\Vert \leqslant \delta
\end{equation}
then
\begin{equation}
 d(G, G') \leqslant C \delta~,
\end{equation}
where $C = 1$ when $X$ has nonpositive curvature, or $C = C(K, r) > 0$ when $X$ has sectional curvature bounded above by $K > 0$ 
in a strongly convex ball of radius $r$ containing the image of $h$.
\end{lemma}

\begin{proof}
This is an immediate consequence of the fact that the function $P_h$ of \eqref{eq:DefinitionCenterOfMass} is $C$-strongly convex under the assumptions of the lemma. Also see \cite[Thm 1.5]{MR0442975}.
\end{proof}

\subsection{Generalized mean value property} \label{subsec:GeneralizedMeanValueProperty}

In this section we show 
a generalized mean value property for smooth harmonic maps between Riemannian manifolds.
Let $f \colon (M,g) \to (N,h)$ be a smooth map between Riemannian manifolds.

Fix $x \in M$. Denote by $S_r$ (resp.~$B_r$) the sphere (resp.~the closed ball) centered at the origin
of radius $r$ in the Euclidean vector space $(\upT_x M, g)$. Also denote $\hat{S}_r$ (resp.~$\hat{B}_r$) the sphere (resp.~the closed ball) 
 centered at $x$ of radius $r$ in $(M,g)$.
The topological space $S_r$ (resp. $B_r$) can be equipped with a natural Borel probability measure
by taking the measure induced from the Euclidean metric $g$ in $T_x M$, 
renormalized so that it has total mass $1$. Similarly, $\hat{S}_r$ (resp. $\hat{B}_r$) can be equipped with a natural Borel probability measure
by taking the measure induced from the Riemannian metric $g$.

\begin{definition} \label{def:AverageMap}
We define four functions $S_r f, B_r f, \hat{S}_r f, \hat{B}_r f \colon M \to N$ as follows. Given $x \in M$:
\begin{itemize}
 \item $S_r f(x)$ (resp.~$B_r f(x)$) is the center of mass of  
 $ f \circ \exp_x \colon S_r \to N$ (resp.~$f \circ \exp_x\colon B_r \to N$).
 \item $\hat{S}_r f(x)$ (resp.~$\hat{B}_r f(x)$) is the center of mass of $f\evalat{\hat{S}_r} \colon \hat{S}_r \to N$
 (resp.~ $f\evalat{\hat{B}_r} \colon \hat{B}_r \to N$).
\end{itemize}
\end{definition}

\begin{remark}
The four functions of \autoref{def:AverageMap} are well-defined as long as $(N,h)$ is a Hadamard manifold, or as long as $r$ is small enough and $(N,h)$ has sectional curvature bounded above and injectivity radius bounded below by a positive number (\eg{} $N$ is compact).
\end{remark}

Note that $S_r f(x)$ and $\hat{S}_r f(x)$ are different in general, as are $B_r f(x)$ and $\hat{B}_r f(x)$. However they are very close when $r$ is small:

\begin{proposition} \label{prop:ComparisonAverages}
Let $f \colon M \to N$ be a smooth map. Then for all $x \in M$:
\begin{equation}
 \begin{aligned}
  d(S_r f(x), \hat{S}_r f(x)) &= \bigO(r^4)\\
  d(B_r f(x), \hat{B}_r f(x)) &= \bigO(r^4)~.\\
 \end{aligned}
\end{equation}
\end{proposition}

\begin{proof}
The proof is technical but not very difficult. It is basically derived from a Taylor expansion of the metric in normal coordinates at $x$ and one use of \autoref{lem:ProximityToCenterOfMass}. We will do several similar proofs in what follows, so we skip the details for brevity.
\end{proof}

Of course, in the case where $M = \bR^m$ and $N = \bR$ (or $N = \bR^n$), $S_r f$ and $\hat{S}_r f$ (resp. $B_r f$ and $\hat{B}_r f$) coincide. We recall the celebrated 
mean property for harmonic functions in this setting:

\begin{theorem} \label{thm:ClassicalMeanProperty}
 $f \colon \bR^m \to \bR$ is harmonic if and only if $S_r f = B_r f = f$ for all $r > 0$.
\end{theorem}

More generally, if $M$ is any Riemannian manifold and $N = \bR$, Willmore \cite{MR0033408} proved that $\hat{S}_r f = f$ characterizes harmonic maps
if and only if $M$ is a \emph{harmonic manifold}.

The central theorem of this subsection is the following:
\begin{theorem} \label{thm:AverageMap}
Let $f \colon M \to N$ be a smooth map. For all $x \in M$, as $r \to 0$:
\begin{align}
 d\left(S_r f(x)~,~ \exp_{f(x)}\left(\frac{r^2}{2 m} \tau(f)_{x}\right)\right) &= \bigO(r^4) \label{eq:AverageMap1}\\
 d\left(B_r f(x)~,~ \exp_{f(x)}\left(\frac{r^2}{2(m+2)} \tau(f)_{x}\right)\right) &= \bigO(r^4) \label{eq:AverageMap2}~.
 \end{align}
\end{theorem}

The following ``generalized mean property for harmonic functions between Riemannian manifolds'' is an immediate corollary of \autoref{thm:AverageMap}:

\begin{theorem} \label{thm:GeneralizedMeanValueProperty}
Let $f \colon M \to N$ be a smooth map. The following are equivalent:
\begin{enumerate}[(i)]
 \item $f$ is harmonic.
 \item $d(f(x), S_r f(x)) = \bigO(r^4)$ for all $x\in M$.
 \item $d(f(x), B_r f(x)) = \bigO(r^4)$ for all $x\in M$.
\end{enumerate}
\end{theorem}

\begin{remark}
 It is an immediate consequence of \autoref{prop:ComparisonAverages} that \autoref{thm:AverageMap} and \autoref{thm:GeneralizedMeanValueProperty}
 also hold for $\hat{S}_r f(x)$ instead of $S_r f(x)$, and $\hat{B}_r f(x)$ instead of $B_r f(x)$.
\end{remark}


In the remainder of this subsection we show \autoref{thm:AverageMap}. 
We only prove \eqref{eq:AverageMap1}; the proof of \eqref{eq:AverageMap2} follows exactly the same lines. Consider a smooth map $f \colon (M,g) \to (N,h)$ and fix $x \in M$.

\begin{lemma} \label{lem:AverageG}
 Let $r > 0$. Denote by $S_r$ the Euclidean sphere of radius $r>0$ in $T_x M$ and $\sigma_r$ its area density. 
Then, as $r \to 0$, the following estimate holds:
\begin{equation} \label{eq:AverageG}
 \frac{1}{\Area(S_r)} \int_{S_r} \exp_{f(x)}^{-1} \circ f \circ \exp_x(u) \, \upd \sigma_r(u) = \frac{r^2}{2 m} \tau(f)_x + \bigO(r^4)
\end{equation}
where $m = \dim M$.
\end{lemma}

\begin{proof}
We write the Taylor expansion of the function $\hat{f} =  \exp_{f(x)}^{-1} \circ f \circ \exp_x \colon T_x M \to T_{f(x)} N$:
\begin{equation}
 \hat{f}(u) = \hat{f}(0) + (\upD \hat{f})_0(u) + \frac{1}{2}(\upD^2 \hat{f})_0(u,u) + \frac{1}{6}(\upD^3 \hat{f})_0(u,u,u) + \bigO(\Vert u \Vert^4)~.
\end{equation}
Let us integrate this identity over $S_r$. We have:
\begin{itemize}
 \item $\hat{f}(0) = 0$, so $\frac{1}{\Area(S_r)} \int_{S_r} \hat{f}(0) \, \upd \sigma_r(u) = 0$.
 \item $(\upD \hat{f})_0(u) = (\upd f)_x(u)$ is an odd function of $u$, so $\frac{1}{\Area(S_r)} \int_{S_r} (\upD \hat{f})_0(u) \, \upd \sigma_r(u) = 0$.
 \item It is straightforward to check that $(\upD^2 \hat{f})_0(u, u) = (\Hess f)_x(u, u)$ by definition of $\Hess f$. Moreover, since this is a quadratic
 function of $u$, we can apply \autoref{lem:BilinearFormAverage}: 
 \begin{equation}
  \frac{1}{\Area(S_r)} \int_{S_r} (\upD^2 \hat{f})_0(u, u) \, \upd \sigma_r(u) = \frac{r^2}{m}\tr((\Hess f)_x) = \frac{r^2}{m} \tau(f)_x~.
 \end{equation}
 \item $(\upD^3 \hat{f})_0(u,u,u)$ is an odd function of $u$, so $\frac{1}{\Area(S_r)} \int_{S_r} (\upD^3 \hat{f})_0(u, u, u) \, \upd \sigma_r(u) = 0$.
\end{itemize}
Putting all this together, we get \eqref{eq:AverageG}.
\end{proof}

The following lemma is required to complete the proof of \autoref{lem:AverageG}:

\begin{lemma} \label{lem:BilinearFormAverage}
Let $(V, g = \langle \cdot, \cdot \rangle)$ be a Euclidean vector space and let $B \colon V \times V \to \bR$ be a symmetric bilinear form. Denote by $S_r = S(0,r)$ the sphere centered at the origin in $V$ with radius $r > 0$, $\upd \sigma_r$ the area density on $S_r$ induced from the metric $g$ and $\Area(S_r) = \int_{S_r} \upd \sigma_r$ its area. Then:
\begin{equation}
 \frac{1}{\Area(S_r)}\int_{S_r} B(x, x) \, \upd \sigma_r = \frac{r^2}{\dim V}\tr_g(B)~.
\end{equation}
\end{lemma}
Here we have denoted by $\tr_g(B)$ the $g$-trace of $B$, \ie{} the trace of the $g$-self adjoint endomorphism of $V$ associated to $B$,
or, equivalently, the trace of a matrix representing $B$ in a $g$-orthonormal basis.

\begin{proof}
Let $(e_1,  \dots, e_n)$ be basis of $V$ which is $g$-orthonormal and $B$-orthogonal (the existence of such a basis is precisely the spectral theorem).
Let $\lambda_k =  B(e_k, e_k)$ for $k \in \{1, \dots n\}$. For any vector $x = \sum_{k=1}^n x_k e_k$, the quadratic form is given by 
$B(x,x) = \sum_{k=1}^n \lambda_k {x_k}^2$, hence:
\begin{equation}
 \int_{S_r} B(x, x) \,\upd \sigma_r = \sum_{k=1}^n \lambda_k \int_{S_r} {x_k}^2  \,\upd \sigma_r~.
\end{equation}
The integrals $I_k = \int_{S_r} {x_k}^2 \,\upd \sigma_r$ can be swiftly computed starting with the observation that any two of them are equal. 
Indeed, for $k \neq l$, one can easily find a linear isometry $\varphi$ such that ${x_k}^2 \circ \varphi = {x_l}^2$; the change of variables theorem ensures that $I_k = I_l$.
One can then write $I_k = \frac{1}{n} \sum_{l=1}^n I_l$ for any $k$. That is $I_k = \frac{1}{n}\int_{S_r} \left(\sum_{k=1}^n {x_k}^2\right) \,\upd \sigma_r$. However $\sum_{k=1}^n {x_k}^2 = g(x,x) = r^2$ for any $x \in S_r$. This yields $I_k = \frac{1}{n}\int_{S_r} r^2 \,\upd \sigma_r = \frac{r^2}{n} \Area(S_r)$.
The desired result follows.
\end{proof}

It is easy to see that \autoref{thm:AverageMap} follows immediately from \autoref{lem:AverageG} when $N = \bR^n$.
When $N$ is not Euclidean, centers of mass in $N$ are only defined implicitly (by equation \eqref{eq:ImplicitEquationForCenterOfMass}), so we have to work harder
to prove \autoref{thm:AverageMap}. The trick is to use \autoref{lem:ProximityToCenterOfMass}. 

First we need a Riemannian geometry estimate in the following general setting. Let $A$, $B$, $C$ be three points in a Riemannian manifold $(M,g)$. We assume that $B$ and $C$ are contained in a sufficiently small ball centered at $A$ for what follows to make sense. Denote by $\vec{u_A} = (\exp_A)^{-1}(B)$, $\vec{u_B} = (\exp_B)^{-1}(C)$, and $\vec{u_C} = (\exp_C)^{-1}(A)$.
If we were in a Euclidean vector space, we could write:
\begin{equation}
 \vec{u_A} + \vec{u_B} + \vec{u_C} = 0~.
\end{equation}
We would like to find an approximate version of this identity in general. Note that the sum $\vec{u_A} + \vec{u_B} + \vec{u_C}$ does not make sense, 
because these vectors are based at different points. Denote $\vec{v}$ the parallel transport
of $-\vec{u_C}$ along the geodesic segment $[C, A]$ and $\vec{w}$ the parallel transport of $\vec{u_B}$ along the geodesic segment $[B, A]$. 
Let us also write $\vec{u} = \vec{u_A}$ for aesthetics. Now the vectors $\vec{u}$, $\vec{v}$, $\vec{w}$ are all based at $A$,
and one expects that $\vec{w} = \vec{v} - \vec{u}$ up to some error term.

\begin{lemma} \label{lem:RiemannianEstimateTriangle1}
Using the setting and notations above, we have as $\Vert \vec{u} \Vert \to 0$ and $\Vert \vec{v} \Vert \to 0$: 
\begin{equation}
 \vec{w} =  \vec{v} - \vec{u} +  \bigO(\Vert \vec{u} \Vert^2 \, \Vert \vec{v} \Vert + \Vert \vec{u} \Vert \, \Vert \vec{v} \Vert^2)~.
\end{equation}
\end{lemma}

We prove in fact the following more precise lemma:

\begin{lemma} \label{lem:RiemannianEstimateTriangle2}
Let $(M,g)$ be a Riemannian manifold, fix $A \in M$. Let $\vec{U}$ and $\vec{V}$ be two tangent vectors at $A$, denote $B(t) = \exp_A(t \vec{U})$ and 
$C(s) = \exp_A(s \vec{V})$. Let $\vec{w}(t, s)$ be the parallel transport of $\vec{u_B} \coloneqq \exp_{B(t)}^{-1}(C(s))$ 
along the geodesic segment from $B(t)$ to $A$. Then:
\begin{equation} \label{eq:RiemannianEstimateTriangle2}
 \vec{w}(t,s) = s\vec{V} - t\vec{U} - \frac{{t}^2 s}{2} R(\vec{V}, \vec{U}) \vec{U} - \frac{t {s}^2}{3} R(\vec{U}, \vec{V}) \vec{V} + \bigO(t^4 + t^3 s + t^2 s^3)~.
\end{equation}
where $R$ is the Riemann curvature tensor of $(M,g)$.
\end{lemma}

\begin{proof}
First let us quickly discuss some general Riemannian geometry estimates in normal coordinates. We refer to \cite{LeoBrewinRiemannNormalCoordinates, MR2534336} for more details
on the computations that follow.

In normal coordinates at a point $A$, the Riemannian metric $g$ has the Taylor expansion
\begin{equation}
 g_{ij} = \delta_{ij} - \frac{1}{3} R_{ikjl}x^k x^l + \bigO(|x|^3)
\end{equation}
where $R_{ijkl}$ is the Riemann curvature tensor at $A$, or rather its purely covariant version (this well-known fact of Riemannian geometry goes back to Riemann's 
1854 habilitation \cite{MR3525305}). One can derive from the expression for the Christoffel symbols $\Gamma^k{}_{ij} = \frac{1}{2} g^{kl} (g_{li,j} + g_{lj,i} - g_{ij,l})$ that
\begin{equation}
\Gamma^k{}_{ij} = -\frac{1}{3} (R^k{}_{ijl} - R^k{}_{jil}) x^l + \bigO(|x|^3)~.
\end{equation}
One can then find the Taylor expansion of any geodesic $x(s)$, say with initial endpoint $x = x(0)$ and initial velocity $v$, by solving the geodesic equation
$\frac{\upd^2 x^k }{\upd s^2} + \Gamma^{k}_{i j }\frac{\upd x^i }{\upd s}\frac{\upd x^j}{\upd s} = 0$. One finds:
\begin{equation} \label{eq:GeodesicNormalCoordinates1}
x^k(s) = x^k + s v^k - \frac{s^2}{3} R^k{}_{ilj} \, v^i v^j x^l + \bigO(s |x|^3)~.
\end{equation} 
We can rewrite \eqref{eq:GeodesicNormalCoordinates1}
as a coordinate-free expression (but still in the chart given by $\exp_A$) as
\begin{equation} \label{eq:GeodesicNormalCoordinates2}
 x(s) = x + s v - \frac{s^2}{3} R(x, v) v + \bigO(s |x|^3)~.
\end{equation}

One can also compute the parallel transport of a vector $v$ along a radial geodesic $x(t) = tx$ by solving the parallel transport equation 
$\frac{\upd v^k}{\upd t} + \Gamma^k{}_{ij}(x(t)) v^i \frac{\upd x^j}{\upd t}$. One finds that
\begin{equation} \label{eq:ParallelTransportEstimate1}
 v^k(t) = v^k + \frac{1}{6} R^k{}_{jil} v^i x^j(t) x^l(t) + \bigO(t |x|^3)~,
\end{equation}
which we can rewrite as
\begin{equation} \label{eq:ParallelTransportEstimate2}
 v(t) = v + \frac{1}{6} R(v, x) x + \bigO(t |x|^3)~.
\end{equation}

Let us now come back to the setting of \autoref{lem:RiemannianEstimateTriangle2}. We shall work (implicitly) in the chart given by $\exp_A$. 
Note that we can write $B = t\vec{U}$
and $C = s\vec{V}$ in this chart. Let us denote by $x(\cdot)$ the unit speed geodesic from $B$ to $C$, so that $x(0) = B$,  $x(r) = C$ where $r = d(B,C)$, 
and $x'(0) = \vec{U_B}$ is the unit vector such that $\exp_B(r \vec{U_B}) = C$. By \eqref{eq:GeodesicNormalCoordinates2} we can write
\begin{equation}
 x(r) = x(0) + r \vec{U_B} - \frac{r^2}{3} R(x(0), \vec{U_B}) \vec{U_B} + \bigO(r |x(0)|^3)~.
\end{equation}
In other words, recalling that $x(0) = B = t \vec{U}$ and $x(r) = C = s\vec{V}$, we have
\begin{equation} \label{eq:VminusUEstimate}
 s\vec{V} - t\vec{U} = r  \vec{U_B} - \frac{t {r}^2}{3} R(\vec{U}, \vec{U_B}) \vec{U_B} + \bigO(r {t}^3)~.
\end{equation}
On the other hand, the parallel transport of $\vec{u_B} = r \vec{U_B}$ back to the origin along the radial geodesic 
$[A, B]$ is given by, according to \eqref{eq:ParallelTransportEstimate2}:
\begin{equation} \label{eq:WEstimate1}
 \vec{w} = r \vec{U_B} - \frac{{t}^2 r}{6} R(\vec{U_B}, \vec{U}) \vec{U} + \bigO({t}^3 r)~.
\end{equation}
Comparing \eqref{eq:VminusUEstimate} and \eqref{eq:WEstimate1}, we see that
\begin{equation} \label{eq:WEstimate2}
 \vec{w} = s\vec{V} - t\vec{U} - \frac{{t}^2 r}{6} R(\vec{U_B}, \vec{U}) \vec{U} - \frac{t {r}^2}{3} R(\vec{U}, \vec{U_B}) \vec{U_B} + \bigO({t}^3 r)~.
\end{equation}
Finally, let's work to have $s$'s and $\vec{V}$'s appear in this equation instead of $r$'s and $\vec{U_B}$'s.
First note that $\vec{r U_B} = s \vec{V} - t \vec{U} + \bigO(t {r}^2 +  r {t}^3)$ according to \eqref{eq:VminusUEstimate}, so using the fact that $R(\vec{U}, \vec{U}) = 0$,
one can write:
\begin{equation}
 \begin{split}
  r R(\vec{U_B}, \vec{U}) \vec{U} &= s R(\vec{V}, \vec{U}) \vec{U} + \bigO(t {r}^2 + r {t}^3)\\
  {r}^2 R(\vec{U}, \vec{U_B}) \vec{U_B} &= {s}^2 R(\vec{U}, \vec{V}) \vec{V} -t s R(\vec{U}, \vec{V}) \vec{U} + \bigO(tsr^2 + t^2sr + t^2 r^3 + t^3sr + t^4 r)~.
 \end{split}
\end{equation}
We thus get in lieu of \eqref{eq:WEstimate2}:
\begin{equation} \label{eq:WEstimate3}
 \vec{w} = s\vec{V} - t\vec{U} - \frac{{t}^2 s}{2} R(\vec{V}, \vec{U}) \vec{U} - \frac{t {s}^2}{3} R(\vec{U}, \vec{V}) \vec{V} + \bigO({t}^3 r + t^2 s r^2)~.
\end{equation}
The conclusion follows, noting that $r = \bigO(t + s)$ 
by the triangle inequality.
\end{proof}

\begin{remark}
 A direct consequence of \autoref{lem:RiemannianEstimateTriangle2} is 
 the expansion of the distance squared: 
 \begin{equation}
  d^2(\exp_A(t\vec{U}), \exp_A(s \vec{V})) = \Vert s\vec{V} - t\vec{U} \Vert^2 - \frac{1}{3}R(U, V, V, U)s^2 t^2 + \bigO((t^2 + s^2)^{\frac{5}{2}})~.
 \end{equation}
The same formula has been observed by other authors, see \eg{} \cite{RazielMO}.
\end{remark}

We are now ready to wrap up the proof of \autoref{thm:AverageMap}:

\begin{proof}[Proof of \autoref{thm:AverageMap}]
 
 Let $u_0 \in S_r \subset \upT_x M$, and consider the triangle in $N$ with vertices $A \coloneqq f(x)$, $B \coloneqq T_r f(x)$, and $C(u_0) \coloneqq f(\exp_x(u_0))$ in $N$. With the notations introduced
 above \autoref{lem:RiemannianEstimateTriangle1}, note that $\vec{u} = \vec{u_A} = \frac{r^2}{2m} \tau(f)_x$, 
 $\vec{v}(u_0) = \exp_{f(x)}^{-1} (f(\exp_x(u_0)))$, and $\vec{w}(u_0) = P (\vec{u_B}(u_0))$ where $\vec{u_B}(u_0) = \exp_{T_r f(x)}^{-1} (f(\exp_x(u_0)))$ and
 $P \colon \upT_A N \to \upT_B N$ is the parallel transport along the geodesic segment $[A,B]$. By \autoref{lem:RiemannianEstimateTriangle1}, we have
 \begin{equation} \label{eq:AM1}
  P (\vec{u_B}(u_0)) = \vec{w}(u_0) = \vec{v}(u_0) - \vec{u} + \bigO(\Vert \vec{u} \Vert^2 \, \Vert \vec{v}(u_0) \Vert + \Vert \vec{u} \Vert \, \Vert \vec{v}(u_0) \Vert^2)~.
 \end{equation}
Because we have $\Vert \vec{u} \Vert = \bigO(r^2)$ and $\Vert \vec{v}(u_0) \Vert = \bigO(r)$, 
\eqref{eq:AM1} may be rewritten:
 \begin{equation} \label{eq:AM2}
  P (\vec{u_B}(u_0)) = \vec{v}(u_0) - \vec{u} + \bigO(r^4)~.
 \end{equation}
 We now integrate \eqref{eq:AM2} over $u_0 \in S_r$:
  \begin{equation} \label{eq:AM3}
  \frac{1}{\Area(S_r)} \int_{S_r} P (\vec{u_B}(u_0))  \, \upd \sigma_r(u_0)  = \frac{1}{\Area(S_r)} \int_{S_r} \left( \vec{v}(u_0) \, \upd \sigma_r(u_0) ~-~ \vec{u} + \bigO(r^4)\right)~,
 \end{equation}
 which we can rewrite as
 \begin{equation}
    P\left(\frac{1}{\Area(S_r)} \int_{S_r} \vec{u_B}(u_0) \, \upd \sigma_r(u_0) \right)  = \left(\frac{1}{\Area(S_r)} \int_{S_r} \vec{v}(u_0) \, \upd \sigma_r(u_0)\right) ~-~ \vec{u} + \bigO(r^4)~.
 \end{equation}
 Now, \autoref{lem:AverageG} says precisely that $\left(\frac{1}{\Area(S_r)} \int_{S_r} \vec{v}(u_0) \, \upd \sigma_r(u)\right) = \vec{u} + \bigO(r^4)$. We thus get
 $\frac{1}{\Area(S_r)} \int_{S_r} \vec{u_B}(u_0) \, \upd \sigma_r(u_0) = P^{-1}(\bigO(r^4)) = \bigO(r^4)$. 
 That is, 
 \begin{equation}
  \int_{S_r} \exp_{T_r f(x)}^{-1} (f(\exp_x(u_0))) \upd \sigma_r(u_0) =  \bigO(r^4)~.
 \end{equation}
Recalling that $S_r f(x)$ is by definition the center of mass of the function $u_0 \in S_r \mapsto f(\exp_x(u_0))$, we can apply \autoref{lem:ProximityToCenterOfMass}
to conclude that $ d(S_rf(x), T_rf(x)) = \bigO(r^4)$.
\end{proof}

\subsection{Center of mass methods}
\label{subsec:CenterOfMassMethods}

We now discuss center of mass methods as an alternative to the heat flow in order to minimize the energy functional. The basic idea is to iterate the process of replacing a function $f \colon M \to N$ by its average on balls (or spheres) of radius $r>0$, hopefully converging to a map $f_r^*$ that is almost harmonic when $r$ is small. Observe that \autoref{thm:AverageMap} shows that this method is very close to a constant step gradient flow for the energy functional, \ie{} an Euler method with fixed stepsize. 

The next proposition 
is claimed in \cite[Lemma 4.1.1]{MR1451625}.
\begin{proposition} \label{prop:AveragingDecreasesApproximateEnergy}
 Let $(M,\mu)$ be a measure space, let $(N,d)$ be a Hadamard metric space and let $\eta \colon M \times M \to [0,+\infty)$ be a measurable symmetric function. 
 Define the Jost energy functional by
 \begin{equation} \label{eq:JostEnergyFunctional}
  E(f) =  \frac{1}{2} \int_M \int_M  \eta(x,y) \, {d(f(x), f(y))}^2 \, \upd \mu(y) \, \upd \mu (x)~.
 \end{equation}
For a measurable map $f \colon M \to N$, let $\varphi(f) \colon M \to N$ be the map such that for all $x\in M$, $\varphi(f)(x)$ is the center of mass of $f$
for the measure $\eta(x, \cdot) \mu$. 
Then for every $f$ with finite energy we have:
\begin{equation} \label{eq:AveragingDecreasesApproximateEnergy}
 E(\varphi(f)) \leqslant E(f)~.
\end{equation}
Moreover, the following are equivalent:
\begin{enumerate}[(i)]
 \item Equality holds in  \eqref{eq:AveragingDecreasesApproximateEnergy}.
 \item $\varphi(f) = f$ almost everywhere in $(M, \mu)$.
 \item $f$ is a minimizer of $E$.
\end{enumerate}
\end{proposition}


\subsubsection*{Center of mass method in the smooth setting}
Now assume that $M$ and $N$ are both Riemannian manifolds. For $r > 0$ we take the kernel $\eta_r(x,y)$ described
in \autoref{subsec:JostEnergy}, so that $E_r$ is the $r$-approximate energy. The map $\varphi(f)$ of \autoref{prop:AveragingDecreasesApproximateEnergy} is then the same as the map $\hat{B}_r f$ introduced in \autoref{def:AverageMap}. It is tempting to iterate the process of averaging $f$ in order to try and minimize $E_r$. The next theorem guarantees the success of this method under suitable conditions. Moreover, recall that the energy functional $E$ on $\upL^2(M,N)$ is the $\Gamma$-limit
of $E_r$ as $r\to 0$ (\cite[Lemma 8.3.4]{MR3726907}), so that minimizers of $E_r$ converge to minimizers of $E$ (possibly up to subsequence).

\begin{theorem} \label{thm:CenterOfMassMethodSmooth}
 Let $M$ and $N$ be Riemannian manifolds, assume $N$ is compact and with nonpositive sectional curvature. 
 In any homotopy class of continuous maps $M \to N$ where
 the $r$-approximate energy $E_r$ admits a unique minimizer $f^*$, the sequence $(f_k)_{k \in \bN}$ defined by $f_{k+1} = \hat{B}_r f_k$ converges locally uniformly to $f^*$ for any choice of a locally Lipschitz continuous $f_0$.
\end{theorem}
\begin{proof}
We reduce the proof to a combination of \autoref{lem:AveragePreservesLipschitz} and \autoref{lem:ConvergingSequence} below. Denote by $X$ the connected component of $f_0$ in $\cC(M,N)$, let $E \colon X \to \bR$ denote the restriction of $E_r$, and let $\varphi \colon X \to X$ be the map $f \mapsto \hat{B}_r f$ (it is easy to see that $\varphi$ preserves $X$). 
\autoref{lem:AveragePreservesLipschitz} guarantees immediately that the sequence $(f_k)_{k\in \bN}$ is equicontinuous. Since $N$ is compact, it follows from the Arzelà-Ascoli theorem that the sequence $(f_k)_{k\in \bN}$ is relatively compact in $X$ for the compact-open topology. By \autoref{prop:AveragingDecreasesApproximateEnergy} and the assumption that $f^*$ is unique, we have $E(\varphi(f)) \leqslant E(f)$ for all $f \in X$, with equality only if $f = f^*$. Conclude by application of \autoref{lem:ConvergingSequence}.
\end{proof}

\begin{lemma} \label{lem:AveragePreservesLipschitz}
Let $f \colon M \to N$ where $M$ and $N$ are Riemannian manifolds, with $N$ complete and nonpositively curved.
If $f$ is locally Lipschitz continuous, then so is $\hat{B}_r f$. Moreover, the Lipschitz constant of $\hat{B}_r f$ 
is bounded above by the Lipschitz constant of $f$ on any compact $K\subseteq M$.
\end{lemma}

\begin{proof}
For simplicity, we assume that $f$ is globally $L$-Lipschitz, and argue that $\hat{B}_r f$ is also $L$-Lipschitz; the proof can easily be extended to the general case by restricting to compact sets. First we assume that $M$ is Euclidean, in fact let us put $M = \bR^m$. Let $x, y \in M$, write $y = x + h$ so that $d(x,y) = \Vert h \Vert$.
By definition, $\hat{B}_r f(y)$ is the point of $N$ such that 
\begin{equation} \label{eq:AveragePreservesLipschitz1}
\frac{1}{\vol(B(y,r))} \int_{B(y,r)} \exp_{\hat{B}_r f(y)}^{-1}(f(v)) \, \upd v_g(v) = 0~.
\end{equation}
Note that the map $u \mapsto u + h$ defines an isometry from $B(x,r)$ to $B(y,r)$. Making the change of variables $v = u+h$, we derive from \eqref{eq:AveragePreservesLipschitz1}:
\begin{equation} \label{eq:AveragePreservesLipschitz2}
\int_{B(x,r)} \exp_{\hat{B}_r f(y)}^{-1}(f(u + h)) \, \upd v_g(u) = 0~.
\end{equation}
It follows that 
\begin{equation} \label{eq:AveragePreservesLipschitz3}
\int_{B(x,r)} \exp_{\hat{B}_r f(y)}^{-1}(f(u)) \, \upd v_g(u) = \int_{B(x,r)} \left[\exp_{\hat{B}_r f(y)}^{-1}(f(u))  - \exp_{\hat{B}_r f(y)}^{-1}f(u + h)) \right] \, \upd v_g(u)~.
\end{equation}
Assume without loss of generality that $N$ is simply connected (one can lift to the universal cover). Then $N$ is a Hadamard manifold and in particular a $\CAT(0)$ metric space, which implies that for any $p \in N$, the map $\exp_p^{-1} \colon N \to T_p N$ is distance nonincreasing (in fact, the converse is also true). We can therefore derive from \eqref{eq:AveragePreservesLipschitz3}:
\begin{equation} \label{eq:AveragePreservesLipschitz4}
 \left\Vert \int_{B(x,r)} \exp_{\hat{B}_r f(y)}^{-1}(f(u)) \, \upd v_g(u) \right\Vert \leqslant  \int_{B(x,r)} d(f(u), f(u + h)) \, \upd v_g(u)~.
\end{equation}
It follows from \eqref{eq:AveragePreservesLipschitz4} and the fact that $f$ is $L$-Lipschitz that
\begin{equation} \label{eq:AveragePreservesLipschitz5}
 \left\Vert \frac{1}{\vol(B(x,r))} \int_{B(x,r)} \exp_{\hat{B}_r f(y)}^{-1}(f(u)) \, \upd v_g(u) \right\Vert \leqslant  L \Vert h \Vert~.
\end{equation}
\autoref{lem:ProximityToCenterOfMass} now applies directly to \eqref{eq:AveragePreservesLipschitz5} to conclude that 
$d(\hat{B}_r f(x), \hat{B}_r f(y)) \leqslant L \Vert h \Vert$. Since $\Vert h \Vert = d(x,y)$, we have shown that $\hat{B}_r f$ is $L$-Lipschitz, as desired.

Now we argue that the argument extends to the case where $M$ is an arbitrary Riemannian manifold using a \emph{local to global} trick. First note that a function is globally $L$-Lipschitz if and only if it is locally $L$-Lipschitz. Here we mean by \emph{locally $L$-Lipschitz} the property that for any $x \in M$, there exists $\delta > 0$ such that $d(y, x) < \delta$ implies $d(f(y), f(x)) \leqslant L d(y,x)$. We leave it to the reader to show that in any path metric space, locally $L$-Lipschitz in this sense implies globally $L$-Lipschitz.

With this observation in mind, let us finish the proof. The key argument that worked above when $M$ is Euclidean is that there exists an isometry from $B(x,r)$ to $B(y,r)$ that displaces every point of at most $d(x,y)$. This is no longer true when $M$ is an arbitrary Riemannian manifold, however note that it is almost true when $x$ and $y$ are very close. Quantifying this properly, clearly one can show that for every $x\in M$ and for every $L' > L$, there exists $\delta > 0$ such that $d(y, x) < \delta$ implies $d(\hat{B}_r f(y), \hat{B}_r f(x)) \leqslant L' d(y,x)$. Thus we have shown that $\hat{B}_r f$ is locally $L'$-Lipschitz, and therefore globally $L'$-Lipschitz. Since this is true for all $L' > L$, $\hat{B}_r f$ is actually $L$-Lipschitz.
\end{proof}

\begin{lemma} \label{lem:ConvergingSequence}
 Let $X$ be a first-countable topological space and let $E \colon X \to \bR$ a continuous function that admits a unique minimizer $x^*$.
 Assume that $\varphi \colon X \to X$ is a continuous map such that:
 \begin{enumerate}[(i)]
  \item \label{lem:ConvergingSequenceHyp1} $E(\varphi(x)) \leqslant E(x)$ for all $x \in X$, with equality only if $x = x^*$.
  \item \label{lem:ConvergingSequenceHyp2} For all $x_0 \in X$, the set $\{\varphi^k(x_0), k\in \bN\}$ is relatively compact in $X$.
 \end{enumerate}
Then for any $x_0 \in X$, the sequence $(\varphi^k(x_0))_{k \in \bN}$ converges to $x^*$.
\end{lemma}

\begin{proof}
In any topological space, in order to show that a sequence $(x_k)_{k \in \bN}$ converges to a point $x^*$, it is enough to show that:
\begin{enumerate}[(a)]
 \item \label{lem:ConvergingSequenceProofItem1} The sequence $(x_k)$ has no cluster points except possibly $x^*$.
 \item \label{lem:ConvergingSequenceProofItem2} Any subsequence of $(x_k)$ admits a cluster point.
\end{enumerate}
Indeed, assume that $(x_k)$ does not converge to $x^*$, then there exists a subsequence of $(x_k)$ that avoids a neighborhood of $x^*$. This subsequence must have a cluster point by \ref{lem:ConvergingSequenceProofItem2}, which cannot be $x^*$. However this point is also a cluster point of the sequence $(x_k)$, contradicting \ref{lem:ConvergingSequenceProofItem1}.

Coming back to \autoref{lem:ConvergingSequence}, let $x_0 \in X$ and denote $x_k = \varphi^k(x_0)$. The sequence $(x_k)$ satisfies \ref{lem:ConvergingSequenceProofItem2}
because of the assumption \ref{lem:ConvergingSequenceHyp2}. So we need to show that $(x_k)$ satisfies \ref{lem:ConvergingSequenceProofItem1} and we are done.
Let $y$ be a cluster point of $(x_k)$, we need to show that $y = x^*$. Since $X$ is first-countable, there exists a subsequence $(x_{k_n})_{n \in \bN}$ converging to $y$. 
Observe that by assumption \ref{lem:ConvergingSequenceHyp1}, since $k_n \leqslant k_n +1 \leqslant k_{n+1}$,
\begin{equation} \label{eq:ConvergingSequenceProof1}
 E(x_{k_{n+1}}) \leqslant E(x_{k_n + 1}) \leqslant E(x_{k_n}) ~.
\end{equation}
By continuity of $E$, we have $\lim E(x_{k_n}) = \lim E(x_{k_{n+1}}) = E(y)$, so \eqref{eq:ConvergingSequenceProof1} implies that $\lim E(x_{k_n + 1}) = E(y)$. On the other hand, since $x_{k_n + 1} = \varphi(x_{k_n})$ and $\varphi$ is continuous, we have $\lim x_{k_n + 1} = \varphi(y)$, so $\lim E(x_{k_n + 1}) = E(\varphi(y))$. Thus $E(\varphi(y)) = E(y)$, and we conclude that $y = x^*$ by \ref{lem:ConvergingSequenceHyp1}.
\end{proof}

\subsubsection*{Discrete center of mass method}
We now prove that \autoref{thm:CenterOfMassMethodSmooth} also holds in the discrete setting developed in \autoref{sec:Discretization}, providing an alternative method to the discrete heat flow discussed in \autoref{subsec:ConvergenceOfDiscreteHeatFlow}. 
Let $\cG$ be a biweighted $\tilde{S}$-triangulated graph (see \autoref{subsec:Graphs}), let $N$ be a Hadamard manifold and let $\rho \colon \pi_1S\to \Isom(N)$ 
be a group homomorphism. We recall that the discrete energy functional $E_\cG \colon \Map_{\text{eq}}(\cG, N) \to \bR$ coincides with Jost's energy functional \eqref{eq:JostEnergyFunctional} for the appropriate choice of kernel $\eta$ (\autoref{prop:JostEnergyGraph}). 

Motivated by \autoref{prop:AveragingDecreasesApproximateEnergy}, we note that in this discrete setting the measure $\eta(x,\cdot) \mu$ is given by the weighted atomic measure
\begin{equation} \label{eq:AtomicMeasure}
\sum_{y\sim x} \frac {\omega_{xy}}{\mu(x)} \; \delta_y~,
\end{equation}
where $\delta_y$ is the Dirac measure at $y$. Now the averaging map $f\mapsto \hat{B}_rf$ takes the following form:
\begin{definition}\label{def:discreteCenterGraph}
The \emph{discrete center of mass method} on $\Map_{\text{eq}}(\cG, N) $ is given by
$f \mapsto \varphi(f)$, where $\varphi(f)(x)$ is the center of mass of $f$ for the atomic measure \eqref{eq:AtomicMeasure}.
\end{definition}

Note that \autoref{prop:AveragingDecreasesApproximateEnergy} 
applies in this setting 
(cf.~\autoref{prop:HarmonicCenterOfMassGraph2}).
Under certain assumptions on $N$ and $\rho$, we obtained strong convexity of $E_\cG$ in \autoref{thm:StrongConvexityEnergyGraphGeneral}, so that, in particular, $E_\cG$ has a unique minimum. The same assumptions have similarly useful consequences here:

\begin{theorem} \label{thm:CenterOfMassMethodDiscrete}
Let $N$ be a manifold of negative curvature bounded away from $0$ and let $\rho$ be a faithful representation whose image is contained in a discrete subgroup of $\Isom(N)$ acting freely, properly, and cocompactly on $N$.
Given any initial discrete equivariant map $f_0 \in \Map_{\text{eq}}(\cG, N) $,
the discrete center of mass method converges to the unique discrete harmonic map.
\end{theorem}

\begin{proof}
The proof is a similar but easier version of the proof of \autoref{thm:CenterOfMassMethodSmooth}. 
Let $f_0 \in \Map_{\text{eq}}(\cG, N) $, and define the sequence $(f_k)_{k \in \bN}$ by $f_{k+1} = \varphi(f_k)$.
The assumption on $\rho$ means that we can work
in a compact quotient of $N$, making the sequence $(f_k)$ pointwise relatively compact. Since the action of $\pi_1S$ on $\cG$ is cofinite, it is easy to see that the condition that the family $\{f_k\}$ is equicontinuous is vacuous. Hence the family $\{f_k\}$ is relatively compact in $\Map_{\text{eq}}(\cG, N)$. 
Since \autoref{prop:AveragingDecreasesApproximateEnergy} holds in this setting, all the requirements are met to conclude with \autoref{lem:ConvergingSequence}.
\end{proof}

\begin{remark}
In \cite{MR2346504}, Jost-Todjihounde describe an iterative process to obtain a discrete harmonic map from an edge-weighted triangulated graph $\cG$ to a target space that admits centers of mass (\eg{} Hadamard spaces).
\autoref{thm:CenterOfMassMethodDiscrete} can be viewed as a strengthened version of their result in two respects: For one, we avoid Jost-Todjihounde's passage to a subsequence of $(f_k)$. 
Moreover, Jost-Todjihounde start by subdividing $\cG$ and pursuing centers of mass in two phases, separately for vertices and midpoints of edges. 
Our discrete center of mass method requires no such subdivision.
\end{remark}

\subsection{\texorpdfstring{$\cosh$-center of mass}{cosh-center of mass}} \label{subsec:CoshCenterOfMass}

\autoref{thm:CenterOfMassMethodDiscrete} provides an effective method to compute discrete equivariant harmonic maps, alternative to the discrete heat flow (see \autoref{subsec:ConvergenceOfDiscreteHeatFlow}), as long as one is able to compute centers of mass.
Unfortunately, non-Euclidean centers of mass are not easily accessible. 
Even finding the barycenter of three points in the hyperbolic plane is a nontrivial task. 
While it is possible to use gradient descent method (see \cite{MR3057324}), it is computationally expensive and, 
in any case, 
not possible to do precisely in finite time. 
We present a clever variant to barycenters, well-suited to hyperbolic space $\bH^n$, that avoids this issue. 
We thank Nicolas Tholozan for bringing this idea to our attention.

\begin{definition} \label{def:CoshCenterOfMass}
Let $(\Omega, \cF, \mu)$ be a probability space, $(X,d)$ be a metric space, and $h \colon \Omega \to X$ a measurable map. A \emph{$\cosh$-center of mass} of $h$ is a minimizer of the function 
\begin{equation} \label{eq:DefinitionCoshCenterOfMass}
\begin{split}
P_h \colon X  & \to \bR\\
x & \mapsto \int_\Omega \left(\cosh d(x,h(y))-1 \right)   \, \upd \mu(y)~.
\end{split}
\end{equation}
\end{definition}


When $X$ is a Riemannian manifold, a $\cosh$-center of mass $G$ is characterized by
\begin{equation} \label{eq:ImplicitEquationForCoshCenterOfMass}
 \int_\Omega \sinhc d(G,h(y)) \exp_G^{-1}(h(y)) \, \upd \mu(y) = 0~,
\end{equation}
where $\sinhc(x) = \sinh(x)/x$ is the cardinal hyperbolic sine function. 

Equation \eqref{eq:ImplicitEquationForCoshCenterOfMass} implies that if $\supp(h_*\mu)$ is contained in a strongly convex region $U$ (\eg{} a ball of small enough radius), then any $\cosh$-center of mass is contained in $U$ as well: if $x$ is outside $U$, then each vector $\exp_x^{-1}(h(y))$, for $y\in \supp(h_*\mu)$, is contained in an open half-space in $\upT_x X$ containing $\exp_x^{-1}(U)$, and \eqref{eq:ImplicitEquationForCoshCenterOfMass} cannot be satisfied.

Let us now specialize to the case where $X = \bH^n$ is the hyperbolic $n$-space. In this setting the function $F(x) = \cosh(d(x_0,x)) -1$ is especially amenable to computations: 
\begin{equation}
 \begin{aligned}
  &\grad F(x) = \sinhc(d(x_0, x)) \exp_x^{-1}(x_0)\\
  &\Hess(F)_x(v,v) = F(x) \Vert v \Vert^2~.
 \end{aligned}
\end{equation}
In particular, $F$ is a strongly convex function on $\bH^n$ with modulus of strong convexity $\alpha = 1$. Existence and uniqueness of the center of mass of any 
function $h \in \upL^2(\Omega, \bH^n)$ quickly follows.

The main advantage of the $\cosh$-center of mass is that it admits an explicit description, much like the Euclidean barycenter. For this we work in the hyperboloid model for $\bH^n$, \ie{}
\begin{equation}
\cH = \{x\in\bR^{n,1}: \langle x,x \rangle=-1,x_{n+1}>0\}~,
\end{equation}
where Minkowski space $\bR^{n,1}$ is defined as $\bR^{n+1}$ equipped with the indefinite inner product 
\begin{equation}
\langle x, y \rangle = x_1 y_1 + \dots + x_n y_n - x_{n+1} y_{n+1}~.
\end{equation}
This inner product induces a Riemannian metric on $\cH$ of constant curvature $-1$.

\begin{proposition}
\label{prop:CoshBarycenter}
The $\cosh$-center of mass in $\bH^n \approx \cH$ is equal the orthogonal projection of the Euclidean barycenter in Minkowski space $\bR^{n,1}$ to the hyperboloid $\cH \subset \bR^{n,1}$.
\end{proposition}

\begin{proof}
We prove \autoref{prop:CoshBarycenter} for a finite collection of points for comfort; the generalization to any probability measure is immediate.
Consider points $p_1,\dots, p_n\in\cH$ with weights $w_1,\dots, w_n$ satisfying $\sum_i w_i=1$, let $p$ be their Euclidean barycenter in $\bR^{n,1}$, and 
let $q$ indicate the orthogonal (\ie{} radial) projection of $p$ to $\cH$. By \eqref{eq:ImplicitEquationForCoshCenterOfMass}, it suffices to check that 
\begin{equation} \label{eq:CoshBarycenterProof}
\sum_i w_i \sinhc d(q,p_i) \exp_q^{-1}(p_i) = 0~.
\end{equation}
Let $P$ be the tangent plane to $\cH$ at $q$, \ie{} the affine plane in $\bR^{n,1}$ which is orthogonal to the line $\bR q$ through $q$.
The orthogonal projection $\pi \colon \bR^{n,1} \to P$ is an affine map, so the identity $\sum_i w_i (p_i - p) = 0$ projects
to $\sum_i w_i (q_i - q) = 0$ on $P$, where $q_i = \pi(p_i)$. It is straightforward to compute $q = \frac{p}{\sqrt{-\langle p, p\rangle}}$
and $q_i = p_i + q + \frac{\langle p_i, p \rangle}{-\langle p, p \rangle}p$.

Geodesics in the hyperboloid are intersections of 2-dimensional subspaces of $\bR^{n,1}$ with $\cH$, so $q_i - q$ is a vector in $\upT_q\cH$ pointing towards $p_i$, 
and we can compute its length:
\begin{equation}
\begin{aligned}
 \Vert q_i - q \Vert^2 &= \langle p_i, q\rangle^2 - 1\\
 &=\sinh^2 d_{\cH}(q,p_i)
 \end{aligned}
\end{equation}
Thus we proved that 
\begin{equation}
q_i - q = \frac{\sinh d_{\cH}(q,p_i)}{d_{\cH}(q,p_i)}{\exp_q}^{-1}(q_i)
\end{equation}
and we get \eqref{eq:CoshBarycenterProof} as desired.
\end{proof}

Another useful feature of the $\cosh$-center of mass is that it is a good approximation of the center of mass for small distances. 
This will be important in \cite{Gaster-Loustau-Monsaingeon2}.

\begin{proposition}
\label{prop:CoshCMvsCM}
If a probability measure is supported in a ball of radius $r$, then its center of mass $p$ and its $\cosh$-center of mass $q$ and  are within $O(r^3)$ of each other.
\end{proposition} 

\begin{proof} Assume $\mu$ has finite support $\{p_1, \dots, p_n\}$ for comfort and denote $w_i = \mu(\{p_i\})$ the weights.
By \eqref{eq:ImplicitEquationForCoshCenterOfMass} we can write:
\begin{equation} \label{eq:CoshCMvsCM1}
 \sum_i w_i \exp_q^{-1}(p_i)  = \sum_i  w_i\left( 1- \sinhc d(p_i, q) \right) \exp_q^{-1}(p_i)~.
\end{equation}
Because $q$ must be contained in the same ball of radius $r$ as $\{p_i\}$, we find that $d(p_i, q) < 2 r$ for each $i$. 
Given that $\sinhc$ is a nondecreasing function we derive from \eqref{eq:CoshCMvsCM1}
\begin{align}
\left\Vert \sum_i w_i \exp_q^{-1}(p_i) \right\Vert
& \leqslant  \sum_i w_i \left(\sinhc d(p_i, q) -1\right) \left\Vert \exp_q^{-1}(p_i) \right\Vert \\
& \leqslant \sum_i w_i \left(\sinhc (2 r) - 1\right) \cdot (2 r) = 2 r \left(\sinhc (2 r) - 1\right)~.
\end{align}
\autoref{lem:ProximityToCenterOfMass} now implies that
\begin{equation}
 d(p,q) \leqslant 2 r \left(\sinhc (2 r) - 1\right)~.
\end{equation}
The conclusion follows, since $2 r \left(\sinhc (2 r) - 1\right) = \frac{4 r^3}{3} + O( r^5)$.
\end{proof}
Note that we did not use in the proof of \autoref{prop:CoshCMvsCM} that we are working in $\bH^n$: this proposition holds in any Riemannian manifold.



%% file: ComputerImplementation.tex
\section{Computer implementation: \Harmony{}}
\label{sec:Harmony}

\subsection{Computer description and availability}

\Harmony{} is a computer program developed by the first two authors (Jonah Gaster and Brice Loustau). 
It is a cross-platform software with a graphical user interface written in \Cpp{} code using the Qt framework. In its current state, it totals about $14,000$ lines of code. 

\Harmony{} is a free and open source software under the GNU General Public License. It is available on GitHub at \url{https://github.com/seub/Harmony}.

For more information, including install instructions and a quick guide to get started using the software, please visit the dedicated web site:
\url{https://www.brice.loustau.eu/harmony/}.

\subsection{Algorithms} 
\label{subsec:Algorithms}

We provide a brief overview of \Harmony{}'s algorithms allowing effective computation of discrete equivariant harmonic maps $\bH^2 \to \bH^2$ with respect to a pair of Fuchsian representations. A flowchart showing how the main algorithms fit into the program is pictured in \autoref{fig:Flowchart}.

We fix an identification of the closed oriented topological surface $S$ of genus $g$ as $P_0/\sim$, where $P_0$ is a topological $4g$-gon with oriented sides 
labelled $a_1$, $b_1$, $a_1^{-1}$, $b_1^{-1}$, $a_2$, \etc{}.

We parametrize hyperbolic structures on $S$ using the famous Fenchel-Nielsen coordinates. This requires choosing a pants decomposition of $S$. 
\Harmony{} is equipped to make such choices for arbitrary $g$ in a way that minimizes future error propagation. 

The input is a pair of Fenchel-Nielsen coordinates for hyperbolic structures $X$ and $Y$ on $S$, the domain and target hyperbolic surfaces respectively. These can be entered by the user in a `Fenchel-Nielsen selector' window: see \autoref{fig:FNselector}.

\subsubsection*{Step 1: Construct the fundamental group and pants decomposition}

After getting the genus $g$ as input, \Harmony{} constructs the fundamental group of the surface as an abstract structure. It then chooses a pants decomposition of the surface, yielding a decomposition of the fundamental group in terms of amalgamated products and HNN extensions of fundamental groups of pairs of pants. This is done recursively on the genus using a binary tree structure.

\subsubsection*{Step 2: Construct representations $\rho_X$ and $\rho_Y$}

This step performs the translation of Fenchel-Nielsen coordinates to Fuchsian representations. \Harmony{} starts by computing the representation of the fundamental group of each pair of pants using formulas that can be found in \eg{} \cite[Prop. 2.3]{MR1288062} or \cite{MR1703565,MR1833242}. It then computes the representation of the whole fundamental group using its decomposition discussed in Step 1.

\subsubsection*{Step 3: Construct fundamental domains $P_X$ and $P_Y$}

This step computes polygonal fundamental domains $P_X$ and $P_Y$ in $\bH^2$ for the Fuchsian groups in the images of $\rho_X$ and $\rho_Y$.
These polygons should be `as convex as possible' in order to ensure good behavior of the discrete heat flow.

Both $P_X$ and $P_Y$ come with $\pi_1S$-equivariant identifications to the topological $4n$-gon $P_0$ that record side pairings. Because the vertices of $P_0$ are all in the same $\pi_1S$-orbit, $P_X$ is determined by a single point in $\bH^2$. 
With this combinatorial setup, a best choice of polygon is obtained minimizing an adequate cost function $F:\bH^2\to \bR^+$. 
This is done with a straightforward Newton method. 


\subsubsection*{Step 4: Construct a triangulation of $P_X$}

This step computes a triangulation of the fundamental domain $P_X$. Finer and finer meshes can then be obtained by subdivision (see \autoref{def:Refinement}).
As explained in \cite{Gaster-Loustau-Monsaingeon2}, it is crucial to keep the smallest angle of the triangulation as large as possible. 

Unfortunately, $P_X$ already typically has very small angles. In order to avoid subdividing these angles further, we first introduce new \emph{Steiner vertices} evenly spaced along the sides of $P_X$. 
The resulting polygon is triangulated with a greedy recursive algorithm maximizing the smallest angle. See \autoref{fig:FNselector} for a sample output. 

\begin{remark}
The algorithm seems to always produce acute triangulations, a necessary condition for the definition of the edge weights (see \autoref{subsec:Meshes}), but we do not know whether this always holds. Acute triangulations of surfaces are part of a fascinating area of current research \cite{MR1064872,MR3016971
}.
\end{remark}

\subsubsection*{Step 5: Construct the $\rho_X$-invariant mesh $\cM$}

The mesh $\cM$ consists of a list of meshpoints $\cM^{(0)} \subset \bH^2$, each of which is equipped with a list of references to its neighboring meshpoints, 
and possibly side-pairing information. This data is initially recorded from the triangulated polygon $P_X$. Then, given a user chosen \emph{mesh depth} $k\geqslant 1$, $\cM$ is replaced with the $k$th iterated midpoint refinement of $\cM$ (see \autoref{def:Refinement}).

\begin{remark}
 Constructing the adequate data structure to efficiently store the mesh data is a difficult challenge: the corresponding \Cpp{} classes are the most sophisticated in the code of \Harmony{}.
\end{remark}

\subsubsection*{Step 6: Initialize an equivariant map}
\label{triangulate Y}
The triangulation of $P_X$ may be transported to one for $P_Y$ via the $\pi_1S$-equivariant maps 
\begin{equation}
P_X \approx P_0 \approx P_Y~. 
\end{equation}
Because the mesh is built from midpoint refinements, this identification
provides an initial discrete equivariant map $f_0 \colon \bH^2 \to \bH^2$. A sample initial map is showed in \autoref{fig:InitialFunctionPic}.

\subsubsection*{Step 7: Run the discrete flow}
\Harmony{} is ready to run: either the discrete heat flow with fixed or optimal stepsize (see \autoref{subsec:ConvergenceOfDiscreteHeatFlow}) or the $\cosh$-center of mass method (\autoref{subsec:CenterOfMassMethods}, \autoref{subsec:CoshCenterOfMass}). \Harmony{} uses several threads so that the flow is displayed ``live'' as it is being computed. See \autoref{fig:MidFlow} for a screenshot of \Harmony{} mid-flow. 

The flow is iterated until the error reaches a preset tolerance, or when it is stopped by the user. Both the discrete heat flow and the center of mass method typically converge very well. Refer to \autoref{subsec:ExperimentalComparison} for a comparison of the methods. \autoref{fig:Output} shows a sample output equivariant harmonic map.

\subsection{High energy harmonic maps between hyperbolic surfaces} \label{subsec:HighEnergy}
 
In this final subsection we briefly explain how \Harmony{} provides visual confirmation of a qualitative phenomenon that is well-known in Teichmüller theory. We hope that future development of the program will allow many more experimental investigation of theoretical aspects. We refer to \cite{MR2349668} for more details on what follows.

As mentioned in \autoref{subsec:harmonicMapsSurfaces}, taking the Hopf differential of harmonic maps allows one to parametrize Teichmüller space by holomorphic quadratic differentials. More precisely, given a closed oriented $S$ of genus $\geqslant 2$, let $\FS$ denote the Fricke-Klein space of $S$, \ie{} the deformation space of hyperbolic structures on $S$ up to isotopically trivial diffeomorphisms. 
Fix a complex structure $X$ on $S$, and denote $Q(X)$ the vector space of holomorphic quadratic differentials on $X$. For any $\sigma \in \FS$, there is a unique harmonic map $f_\sigma \colon X \to (S, \sigma)$. Taking its Hopf differential yields a map
\begin{equation} \label{eq:WolfParam}
\begin{split} 
 H \colon \FS &\to Q(X)\\
 \sigma & \mapsto \varphi_{f_\sigma}~.
 \end{split}
\end{equation}
Wolf \cite{MR982185} proved that the map $H$ is a global diffeomorphism from $\FS$ to $Q(X)$. 
Furthermore, $H$ continuously extends to the ``boundaries at infinity'':
\begin{equation} \label{eq:WolfParamCompactified}
 \partial H \colon \partial\FS \to \partial Q(X)~.
\end{equation}
Here the boundary $\partial\FS$ compactifying Fricke-Klein space is the \emph{Thurston boundary} $\partial \FS \coloneqq \mathcal{PMF}(S)$, the projective space 
of measured foliations on $S$. The boundary of the vector space $Q(X)$ is simply its projectivization: $\partial Q(X) \coloneqq \bP(Q(X))$. It can be identified
to the projective space of measured laminations $\mathcal{PML}(S)$ by assigning to any holomorphic quadratic differential its horizontal foliation.
Thus the boundary map $\partial H$ of \eqref{eq:WolfParamCompactified} may be described as a map
\begin{equation}
 \partial H \colon \mathcal{PMF}(S) \to \mathcal{PML}(S)~.
\end{equation}
This map has a nice geometric interpretation as the well-known ``pull tight'' map which assigns to each non-singular leaf
of a measured foliation the unique geodesic in its homotopy class.

Concretely, this means that a high energy harmonic map $f \colon X \to (S, \sigma)$ typically has a specific behavior dictated by its Hopf differential $\varphi$, namely:
the zeros of $\varphi$ are blown up to ideal hyperbolic triangles, while the rest of the surface is compressed onto the measured lamination dual to $\varphi$. This behavior is well verified by \Harmony{}. As an example, \autoref{fig:HighEnergy} shows the image of a high energy harmonic map: observe how the most contracted (darker) regions approach a geodesic lamination, while the most dilated (lighter) regions approach a union of ideal triangles.


\newpage
\newgeometry{top=0.09\paperheight, bottom=0.11\paperheight, left=0.06\paperheight, right=0.06\paperheight}
\subsection{Illustrations of \Harmony{}}

\begin{figure}[!ht]
	\centering
	\includegraphics[width=\textwidth]{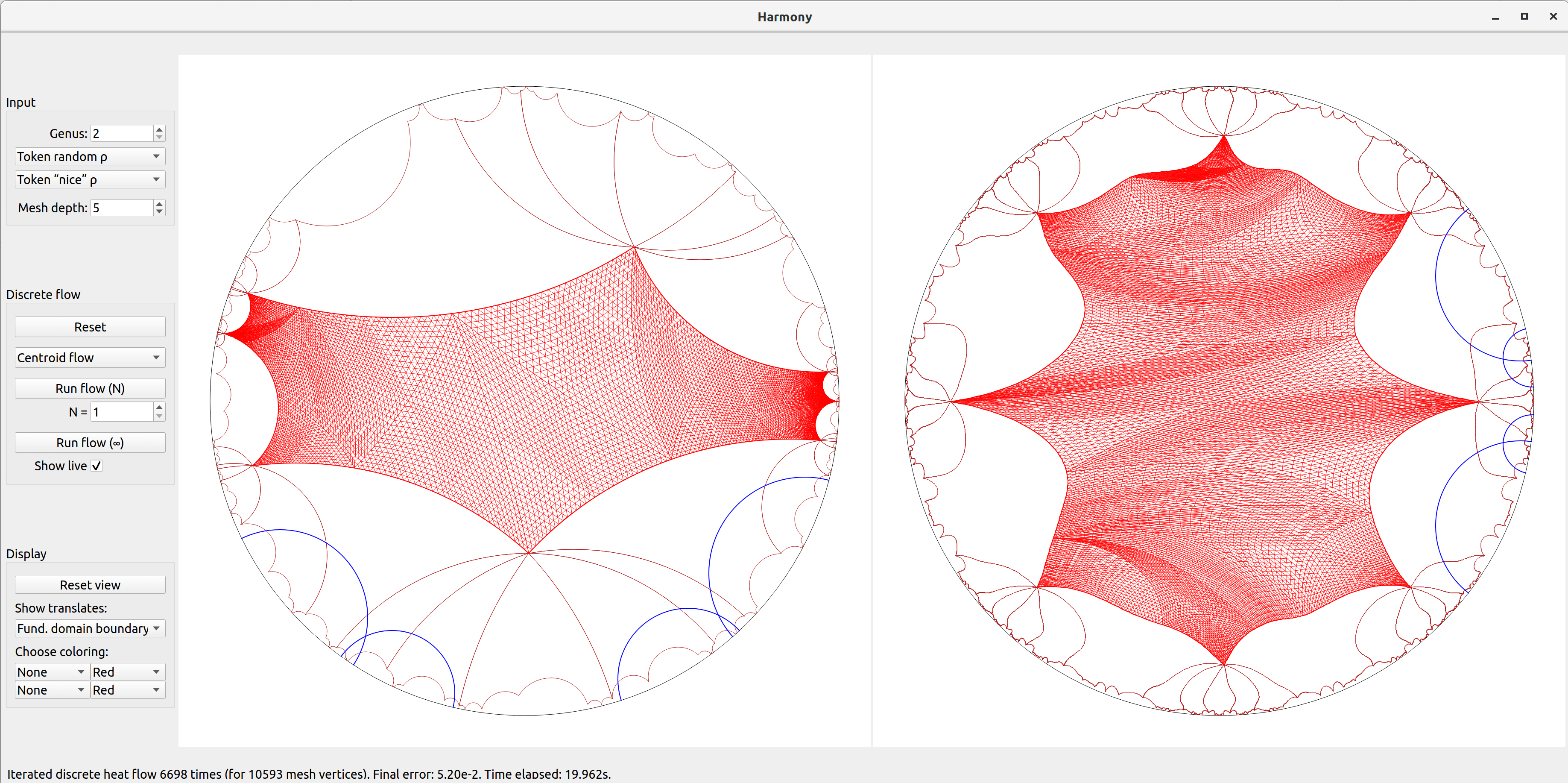}
	\caption{\Harmony{}'s main user interface.}
	\label{fig:HarmonyUserInterface}
\end{figure}

\begin{figure}[!ht]
	\centering
	\includegraphics[width=0.79\textwidth]{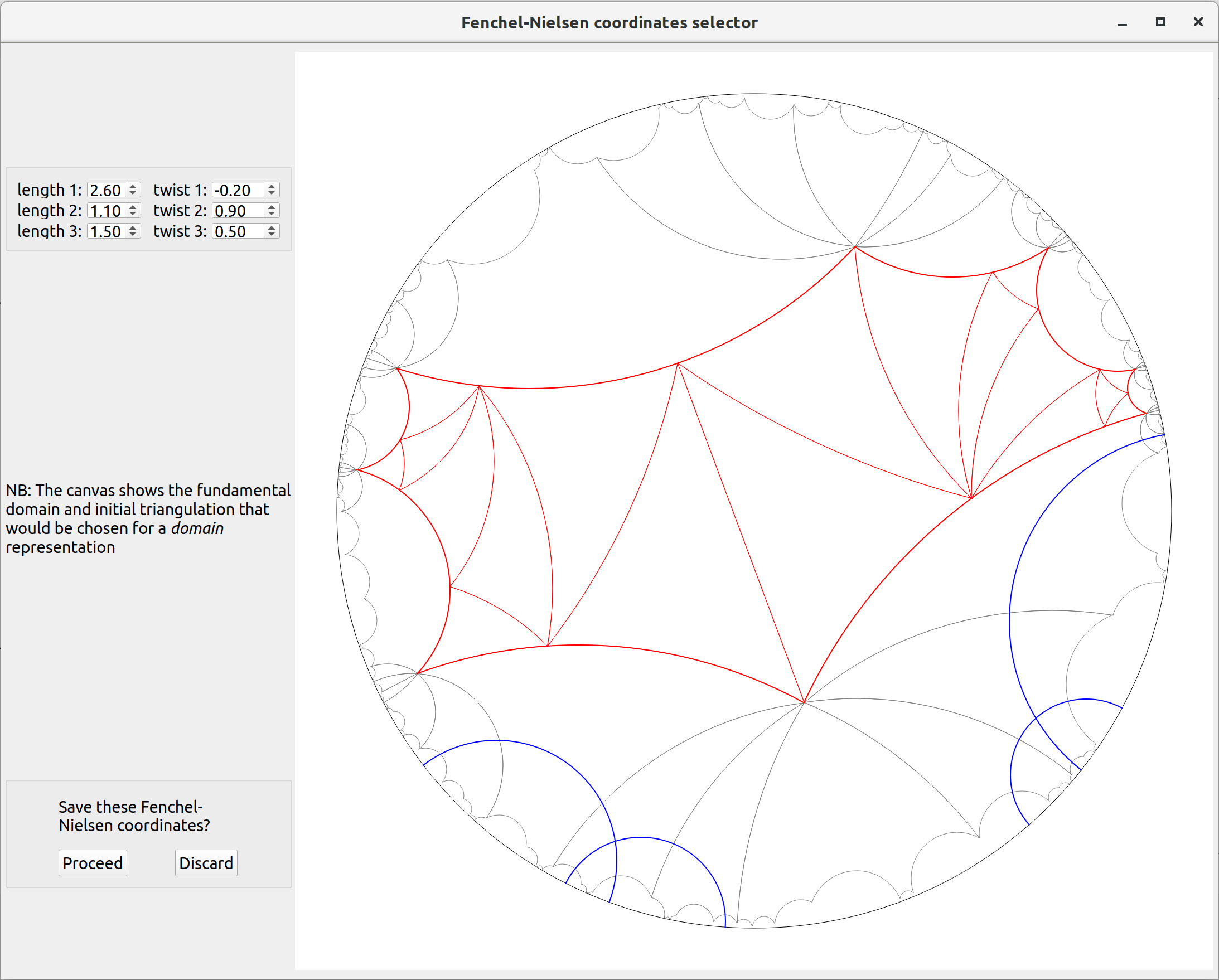}
	\caption{Fenchel-Nielsen coordinates selector.}
	\label{fig:FNselector}
\end{figure}

\begin{figure}[!ht]
	\centering
	\includegraphics[width=1.0\textwidth]{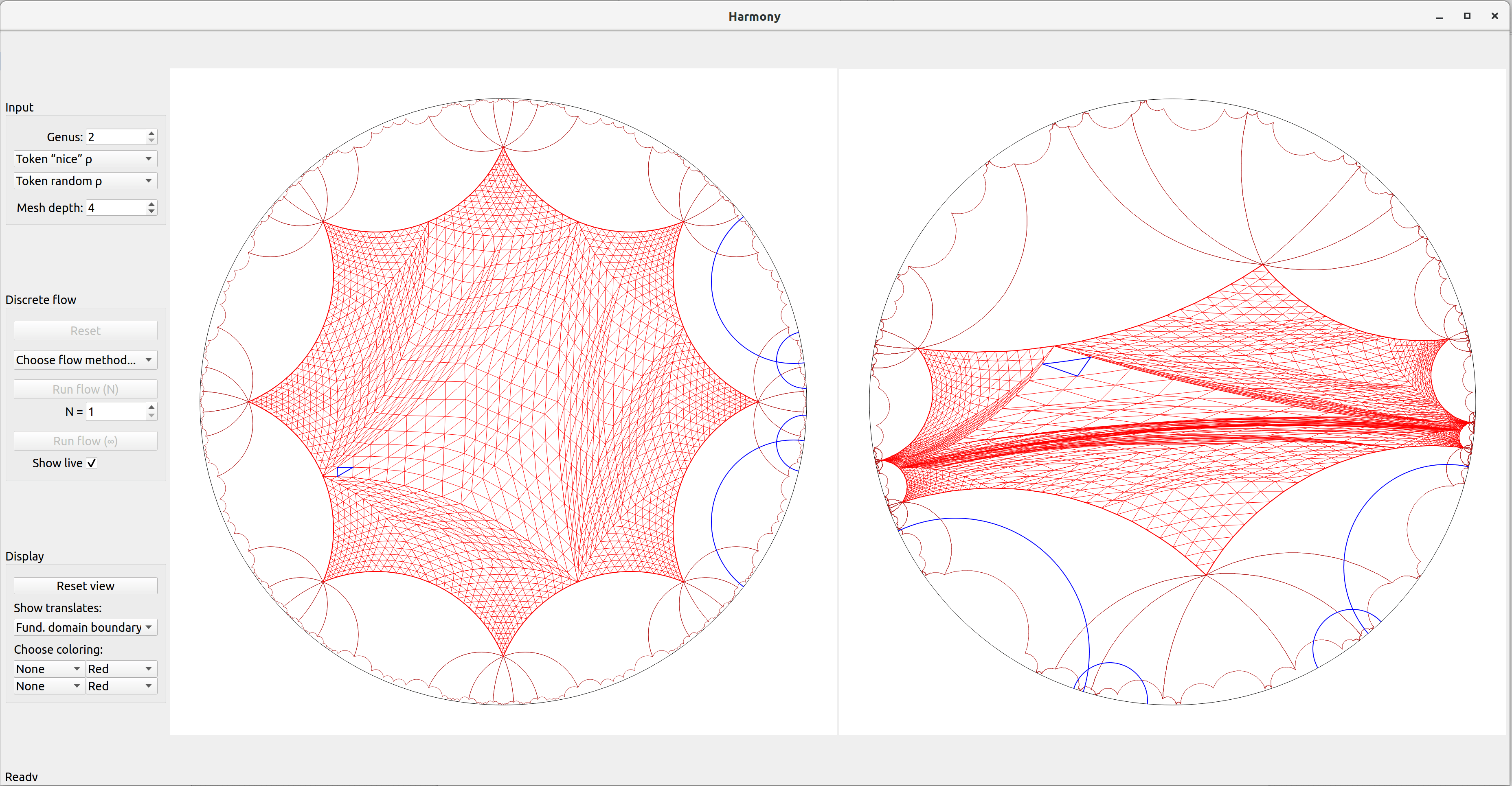}
	\caption{An initial discrete equivariant map $f_0$. The highlighted blue triangles are matched.}
	\label{fig:InitialFunctionPic}
\end{figure}

\begin{figure}[!ht]
	\centering
	\includegraphics[width=\textwidth]{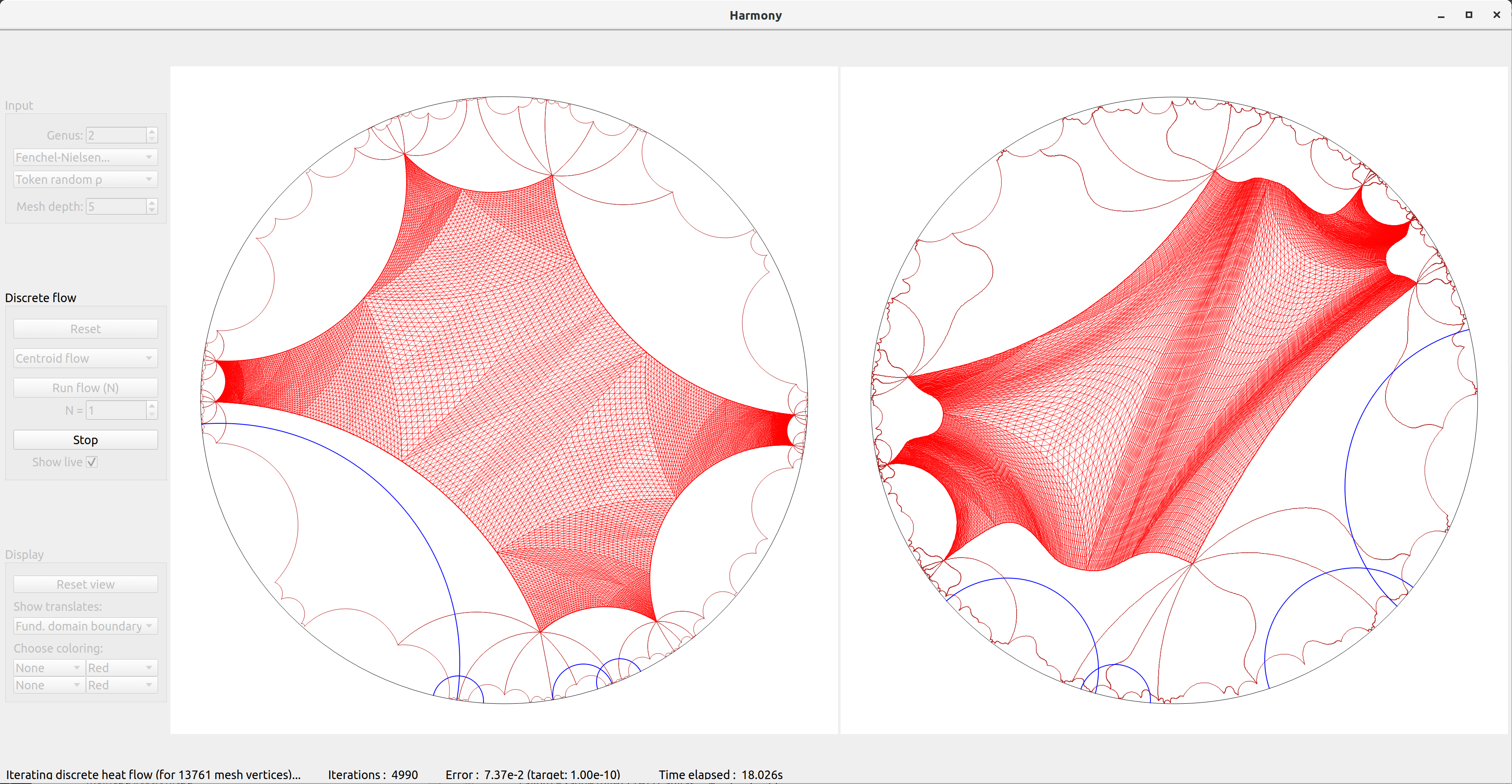}
	\caption{\Harmony{} mid-flow.}
	\label{fig:MidFlow}
\end{figure}

\begin{figure}[!htp]
	\centering
	\includegraphics[width=\textwidth]{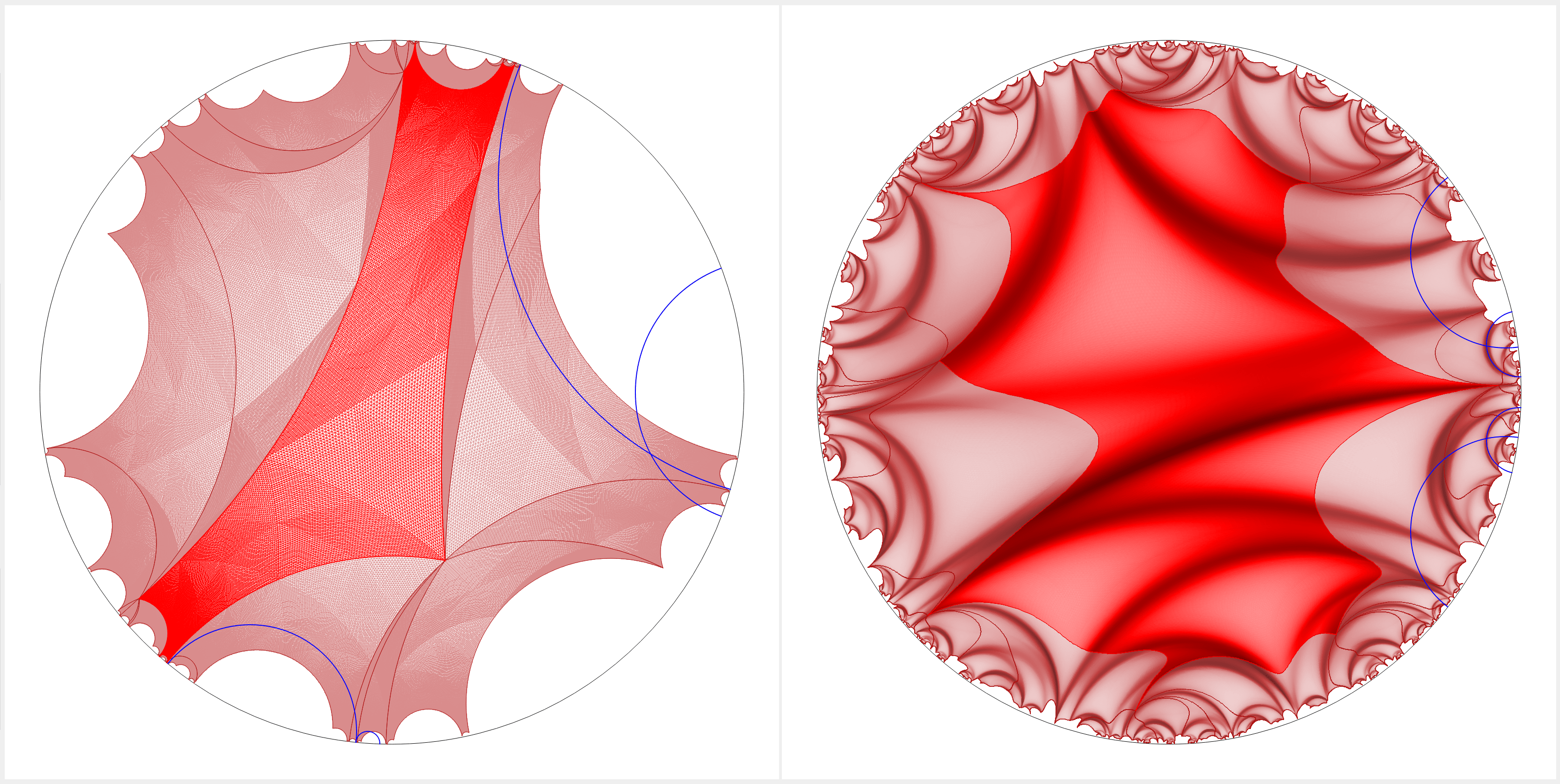}
	\caption{Sample output harmonic map. The brighter central regions are fundamental domains.}
	\label{fig:Output}
\end{figure}

\begin{figure}[!htp]
	\centering
	\includegraphics[width=0.73\textwidth]{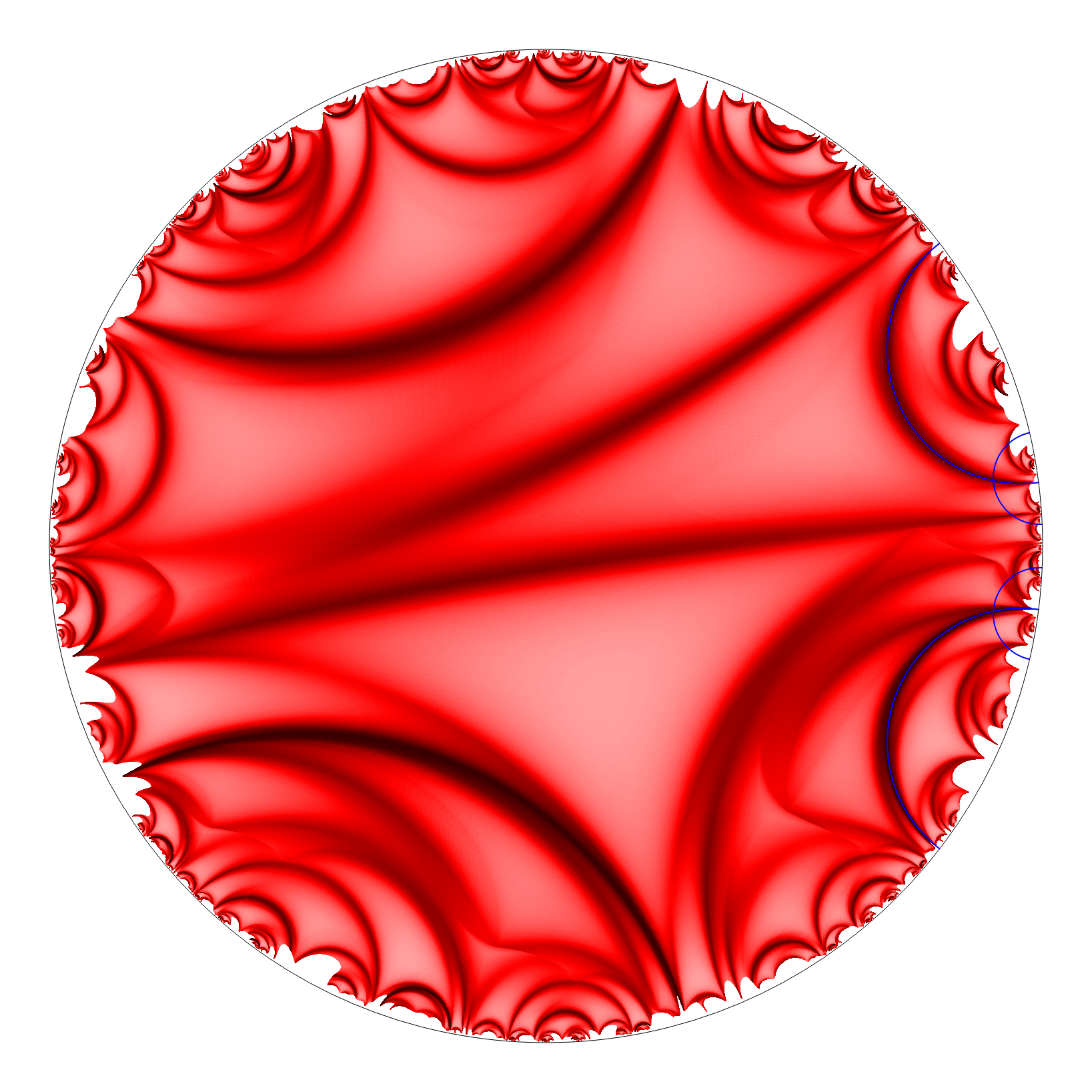}
	\caption{Sample output high energy equivariant harmonic map.}
	\label{fig:HighEnergy}
\end{figure}

 \begin{figure}[!ht]
 	\centering
 	\includegraphics[width=0.75\textwidth]{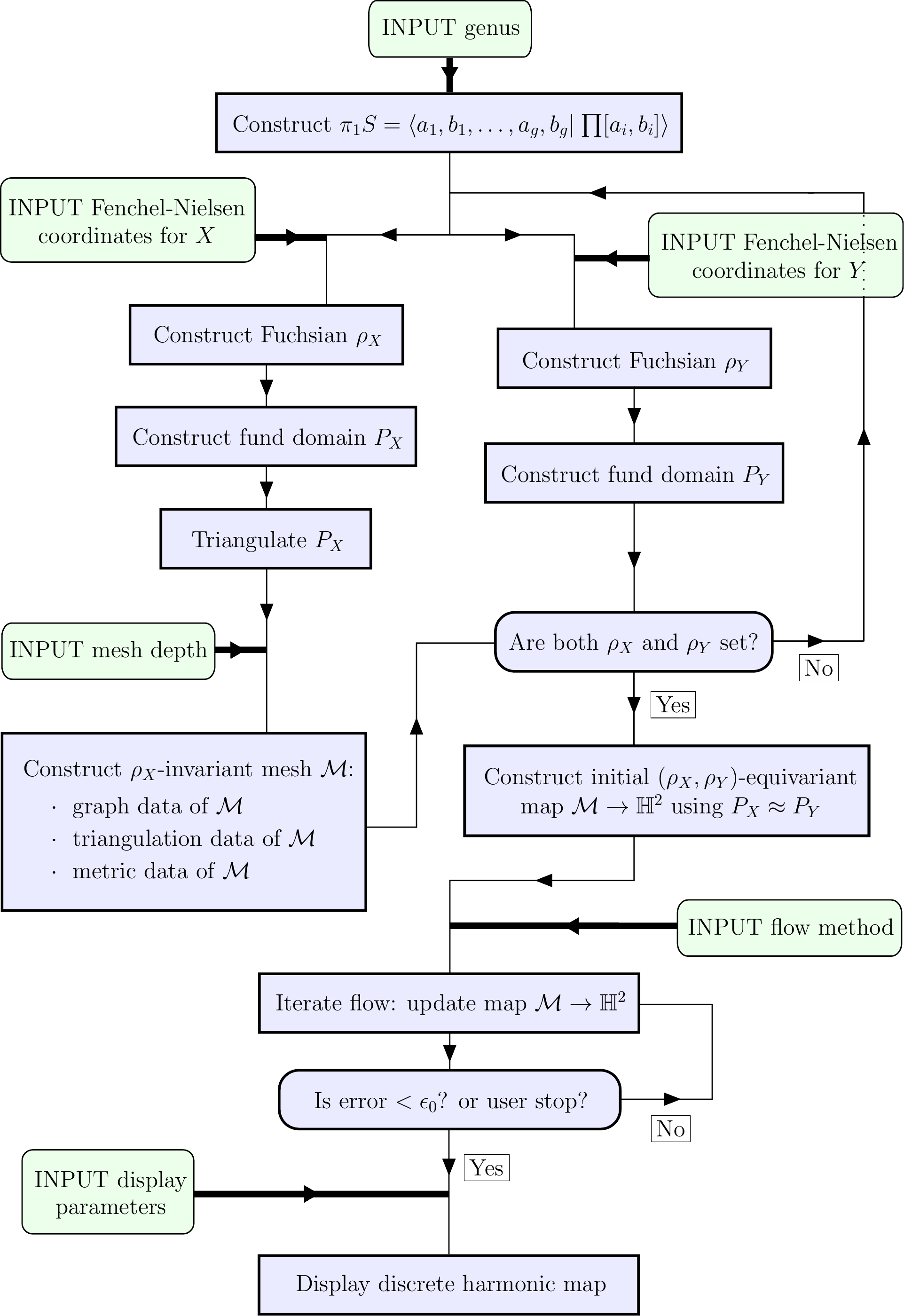}
 	\caption{Flowchart representing \Harmony{}'s main algorithms.}
 	\label{fig:Flowchart}
 \end{figure}

\restoregeometry